\documentclass[trcs]{informs3}

\OneAndAHalfSpacedXI 

\usepackage{endnotes}

\let\footnote=\endnote
\let\enotesize=\normalsize
\def\notesname{Endnotes}%
\def\makeenmark{\hbox to1.275em{\theenmark.\enskip\hss}}
\def\enoteformat{\rightskip0pt\leftskip0pt\parindent=1.275em
  \leavevmode\llap{\makeenmark}}

\def\R{\ensuremath{\mathbb{R}}}
\def\Z{\ensuremath{\mathbb{Z}}}

\newcommand{\set}[1]{\ensuremath{\mathcal{#1}}}
\newcommand*{\boldone}{\text{\usefont{U}{bbold}{m}{n}1}}
\newcommand*{\LS}{\texttt{DD-LS}}
\newcommand*{\LSt}{\texttt{DD-LS}~}

\usepackage{natbib}
 \bibpunct[, ]{(}{)}{,}{a}{}{,}%
 \def\bibfont{\small}%

\TheoremsNumberedBySection  

\EquationsNumberedThrough    
\usepackage{algorithm} 
\usepackage{color}
\usepackage{xcolor}
\usepackage{algpseudocode}
\usepackage{parallel}
\usepackage{hyperref}
\usepackage{cleveref}
\usepackage{amsmath}
\usepackage{theorem}
\usepackage{array}
\usepackage{makecell}
\usepackage{caption}
\usepackage{subcaption}
\usepackage{booktabs}
\usepackage{multirow}
\usepackage[colorinlistoftodos,prependcaption,textsize=small]{todonotes}
\usepackage{enumerate}
\newcommand{\aref}[1]{\hyperref[#1]{Appendix~\ref*{#1}}}
\definecolor{OriginalDarkBlue}{rgb}{0.1, 0.0, 0.9}
\colorlet{blue}{OriginalDarkBlue!80!white}

\usepackage{fancyhdr}
\usepackage[margin=1in]{geometry}
\begin{document}


\RUNAUTHOR{J. Deng and G. Pantuso}

\RUNTITLE{Mobility pricing under decision-dependent demand uncertainty}

\TITLE{Pricing mobility services under decision-dependent demand uncertainty: a carsharing case}

\ARTICLEAUTHORS{%
\AUTHOR{Jiali Deng}
\AFF{Department of Mathematical Sciences, University of Copenhagen, Copenhagen, 2100, \EMAIL{jd@math.ku.dk}} 
\AUTHOR{Giovanni Pantuso}
\AFF{Department of Mathematical Sciences, University of Copenhagen, Copenhagen, 2100, \EMAIL{gp@math.ku.dk}}
} 

\ABSTRACT{
The problem of pricing mobility services has attracted significant attention.
In most studies, uncertain demand is modeled as an exogenous random variable with known distribution. 
This assumption overlooks the likely effect of prices on user adoption decisions. To address this dependency, we formulate the pricing problem as a stochastic program with decision-dependent demand uncertainty. Specifically, we make the non-standard assumption that the probability distribution of demand depends on pricing decisions. 
We show that the problem can be written as a mixed-integer linear program whose size is exponential in the input parameters. To find exact numerical solutions we specialize the L-shaped method for stochastic programs with decision-dependent uncertainty. In particular, we devise efficient separation routines by proving closed-form primal and dual solutions to the involved subproblems. In addition, we develop problem-specific valid inequalities and cut-sharing mechanisms which significantly improve convergence. 
We show that the method outperforms by far a commercial solver used to solve the monolithic formulation. 
Furthermore, in a case study based on a real-world carsharing system, we show that incorporating decision-dependent uncertainty improves expected profits by $8.39\%$ compared to a benchmark that considers deterministic price-elastic demand, and by $8.53\%$ compared to a benchmark that considers exogenous random demand, on average. In addition, we evaluate the performance of preventive pricing and relocation decisions under two vehicle allocation policies. The results suggest that a controlled allocation of vehicles to customers can improve service rates while only marginally affecting profits. \color{black}
}

\maketitle

\section{Introduction}
Shared mobility has become an important component of modern urban transportation systems. However, offering flexible on-demand mobility in a profitable manner poses significant management challenges. A central element of complexity is the necessity of making operational decisions without complete knowledge of near-future demand. 
Consequently, vehicle relocation has been widely studied as an instrument to prevent or ameliorate spatial mismatches between demand and supply, particularly in the carsharing domain, see, e.g., \citet{wang2010dynamic, nair2011fleet, weikl2013relocation, nourinejad2015vehicle, santos2015mip}. Relocations are, however, inherently expensive and prone to inefficiency \citep{lu2018optimizing}. For this reason, pricing-based demand management is gaining popularity. Real-world examples include, among others, the German shared mobility company \textit{Share Now} \citep{Christan2024sharenow}, which uses real-time price adjustments based on updated demand information, and the software company \textit{Zoba} \citep{zoba}, which updates service prices in a regular interval (e.g., every 30 minutes) given real-time supply and demand information. 

\textcolor{black}{In the context of one-way carsharing services, the question of how to effectively differentiate service prices by location and/or time has gained popularity.} The optimization literature has offered models and methods where prices act either as a substitute \citep{jorge2015trip, wang2019pricing, xie2019optimal, deng2025zonification} or a supplement \citep{huang2018solving, xu2018electric, lu2021performance, pantuso2022exact} to vehicle rebalancing efforts. Accordingly, \textit{carsharing system operators} (CSOs) use the pricing levers to mitigate the geographical or/and temporal demand-supply mismatch and, consequently, improve service rate and revenue potential. Hence, the relationship between prices and demand becomes an important component of the models.

In the available shared mobility research, the effect of prices on rental demand is typically modeled via their effects on demand location or volume as well as on individual mode choice. \citet{pfrommer2014dynamic} assumed that price incentives would induce users to return bikes to stations other than the one closest to their destination. \citet{singla2015incentivizing} extended the effect of prices on user's station choice to both pick-up and drop-off bicycle stations. In a one-way carsharing pricing problem studied by \citet{wang2019pricing}, similar price-induced demand location shifts were also considered on an individual user basis. Apart from this effect, \citet{wang2019pricing} modeled the total demand generated from each station as a reverse logistic function of pick-up prices. \citet{jorge2015trip} assumed that carsharing demand decreases linearly with price at a constant rate. 
Likewise, \citet{xie2019optimal} approximated the demand as a linear function of price with constant elasticity. \citet{xu2018electric} adopted an exponential function to capture the nonlinear relationship between demand and price. The authors truncate the demand function at a given price threshold to prevent the price from rising indefinitely. \citet{huang2018solving} modeled the behaviour of individual carsharing customers through a Logit model where travel time and price are accounted for as factors in the utility function. The portion of carsharing demand between two traffic zones was then determined by mode choice of individual customers. A similar approach was used by \citet{deng2025zonification}. The authors focus on the partition of the service region as well as zone-to-zone prices. In the above studies, carsharing demand is a deterministic function of price. Hence, while the effect of prices on demand is accounted for, the stochastic nature of demand is not accurately modeled.

When uncertain demand is taken into account, it is usually modeled as an exogenous random variable. In vehicle relocation problems, both \citet{nair2011fleet} and \citet{lu2018optimizing} assumed the demand for carsharing to be known only in distribution. \citet{nair2011fleet} used chance constraints to guarantee satisfaction of a given proportion of demand. \citet{lu2018optimizing} formulated a two-stage mixed-integer linear programming (MILP) model with finite demand scenarios generated from a joint distribution of one-way and round-trip demand rentals. \citet{lu2018optimizing} also conducted experiments on the case that one-way rental demand was endogenously determined by prices, where empirical demand distributions were generated from historical data with breaking rental price points. Nevertheless, the stochastic optimization model proposed by \citet{lu2018optimizing} does not include pricing decisions.
\citet{huang2021innovative} formulated a long-term fleet sizing and pricing problem as a two-stage stochastic program, where short-term vehicle relocations represent second-stage decisions. Demand was jointly determined by an exogenous random base-values (with given distribution) and a constant demand elasticity. \citet{pantuso2022exact} proposed a two-stage stochastic program where first-stage decisions are pricing and relocation decisions while second-stage decisions model rental decisions. Uncertainty of demand is accounted for by means of the randomness in the model of customer behavior (e.g., response to prices). A potential drawback of their approach is the necessity to individually include customers in the optimization model, i.e., demand is not aggregated.

In the account given above, it emerges that pricing decisions (i) are often made before accurately knowing customer demand and (ii) influence adoption decision and hence rental demand.
Stochastic optimization naturally became a reference decision-support methodology, see, e.g., \citet{lu2018optimizing, huang2021innovative, pantuso2022exact}. 
However, the existing literature typically models carsharing demand as an \textit{exogenous} random variable, that is, independent of the CSO's decisions. The dependency of random demand on prices remains, to a large extent, unexplored. Part of the explanation behind this lack of studies is the sparsity of general methodology for stochastic optimization with \textit{endogenous} (or \textit{decision-dependent}) uncertainty. While applications can be found, such as \citet{zhan2016generation} for power generation expansion, \citet{lejeune2018aeromedical} for medical evacuation, \citet{kopa2021decision} for asset–liability management, general-purpose methodology is somewhat sparse. To the best of our knowledge, we propose the first application in the domain of shared mobility.

In this paper, we address the problem of pricing mobility services accounting for \textit{endogenous} demand uncertainty.
The problem is studied in the context of a one-way station-based carsharing system where rental prices may depend on the origin and/or destination of the rental. We assume that prices have an impact on rental demand. In particular, we assume that the probability distribution of rental demand is fully specified by the pricing decisions made.
Hence, the CSO seeks to optimally set carsharing rental prices between different locations, and operate the necessary relocations, in order to maximize the conditional (on the price) expectation of rental profits in a near-future interval of time. Carsharing and, more generally, shared mobility services have developed into a variety of practical operating models. We do not aim to provide the most accurate model for a specific business case or for specific assumptions on customers behavior. In contrast, we provide a recipe that can be easily adapted to various specifications of the service. 

The contributions and main findings of this paper can be summarized as follows.
\begin{itemize}
    \item We propose a \textit{two-stage mixed-integer stochastic program with decision-dependent uncertainty} to model pricing and relocation decisions in a one-way carsharing service, where the random demand follows a price-dependent probability distribution. 
    \item \color{black}To address the exponential size of the model, we propose an exact solution method which specializes a variant of the L-shaped method for problems with decision-dependent uncertainty developed by \cite{pantuso2025shaped}. In particular, (i) we derive distribution-specific optimality cuts, (ii) we prove closed-form primal and dual solutions to the second-stage problems; this guarantees efficient separation routines, (iii) we propose effective problem-specific valid inequalities and, finally, (iv) we devise cut-sharing mechanisms among different versions of the second-stage problem.\color{black} 
    \item In extensive computational experiments, we demonstrate the scalability of the proposed method. In particular, the method could solve more than twice as many instances to optimality compared to a commercial solver.
    \item Finally, in a case study based on a carsharing system in Copenhagen, Denmark, we show and discuss the effects of accounting for decision-dependent uncertainty. \color{black}Taking into account decision-dependent demand uncertainty increased expected profits by between $7.04\%$ and $10.82\%$ compared to a benchmark with deterministic price-elastic demand, and by between $3.01\%$ and $13.54\%$ compared to a benchmark with decision-independent stochastic demand. \color{black}
\end{itemize}

In the remainder of this paper, we first give a detailed problem description and a general formulation of the carsharing pricing and relocation problem with endogenous demand uncertainty in \Cref{sec:problem_describe}.
In \Cref{sec: L-shaped method}, we present the solution algorithm. In \Cref{sec: computational study} we provide evidence on the effectiveness of the solution method based on the result of extensive numerical experiments. In \Cref{sec: case_study}, through a real-world case study based on a carsharing system operated in Copenhagen, Denmark, we discuss the practical applicability of our proposed model and the effects of accounting for decision-dependent uncertainty. Finally, we draw conclusions and make final remarks in \Cref{sec: conclusion}.

\section{Problem} \label{sec:problem_describe}

Given a target interval of time (e.g., $17:00-18:00$), the CSO needs to determine (i) the rental prices to apply during that interval of time and (ii) the relocations to perform before the beginning of (hence in preparation for) that interval of time.

Let $\set{G}:=(\set{I}, \set{A})$ denote the carsharing network with $\set{I}$ being the set of carsharing stations and $\set{A}:=\{(i, j)| i, j \in \set{I}, i \not= j\}$. 
We assume the set $\set{I}$ is partitioned into zones (e.g., geographical districts of the city). This is without loss of generality, since it is always possible to define one zone for each station. Let $\set{Z}$ denote the set of zones and $\set{I}^z \subset \set{I}$ the set of stations in zone $z$. For each pair of zones, the CSO may select a rental price from a discrete set of prices $\set{L}$. Hence, pricing decisions are modeled by binary variables $x: = (x_{z_1,z_2,l})_{z_1, z_2\in\set{Z}, l\in\set{L}}$, with $x_{z_1, z_2, l}$ taking value $1$ if price $l$ is applied for rentals that originate in $z_1$ and terminate in zone $z_2$, $0$ otherwise. Zone-based rental prices may represent e.g., pick-up or drop-off fees as employed in a number of real-world services. For instance, \textit{Free2move} charges a drop-off fee for trips ending in designated zones of the city, and \textit{Zity} in Milan imposes both drop-off and pick-up fees for trips ending and starting in different zones, see \citet{free2move} and \citet{zity} respectively.
The specification of the rental fee is problem specific and does not have an impact on the derivation of the model. We will provide an example in our experiments, see, \Cref{sec: instance_gen}. 

Let $\set{V}$ denote the set of shared vehicles. Relocation decisions are modeled by binary variables $s:= (s_{vi})_{v \in \set{V}, i \in \set{I}}$ where $s_{vi}$ takes value $1$ if vehicle $v$ is made available at station $i$, $0$ otherwise. Relocating vehicle $v$ from its current station to station $i$ generates a cost $C_{vi}$.

The demand within the target interval of time is uncertain at the time pricing and relocation decisions are made. In addition, the probability distribution of demand is influenced by pricing decisions, hence \textit{endogenous} to the problem. Given a measurable space $(\Omega,\mathcal{F})$, we model the demand for carsharing rides during the target interval of time as a random variable $\boldsymbol{\xi}: \Omega \rightarrow \Z^{|\set{A}|}_{\geq 0}$. Hence, element $\boldsymbol{\xi}_{ij}$ represents the random demand on arc $(i,j)\in\set{A}$.
In particular, we assume that there exist $K$ potential customers in the system (e.g., all the users registered to the service). Hence, we define $\Omega:=\{0,1\}^{K}$ and we set $\set{F}$ to be the power set of $\Omega$. This means that each random event $\omega$ is a binary vector of length $K$ that represents whether or not each customer requests to rent a vehicle in the target period. 
Given this setup, for $\omega\in \Omega$, the realization of the $(i,j)$-th random variable is defined as
$$\xi_{ij}(\omega) = \sum_{k|i(k)=i, j(k)=j}\omega_k \qquad \forall i,j\in \set{A}$$
where $i(k)\in\set{I}$ and $j(k)\in\set{I}$ represent the origin and destination stations of customer $k$, respectively. Then $\xi(\omega):=(\xi_{ij}(\omega))_{(i,j)\in\set{A}}$.
\Cref{fig: initial_prob_define} provides an example of the realization assuming only one origin-destination pair. 

\begin{figure}[htbp!]
\centering
\includegraphics[width=0.57\linewidth]{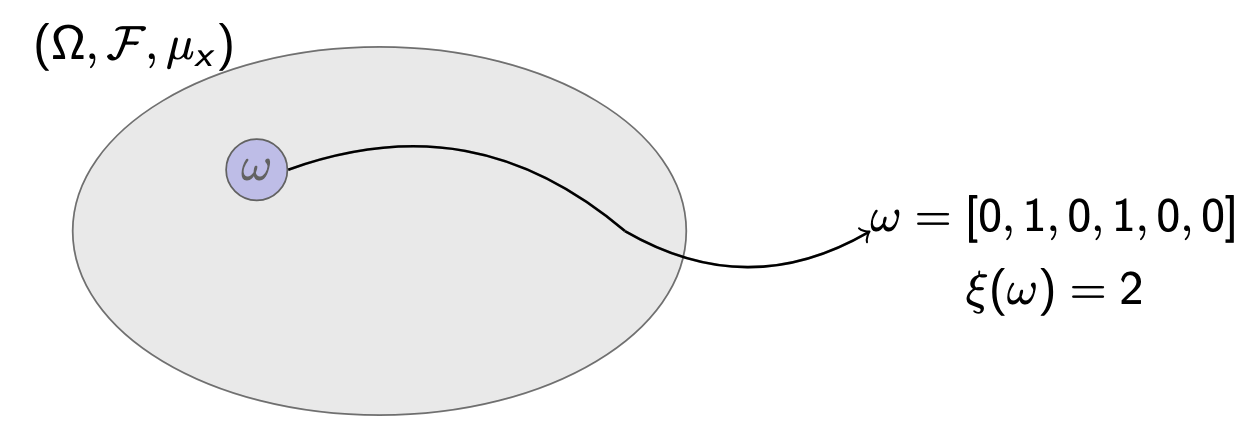}
\caption{Illustration of probability space $(\Omega, \set{F}, \mu_x)$.}
\label{fig: initial_prob_define}
\end{figure}

We let the probability measure on $(\Omega,\mathcal{F})$, denoted $\mu_x$, be fully determined by pricing decisions $x$. 
In particular, for each potential customer $k$, we assume the probability of using carsharing conditional on the price decision $x$ is known, denoted by $P(\omega_k=1|x)$. This probability can be obtained, for example, using discrete choice models such as the Logit model. Then, the probability measure $\mu_x$ can be defined for each $\omega\in\Omega$ as 
\begin{align}\label{eq:mu-omega}
\mu_x(\{\omega\})=\prod_{k\vert \omega_k=1}P(\omega_k=1\vert x)\prod_{k\vert \omega_k=0}(1-P(\omega_k=1\vert x))    
\end{align}
Finally, with these elements in place, we define the distribution of $\boldsymbol{\xi}$ on $\mathbb{Z}^{|\mathcal{A}|}_{\geq 0}$ conditional on $x$ as $P_x$. Given an integer vector $N\in\mathbb{Z}^{|\set{A}|}_{\geq 0}$, this is defined as 
\begin{align}
\label{form: prob}
P_x(\boldsymbol{\xi} = N) = \mu_x\left(\left\{\omega \in \Omega | \xi(\omega) = N \right\}\right)    
\end{align}
\color{black}Also in this case, the specific way in which $P(\omega_k=1\vert x)$ for each $k$ is obtained depends on the particular case and on the assumptions made regarding customer behavior. This does not have an impact on the general model proposed. A particular specification of $\mu_x(\{\omega\})$ is described in \Cref{sec: instance_gen} and \aref{app:utility function}.
\color{black}

We model the decision problem as a stochastic program. Here, pricing and relocation decisions are made prior to knowing the realization of demand and in preparation for that. Once decisions have been made, rental demand in the target time interval arrives at random according to the conditional distribution $P_x$. Then, rental profit $Q(x, s, \xi)$ is obtained as a consequence of pricing and relocation decisions and the demand materialized (i.e., $\xi$). 
In \Cref{sec:2nd-stage problem} we model $Q(x, s, \xi)$ as an optimization problem, hence obtaining a two-stage stochastic program. 
The carsharing pricing and relocation problem can be modeled as the following two-stage stochastic program with decision-dependent uncertainty.
\begin{subequations}
\label{form: general_form}
    \begin{align}
    \max &-\sum_{v \in \set{V}} \sum_{i \in \set{I}} C_{vi} s_{vi} +\mathbb{E}_{P_x} \big[Q(x, s, \xi)\big] \\
   s.t.~ & \sum_{l \in \mathcal{L}}x_{z_1, z_2, l} =1 & \forall z_1, z_2 \in \set{Z}, \forall l \in \set{L} \label{constr: price_set}\\
    & \sum_{i \in \set{I}} s_{vi} = 1& \forall v \in \set{V} \label{constr: veh_relocation}\\
    & x_{z_1,z_2, l} \in \{0,1\} & \forall z_1, z_2 \in \set{Z}, \forall l \in \set{L} \label{domain:var_x}\\
    & s_{vi} \in \{0,1\} & \forall v \in \set{V}, \forall i \in \set{I} \label{domain:var_s}
    \end{align} 
\end{subequations}

In problem \eqref{form: general_form}, the objective function consists of a negative term corresponding to relocation costs and the conditional expectation of second-stage profits. 
Constraints \eqref{constr: price_set} state that exactly one price level must be assigned to each pair of zones. Constraints \eqref{constr: veh_relocation} state that each shared vehicle should be made available at exactly one carsharing station $i \in \set{I}$. Constraints \eqref{domain:var_x} and \eqref{domain:var_s} define the domain of the decisions variables. Problem \eqref{form: general_form} allows for various definitions of the second-stage profit $Q(x, s, \xi)$ depending on the specifics \color{black} and operating strategies \color{black} of the carsharing system. We provide \color{black} one such definition -- and remark on possible extensions of it -- \color{black} in \Cref{sec:2nd-stage problem}, assuming a demand-pooling strategy.

\subsection{Expected revenue}\label{sec:2nd-stage problem}
To model expected revenue conditional on pricing decisions we assume that the CSO adopts a ``\textit{demand pooling}" strategy, see, e.g., \citet{nair2011fleet,zhao2018integrated, li2022data}. This means that, to rent a vehicle during a given interval of time, users must send requests before a given deadline. After the deadline, the CSO matches supply and demand, hereby fulfilling or rejecting rental requests.
\color{black}
This entails that, periodically, the operator is faced with the following problem: \textit{Given a set of rental requests, decide which ones to accept}. 
There are obviously many policies an operator may adopt. The most straightforward is to let profits drive the decisions. For simplicity, we present this policy in \Cref{sec:recourse2} and show that it can be implemented by means of a MILP problem.

However, an operator might wish to adopt different policies. For example, a \textit{fixed assignment} of vehicles to requests based on, e.g., characteristics of the request such as the destination station or the type of customer. This might provide the operator additional instruments to control, e.g., the levels of service, the overall accessibility of the service, the future distribution of the fleet and which zones get access to the service. 
For example, the operator may prioritize the allocation of vehicles to customers traveling to specific destinations -- that are not necessarily the most profitable -- in order to ensure more homogeneous access to the service across the operating area. 

We stress that, while the implemented policy presented in \Cref{sec:recourse2} is, perhaps, simplified in some real-life features (e.g., they assume each vehicle is rented at most once during the target time interval), \color{black} it results in a reasonable proxy of expected profits. 
To illustrate the performance of the method on more involved assignment policies, in \aref{sec: recourse problem wa} we provide a version of $Q(x,s,\xi)$ where the assignment of vehicles to customers follows a fixed policy based on the customers destinations. In particular, we assume, as in \citet{soppert2022differentiated}, that the operator bounds the number of vehicles directed to each destination by the share of demand to that destination. 

\subsubsection{A profit-driven vehicle allocation policy}\label{sec:recourse2}
\color{black}
Given a first-stage solution $(x, s)$ and a demand realization $\xi$ of $\boldsymbol{\xi}$, we formulate the recourse problem of allocating supply to demand with the goal of maximizing profits as a MILP. We show that $Q(x, s, \xi)$ becomes a special transportation problem with independent sources and mutually-exclusive sinks. 

\color{black}
Let $P_{ijl}$ denote the rental price for picking up a vehicle at station $i$ and returning it at station $j$ when price level $\set{L}$ is applied. 
Define a mapping $\zeta: \set{I} \rightarrow \set{Z}$ such that $\zeta(i)$ identifies the zone to which station $i$ belongs. Given a feasible first-stage solution $x$, the price for renting a vehicle at station $i$ and returning it at station $j$ is  $\sum_{l\in\set{L}}P_{ijl} x_{\zeta(i),\zeta(j), l}$.
Note that $\sum_{l\in\set{L}}x_{\zeta(i),\zeta(j), l}=1$ since $x$ is feasible. We define $\set{J}_i:=\{j\in\set{I}:(i,j)\in\set{A}\}$ as the set of destination stations connected to a source station $i \in \set{I}$. We make $|\set{L}|$ copies of each destination station $j \in \set{J}_i$. These represent the sinks connected to source station $i$. We then define variables $r_{ijl}$ to represent the flow of vehicles from source $i$ to sink $(j,l)$, that is, the number of vehicles serving demand on arc $(i,j)$ at price level $l$, see \Cref{fig:rpwopa}. 
\begin{figure}[htbp!]
    \centering
    \includegraphics[width=0.48\linewidth]{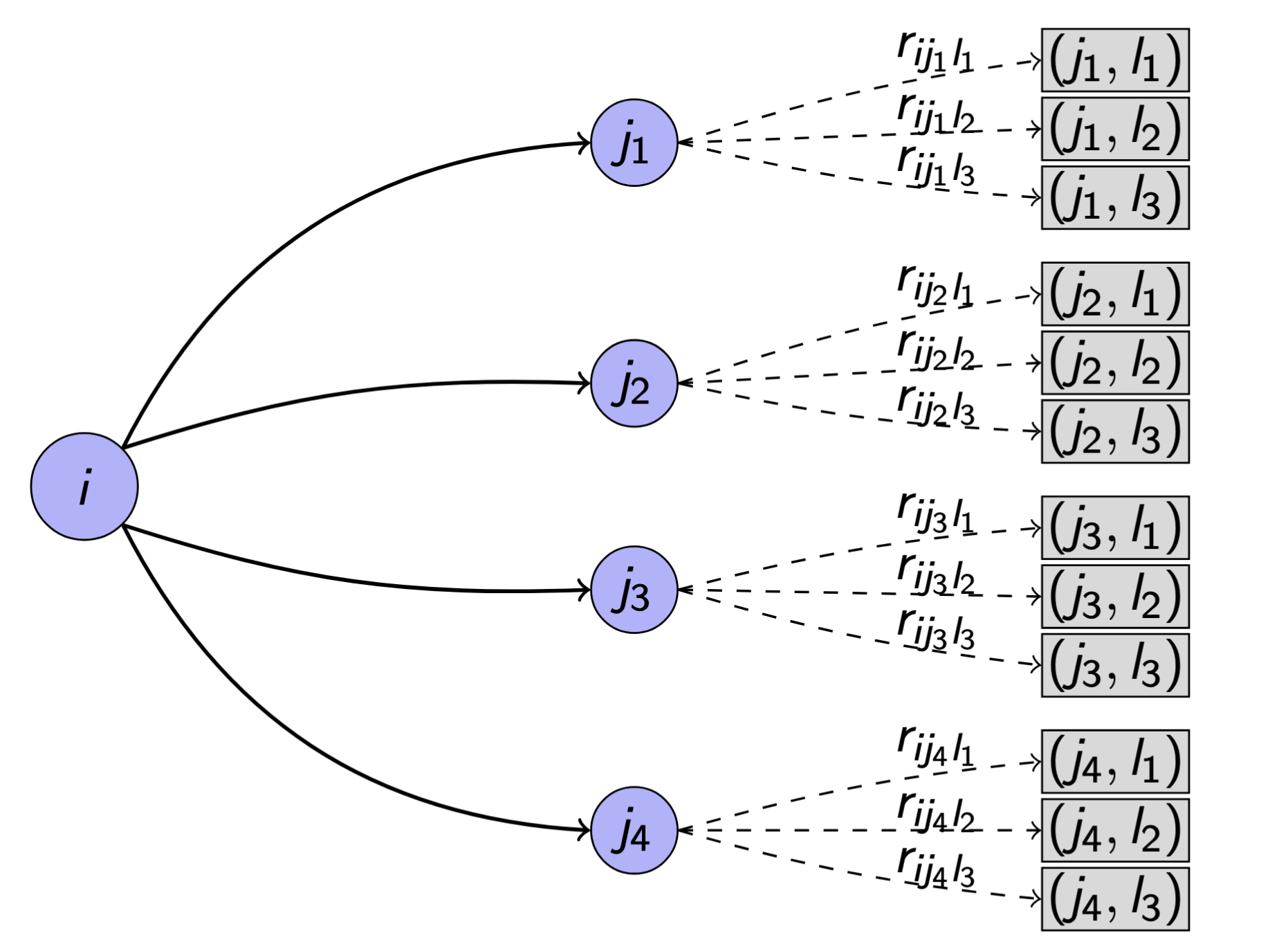}
    \caption{Sinks $(j,l)$ for a given source station $i$.}
    \label{fig:rpwopa}
\end{figure}

With this notation in place we can formulate the recourse problem as
\begin{subequations}
\begin{align}
Q(x, s, \xi) = \max & \sum_{(i,j)\in\set{A}}\sum_{l\in\set{L}} P_{ijl} r_{ijl}   & \label{bilinear_obj_2}\\
s.t. ~& r_{ijl} \leq \xi_{ij} x_{\zeta(i), \zeta(j), l}, &\forall (i,j)\in\set{A}, \forall l \in \set{L} \label{constr_tp1}\\
&\sum_{(j, l) \in \set{J}_i \times \set{L}} r_{ijl} \leq \sum_{v \in \set{V}}s_{vi}, & \forall i \in \set{I} \label{constr_tp2}\\
&r_{ijl} \in \mathbb{Z}_{\geq 0}, &\forall (i, j)\in \set{A}, \forall l \in \set{L}
\end{align}
\label{form: rpwopa}
\end{subequations}
Constraints \eqref{constr_tp1} state that no vehicles will be sent to sink $(j, l)$ if $x_{\zeta(i), \zeta(j), l}= 0$, and we can send up to the demand of customer $j$ (i.e., $\xi_{ij}$), otherwise. Observe that, if $x$ is feasible for \eqref{form: general_form}, only one of the $(j,l)$ sinks will have positive demand for each $j$. Hence, for each $(i,j)$ only one of these $r_{ijl}$ variables can take a positive value. Constraint \eqref{constr_tp2} ensures that the total fulfilled demand does not exceed the supply at station $i$. 

\color{black}

In \aref{sec: recourse problem wa} we extend problem \eqref{form: rpwopa} to include a proportional assignment of vehicles depending on the destination of the requests. 
In what follows, problem \eqref{form: rpwopa} will be referred to as the \textit{recourse problem without proportional vehicle assignment} (RPWoPA) while the problem in \aref{sec: recourse problem wa} will be referred to as the \textit{recourse problem with proportional vehicle assignment} (RPWPA).
Later, in \Cref{sec: VIs}, we show that the profit delivered by problem RPWoPA provides an upper bound to that of problem RPWPA. In \Cref{sec:system_perform} we compare the two allocation policies using different performance metrics, based on a case study.

\color{black}

\section{Exact solution method} \label{sec: L-shaped method}
We propose an exact decomposition method that accounts for decision-dependent demand uncertainty. 
\color{black}
The method instantiates and extends the L-shaped method for stochastic programs with decision-dependent uncertainty (henceforth, \LS) developed by \cite{pantuso2025shaped}. The \LSt method is applicable to problems where the space of feasible solutions is partitioned into a finite collection of subsets. To each subset there corresponds a specific probability distribution of the random parameters. The method consists of adding distribution-specific optimality and feasibility cuts. The specificity of the cuts is as follows: given a probability distribution, the corresponding optimality and feasibility cuts are effective only in the subregion which enforces that distribution, and redundant otherwise. 

We adopt a number of problem-specific instantiations and extensions of the method. 
\begin{itemize}
    \item In \cite{pantuso2025shaped} the authors test the method on a problem where the feasible region is partitioned into \textit{polyhedra}. We show how the method can be implemented on a problem where the feasible region has a \textit{combinatorial} structure and is partitioned into a collection of finite-dimensional subsets. 
    \item We develop problem-specific optimality cuts. We show that these cuts can be separated efficiently by exploiting the specific structure of both the RPWoPA and RPWPA problems. In particular, we develop closed-form primal and dual solutions, which enable efficient separation routines. 
    \item We develop effective problem-specific valid inequalities that provide non-trivial upper bounds for the over-estimator of the conditional expected revenue. These are crucial to an efficient implementation of the method.
    \item We provide cut-sharing mechanisms. These allow us to reuse the cuts developed for the problem with RPWoPA when solving the problem with RPWPA. These cuts significantly improve the performance of the method when solving the problem with RPWPA.
\end{itemize}
\color{black}

We start by observing that the feasible region of problem \eqref{form: general_form} can be seen as the union of disjoint subregions. All solutions in a given subregion determine the same distribution $P_x$. Let $\set{C}$ denote the set of feasible $(x, s)$ solution, that is 
$$\set{C} = \left\{(x, s) \in \{0,1\}^{|\set{Z}| \times|\set{Z}|\times |\set{L}|} \times \{0,1\}^{|\set{V}| \times |\set{I}|}: \eqref{constr: price_set}-\eqref{constr: veh_relocation} \right\}$$
Let then $\set{X} = \{x \in \{0,1\}^{|\set{Z}| \times|\set{Z}|\times |\set{L}|}: \eqref{constr: price_set}\}$ identify the set of all feasible pricing solutions. 
Observe that $\set{X}$ is finite and countable. Let $\set{D} := \{1, 2, \cdots, |\set{X}|\}$ index the solution in $\set{X}$, that is $x_d$, $d=1,\ldots, |\set{X}|$. For each $d\in\set{D}$, we obtain a subset $\set{C}_d =\left \{(x_d, s) \in \set{C}\right \}\subset \set{C}$.
In words, the subset $\set{C}_d$ contains all $(x, s)$ solutions having component $x=x_d$. Observe that 
\begin{itemize}
    \item $\set{C}_{d_1}\cap \set{C}_{d_2} = \emptyset, \forall d_1,d_2 \in \set{D}, d_1\neq d_2$.
    \item $\cup_{d\in\set{D}} \set{C}_d = \set{C}$.
\end{itemize}
Observe further that the distribution of demand $\boldsymbol{\xi}$ depends only on $x$, see \Cref{sec:problem_describe}. This entails that there are $\lvert \set{D}\rvert$  distributions, one for each $x\in\set{X}$. Let such distributions be denoted as $P_d$.
Denote also $\boldone_d: \set{C} \rightarrow \{0,1\}$ the characteristic function of $\set{C}_d$, that is $\boldone_d(x,s)=1$ if $(x,s)\in \set{C}_d$, $0$ otherwise.
Then we can reformulate problem \eqref{form: general_form} as  
\begin{align}
\label{form: general_reform}
   \max_{(x, s)\in \set{C}} \bigg\{-\sum_{v\in\set{V}}\sum_{i\in\set{I}}C_{vi}s_{vi}+\sum_{d \in \set{D}} \boldone_d(x, s) \mathbb{E}_{P_d} \big[Q(x,s, \xi)\big]\bigg\}
\end{align}
Let $Q(x, s) = \sum_{d\in\set{D}} \boldone_d(x, s) \mathbb{E}_{P_d} \big[Q(x,s, \xi)\big]$. 
Observe that each $P_d$, for $d \in \set{D}$, is a distribution on $\mathbb{Z}^{|\set{A}|}_{\geq 0}$ and is thus naturally discrete. Therefore, we can let $\{\xi^{nd}\}_{n=1}^{N_d}$ be the support of $\boldsymbol{\xi}$ under distribution $P_d$ and let $\pi^{nd}$ the probability of realization $n$ under $P_d$. Then 
\begin{equation}\label{eq:condexp}Q(x, s) = \sum_{d\in\set{D}} \boldone_d(x, s)  \sum_{n=1}^{N_d}  \pi^{nd} Q(x, s, \xi^{nd})\end{equation}
In practical applications, the support of $P_d$, $d\in \set{D}$ may still be very large. In that case, as in our experiments in \Cref{sec: computational study}, one can use approximations which rely on a smaller set of realizations.

Let $\phi$ be a non-negative variable. We can then reformulate problem \eqref{form: general_reform} as follows. 
\begin{align}\tag{MP}
\label{form: mp}
    \max_{(x, s) \in \set{C}, \phi \geq 0}  \left\{-\sum_{v\in\set{V}}\sum_{i\in\set{I}}C_{vi}s_{vi}+\phi |s.t. ~ \phi \leq Q(x, s)\right\}
\end{align}

The method relaxes constraints $\phi \leq Q(x, s)$ and re-builds them iteratively. 
That is, at each iteration $k=0,1,\ldots$ the method solves problem \eqref{RMP}
\begin{align}
\label{RMP}\tag{RMP}
     \max_{(x, s)\in\set{C}, \phi\geq 0} \big\{-\sum_{v\in\set{V}}\sum_{i\in\set{I}}C_{vi}s_{vi} + \phi|s.t. ~ \phi \leq o(x,s), o\in\set{O}^k\big\}
\end{align}
where $\set{O}^k\subseteq \set{O}$ and $\set{O}$ is a set of distribution-specific optimality cuts $o:\set{C}\to \R$.
The solution procedure is described in \Cref{pseudocode: L-shaped method}. 
The main difference with the classical L-Shaped method is that, prior to generating cuts, the algorithm identifies the probability distribution enforced by $x^k$ (see, \cref{line: d_indentify}). The value of $Q(x^k,s^k)$ can then be computed as 
$$Q(x^k, s^k) =   \sum_{n=1}^{N_{d_k}}  \pi^{nd_k} Q(x, s, \xi^{nd_k})$$
since only $\boldone_{d_k}(x^k, s^k)$ takes value one in \eqref{eq:condexp}.
Hence, the algorithm generates distribution-specific cuts that are effective only under the distribution they were generated for. If no such cut is identified, the algorithm terminates. 

\color{black}
The form of the optimality cuts depends on the characteristics of $Q(x,s,\xi)$ and, therefore, on whether we use the RPWoPA or the RPWPA. In \Cref{sec:ds_cuts_rpwopa} we derive optimality cuts for the RPWoPA and show how to separate them efficiently. In particular, we show that, even if the problem is a MILP, we can generate exact optimality cuts using linear programming duality. In \Cref{sec:ds_cuts_integer} we derive optimality cuts for the RPWPA. These cuts can be adopted every time $Q(x,s,\xi)$ is an integer programming problem. For both the RPWoPA and RPWPA  there exists a finite set $\set{O}$ of optimality cuts. These finitely many optimality cuts ensure that the algorithm converges finitely. Finally, in \Cref{sec: VIs} we introduce a number of valid inequalities which significantly aid convergence. In particular, we show that the cuts developed for the RPWoPA can be used as valid inequalities when solving the problem with RPWPA. \color{black}

\begin{algorithm}
\footnotesize
	\caption{Pseudocode of the algorithm for solving problem \eqref{form: general_reform}.}
 \begin{algorithmic}[1]
     \State $\texttt{solved} \gets \texttt{false}$
     \State $k\gets 0$
     \While{not $\texttt{solved}$} \label{alg_line:converge}
     \State Solve \eqref{RMP}, and let $(x^k, s^k, \phi^k)$ be the optimal solution
     \State Identify $d_k \in \set{D}$ such that $(x^k, s^k) \in \set{C}_{d_k}$ \label{line: d_indentify}
     \State Compute $Q(x^k, s^k)$
     \If {$\phi^k > Q(x^k, s^k)$} \Comment{Optimality condition is violated.}
     \State Identify a violated optimality cut $o$ from $\set{O}$
     \State $\set{O}^{k+1}=\set{O}^k\cup \{o\}$
     \State $k\gets k+1$
     \Else 
     \State $\texttt{solved} \gets \texttt{true}$    \Comment{The algorithm terminates.} \EndIf
     \EndWhile
 \end{algorithmic} \label{pseudocode: L-shaped method}
\end{algorithm}

Observe that, since \eqref{RMP} is a mixed-integer linear program, \Cref{pseudocode: L-shaped method} can be embedded into a Branch-and-Bound (B\&B) procedure as described in \cite{laporte1993integer}. In practical implementations (as in \Cref{sec: computational study}) the routines in \Cref{pseudocode: L-shaped method} are only invoked once the B\&B method reaches an integer node. This entails that the B\&B tree is created only once, upon starting the algorithm. This is, in general, preferable to building a new tree every time \eqref{RMP} is solved. 

\color{black}
\subsection{Distribution-specific duality-based optimality cuts} \label{sec:ds_cuts_rpwopa}

In this section we derive optimality cuts for the RPWoPA, i.e., \eqref{form: rpwopa}. The recipe used for the cuts is immediately applicable any time $Q(x,s,\xi)$ forms a linear programming problem.
\color{black}

We start by observing that problem \eqref{form: rpwopa} can be decomposed in one smaller problem for each $i\in\set{I}$. That is, 
$$Q(x,s,\xi)=\sum_{i\in\set{I}}Q^i(x,s,\xi)$$ where
\begin{subequations}\label{form: rpwopa_i}
\begin{align}
Q^i(x, s, \xi) = \max & \sum_{j\in\set{J}_i}\sum_{l\in\set{L}} P_{ijl} r_{jl}   & \\
s.t. ~ & r_{jl} \leq \xi_{ij} x_{\zeta(i), \zeta(j), l}, &\forall j\in\set{J}_i, \forall l \in \set{L} \\
&\sum_{(j, l) \in \set{J}_i \times \set{L}} r_{jl} \leq \sum_{v \in \set{V}}s_{vi}, & \\
&r_{jl} \in \mathbb{Z}_{\geq 0}, &\forall  j\in \set{J}_i, \forall l \in \set{L}
\end{align}
\end{subequations}
Now, each $Q^i(x,s,\xi)$ is an integer linear program. However, we can now show that its optimal solution can be found by solving its linear programming relaxation.
\begin{proposition}\label{prop:rpwopa_tu}
    Let $(x,s)$ be feasible for \eqref{RMP} and $\xi\in \Z_{\geq 0}^{\vert\set{A}\vert}$.   
    For given $i\in\set{I}$ let $Q^i_{LP}(x,s,\xi)$ be the linear programming relaxation of $Q^i(x,s,\xi)$.
    Let $(\hat{r}_{jl})_{j\in \set{J}_i,l\in\set{L}}$ be the optimal solution to $Q^i_{LP}(x,s,\xi)$ and $\set{R}$ be the set of optimal integer solutions to $Q^i(x,s,\xi)$. 
    Then  $(\hat{r}_{jl})_{j\in \set{J}_i,l\in\set{L}}\in\set{R}$ and $Q^i_{LP}(x,s,\xi)=Q^i(x,s,\xi)$.
\end{proposition}

\begin{proof}{Proof of \Cref{prop:rpwopa_tu}}
    The proof follows immediately by observing that
    \begin{enumerate}
        \item Since $(x,s)$ is feasible for \eqref{RMP} and $\xi\in \Z_{\geq 0}^{\vert\set{A}\vert}$, the right-hand side of problem $Q^i(x,s,\xi)$ is an integer vector.
        \item The matrix of coefficients of $Q^i(x,s,\xi)$ is \textit{totally unimodular}.
    \end{enumerate}
    Under these conditions the solution to $Q^i_{LP}(x,s,\xi)$ solves $Q^i(x,s,\xi)$, see \cite[Chapter 3]{wolsey2020integer}. 
    It remains to show that the matrix of coefficients is totally unimodular. \color{black}For that purpose, observe that the matrix of coefficients consists of a row of ones appended below an identity matrix. That is,   \color{black}
    \begin{equation*}   
    \begin{pmatrix} 
        1 & 0 & 0 & 0 &  \cdots & 0 \\
        0 & 1 & 0 & 0 & \cdots & 0 \\
        \vdots & \vdots & \vdots & \vdots & \ddots & \vdots \\
        0 & 0 & 0 & 0 & \cdots & 1 \\
        1 & 1 & 1 & 1 & \cdots & 1\\
        \end{pmatrix}
    \end{equation*}
    Let $a_{ij}$ be its elements. It is easy to verify that
    \begin{enumerate}[(i.)]
    \item Each element in the matrix takes the value in set $\{+1, -1, 0\}$. 
    \item Each column contains at most two nonzero coefficients.
    \item There exists a partition $(M_1, M_2)$ of the set of rows such that each column $j$ containing two nonzero coefficients satisfies $\sum_{m\in M_1} a_{m j} - \sum_{m \in M_2} a_{m, j} =0$. Such partition can easily be obtained by taking the first $\vert \set{J}_i \vert \vert \set{L} \vert$ rows as set $M_1$ and the last row as set $M_2$.
\end{enumerate}
These represent sufficient conditions for the matrix to be totally unimodular, see \cite[Proposition 3.2]{wolsey2020integer}.
\Halmos
\label{proof:TU_matrix} 
\end{proof}

By virtue of \Cref{prop:rpwopa_tu} we can derive duality-based optimality cuts. In particular, \cite{pantuso2025shaped} show that inequalities
\begin{align}
   \label{form: LP_cuts_tp_dual} 
   \phi \leq \sum_{i \in \set{I}} \sum_{n=1}^{N_{d}} \pi^{nd} \left ( \sum_{(j,l)\in \set{J}_i \times \set{L}} \alpha_{ijl}^{nd} \xi^{nd}_{ij} x_{\zeta(i), \zeta(j), l} + \beta_{i}^{nd} \sum_{v \in \set{V}} s_{vi} \right) + U(2|\set{Z}| - \sum_{z_1 \in \set{Z}} \sum_{z_2 \in \set{Z}} x_{z_1, z_2, l(z_1,z_2,d)}) 
\end{align}
generated for all $d\in\set{D}$ and for all extreme points $(\alpha_{ijl}^{nd})_{(i,j)\in\set{A},l\in\set{L}}$ and $(\beta_i^{nd})_{i\in\set{I}}$ in the feasible region of the dual problem to \eqref{form: rpwopa} form a finite set of optimality cuts that ensure $\phi \leq Q(x,s)$ for all $(x,s)\in\set{C}$.

Observe how the last term in \eqref{form: LP_cuts_tp_dual} makes the cut distribution-specific. In particular, it enforces the cut only for solutions $(x,s)\in\set{C}_d$. For these solutions, the last term of the cut vanishes and the inequality becomes a supporting hyperplane to the epigraph of $Q(x,s)=\sum_{n=1}^{N_{d}}  \pi^{nd} Q(x, s, \xi^{nd})$. For distributions other than $d$, the cut yields a valid upper bound which makes the cut redundant. 

The computational effort to generate a cut \eqref{form: LP_cuts_tp_dual} from a given solution $(x^k,s^k)$ amounts to the cost of solving one linear program $Q^i(x^k,s^k,\xi^{nd_k})$ for each $i\in\set{I}$ and $n=1,\ldots,N_{d_k}$. 
We can now show that primal and dual solutions to $Q^i(x,s,\xi)$ can be obtained in closed form. 

\begin{proposition}
    \label{prop: veh_lack 2}
    Let $(x, s)$ be a feasible solution to \eqref{RMP}.
    For a given station $\hat{i} \in \set{I}$ if the vehicle supply is larger than the total demand, that is
    $$\sum_{v \in \set{V}}s_{v \hat{i}} \geq \sum_{(j,l)\in\set{J}_{\hat{i}}\times\set{L}}\xi_{\hat{i}, j}x_{\zeta(i), \zeta(j), l}$$
    then the optimal solution is
    $$\hat{r}_{jl} = \xi_{\hat{i}, j}x_{\zeta(\hat{i}), \zeta(j), l}, ~~\forall (j, l) \in \set{J}_{\hat{i}}\times\set{L}$$
    Otherwise, assume
    $$\sum_{v\in\set{V}}s_{v\hat{i}} < \sum_{(j,l)\in\set{J}_{\hat{i}}\times\set{L}}\xi_{\hat{i}, j}x_{\zeta(i), \zeta(j), l}$$
    Let $\set{K}=\{(j, l) \in \set{J}_{\hat{i}} \times \set{L}|x_{\zeta(\hat{i}), \zeta(j), l} = 1\}$. Observe that $|\set{K}| = |\set{J}_{\hat{i}}|$ as $x$ is a feasible pricing decision.
    Sort the values $P_{jl}$ for $(j, l)$ indices in set $\set{K}$ such that, for  $k=2,\ldots, \vert\set{J}_{\hat{i}}\vert-1$ 
    $$P_{(j,l)^{(k-1)}} \geq P_{(j,l)^{(k)}} \geq P_{(j,l)^{(k+1)}}$$
    For a given $(j, l)^{(k)}$, denote the corresponding $j$ and $l$ as $j^{(k)}$, $l^{(k)}$, i.e., $(j,l)^{(k)}=(j^{(k)},l^{(k)})$. 
    Identify the index $2\leq\hat{k} \leq |\set{J}_{\hat{i}}|$ such that 
    $$\sum_{k=1}^{\hat{k}-1} \xi_{\hat{i}, j^{(k)}} \leq \sum_{v\in\set{V}}s_{v\hat{i}}, ~\text{and}~\sum_{k=1}^{\hat{k}} \xi_{\hat{i}, j^{(k)}}> \sum_{v\in\set{V}}s_{v\hat{i}}$$ 
 
    Then, the optimal solution is 
    \begin{equation}
    \hat{r}_{jl} = 
    \begin{cases}
    \xi_{\hat{i},j} & \quad \forall (j,l) \in \{(j,l)^{(1)}, (j,l)^{(2)}, \cdots, (j,l)^{(\hat{k}-1)} \}, \\[1ex]
    \sum_{v\in\set{V}}s_{v\hat{i}}-\sum_{k=1}^{\hat{k}-1}\xi_{\hat{i},j^{(k)}} & \quad (j, l) = (j,l)^{(\hat{k})}, \\[1ex]
    0 & \quad \forall (j,l) \in \set{J}_{\hat{i}} \times \set{L} \setminus \{(j,l)^{(1)}, (j,l)^{(2)}, \cdots, (j,l)^{(\hat{k})} \}.
\end{cases}
\end{equation}
\end{proposition}
\begin{proof}{Proof of \Cref{prop: veh_lack 2}}
\color{black}
The proof of the first part of \Cref{prop: veh_lack 2} is trivial. If the supply is sufficient, all rental demand is satisfied and constraint \eqref{constr_tp2} becomes redundant. 

When the total supply is smaller than the demand, we first fulfill the demand to sinks $(j, l) \in \set{K}$ with non-increasing order of $P_{jl}$. Then, we observe that $j^{(\hat{k})}$ is the first state for which demand cannot be fully satisfied. Hence, we allocate the remaining vehicles of $\sum_{v\in\set{V}}s_{v\hat{i}}$ to this station. The demand of the remaining stations is not satisfied. No feasible allocation of demand would generate a higher objective value. \Halmos
\color{black}
\end{proof}
\color{black}

We can now proceed to show that, given a primal solution to $Q^{i}(x,s,\xi)$, a dual solution can likewise be obtained in closed form.

\begin{proposition}
\label{prop:dual_solutions} 
Let $(x, s)$ be a feasible solution to \eqref{RMP}. Assume $P_{ijl}\geq 0$ for all $(i,j)\in\set{A}$, $l\in \set{L}$.
Let $\hat{r}:=(\hat{r}_{jl})_{j\in\set{J}_{\hat{i}}, l\in\set{L}}$ be a primal solution to $Q^i(x, s, \xi)$ for a given $\hat{i}$ and a realization $\xi$. Then the dual solutions to $Q^i(x, s, \xi)$ is obtained as follows. 

If (case 1) $\sum_{v \in \set{V}}s_{v\hat{i}} \geq \sum_{(j, l) \in \set{J}_{\hat{i}} \times \set{L}} \xi_{\hat{i}j} x_{\zeta(\hat{i}), \zeta(j), l}$, then
\begin{align}
        &\hat{\beta} = 0, ~\hat{\alpha}_{jl} = P_{jl} & \forall (j ,l) \in \set{J}_{\hat{i}} \times \set{L} \label{dual: condition1}
\end{align}
Otherwise (case 2)
\begin{align}
        \hat{\beta} = \max_{(j, l) \in \set{J}_{\hat{i}} \times \set{L}}\left\{P_{jl}\vert r_{jl} < \xi_{ij}x_{\zeta(\hat{i}), \zeta(j), l}, \xi_{ij} x_{\zeta(\hat{i}), \zeta(j), l}>0\right\}, ~\hat{\alpha}_{jl} = \max\{0, P_{jl} - \hat{\beta}\}, \forall (j ,l) \in \set{J}_{\hat{i}} \times \set{L} \label{dual: condition2} 
    \end{align} 
\end{proposition}

\begin{proof}{Proof of Proposition \ref{prop:dual_solutions}.}
Given a specific origin station $\hat{i}$ and a certain scenario $\xi$, the dual problem to $Q^i(x, s, \xi)$ is
\begin{subequations}
\label{dual_problem:tp}
    \begin{align}
        \min &Q^D=\sum_{(j, l) \in \set{J}_{\hat{i}} \times \set{L}} \xi_{\hat{i} j} x_{\zeta(\hat{i}),\zeta(j), l} \alpha_{jl} + \sum_{v \in \set{V}}s_{v \hat{i}} \beta & \label{dual_sp_obj:tp}\\
        s.t.~& \alpha_{jl} + \beta \geq P_{jl}, &\forall (j,l) \in \set{J}_{\hat{i}} \times \set{L} \label{dual_sp_constr:tp}\\
        &\alpha_{jl} \geq 0, & \forall (j, l) \in \set{J}_{\hat{i}} \times \set{L}\\
        &\beta \geq 0 &
    \end{align}
\end{subequations}
We prove case 1 and case 2 separately.

Case 1. Consider first the sub-case where $\sum_{v \in \set{V}}s_{v\hat{i}} > \sum_{(j, l) \in \set{J}_{\hat{i}} \times \set{L}} \xi_{\hat{i}j} x_{\zeta(\hat{i}),\zeta(j),l}$. We have 
$$\sum_{v \in \set{V}}s_{v\hat{i}} > \sum_{(j, l) \in \set{J}_{\hat{i}} \times \set{L}} \xi_{\hat{i}j} x_{\zeta(\hat{i}),\zeta(j),l} \stackrel{\eqref{constr_tp1}}{\geq} \sum_{(j, l) \in \set{J}_{\hat{i}} \times \set{L}} \hat{r}_{jl}$$ 
Hence, constraints \eqref{constr_tp2} are not tight. Complementary slackness conditions state 
$$\beta(\sum_{v \in \set{V}}s_{v \hat{i}} - \sum_{(j,l)\in \set{J}_{\hat{i}} \times \set{L}} r_{jl})=0$$
Hence, we must set $\hat{\beta}=0$. 
Then, from \eqref{dual_problem:tp} we obtain that $\hat{\alpha}_{jl} = P_{jl}$ for all $(j,l)\in\set{J}_{\hat{i}}\times\set{L}$. 

In the remaining sub-case, namely when $\sum_{v \in \set{V}}s_{v\hat{i}} = \sum_{(j, l) \in \set{J}_{\hat{i}} \times \set{L}} \xi_{\hat{i}j} x_{\zeta(\hat{i}),\zeta(j),l}$. Then, by optimality of $\hat{r}_{jl}$ we have (see \Cref{prop: veh_lack 2})
$$\hat{r}_{jl} = \xi_{\hat{i}j} >0 , \forall(j, l) \in \{(j, l) \in \set{J}_{\hat{i}} \times \set{L} | \xi_{\hat{i}j} x_{\zeta(\hat{i}), \zeta(j), l} > 0\}$$ 
The corresponding complementarity slackness conditions require that 
$$P_{jl} - \alpha_{jl} - \beta = 0$$
Substituting for $\alpha_{jl}$ and $\sum_{v\in\set{V}}s_{v\hat{i}}$ in \eqref{dual_sp_obj:tp} we obtain the objective function as
$$\sum_{(j,l)\in\set{J}_{\hat{i}} \times \set{L}: \xi_{ij} x_{\zeta(\hat{i}), \zeta(j), l} > 0} P_{jl} \xi_{ij}$$
Hence, the claimed $\hat{\alpha}_{jl}$, $\hat{\beta}$ values are in the set of optimal solutions.

Case 2. We have
$$\sum_{v\in \set{V}}s_{v\hat{i}} < \sum_{(j, l) \in \set{J}_{\hat{i}} \times \set{L}} \xi_{\hat{i}j} x_{\zeta(\hat{i}), \zeta{j}, l}$$
Then, there exist some pair $(j,l)$ with $\xi_{\hat{i}j} x_{\zeta(\hat{i}), \zeta{j}, l} > 0$ for which the demand can not be fully satisfied, i.e., $\hat{r}_{jl} < \xi_{\hat{i}j} x_{\zeta(\hat{i}), \zeta{j}, l} $.
The complementary condition states 
$$\alpha_{jl}(\xi_{\hat{i}j} x_{\zeta(\hat{i}), \zeta{j}, l} - r_{jl})=0$$
Hence, we need $\hat{\alpha}_{jl} = 0$ for these $(j, l)$ pairs.
The corresponding dual constraints reduce to
$$\beta \geq P_{jl}$$
Hence, we set 
$$\hat{\beta} = \max_{(j, l) \in \set{J}_{\hat{i}} \times \set{L}}\left\{P_{jl}\vert r_{jl} < \xi_{ij}x_{\zeta(\hat{i}), \zeta(j), l}, \xi_{ij} x_{\zeta(\hat{i}), \zeta(j), l}>0\right\}$$ 

The remaining $\hat{\alpha}_{jl}$ are necessarily set to $\max\{0, P_{jl}-\hat{\beta}\}$ by \eqref{dual_sp_constr:tp} and nonnegativity of $\alpha_{jl}$.
Thereby, optimal solution is $\hat{\alpha}_{jl}=\max\{0, P_{jl}-\hat{\beta}\}, \forall (j, l) \in \set{J}_{\hat{i}} \times \set{L}$.
\Halmos
\end{proof}

\subsection{Distribution-specific integer optimality cuts} \label{sec:ds_cuts_integer}
\color{black}
In this section we derive integer optimality cuts for the RPWPA, i.e., \eqref{eq:rpwpa}. These cuts are applicable every time the recourse problem takes the form of a MILP. 
The only assumption required is that $Q(x,s)$ is bounded and can be computed for all $(x,s)$.
\color{black}

We start by identifying constants $U_d$ for each distribution $d\in\set{D}$ that satisfy
$$U_d \geq \max_{(x, s) \in \set{C}_d}Q(x,s)$$
and a constant $U$ that satisfies
$$\infty > U \geq \max_{d \in \set{D}} U_d - \min_{d \in \set{D}} U_d$$
These constants exist under the reasonable assumption that the problems $Q(x,s,\xi)$ are bounded with probability one for all $x$ and $s$ (i.e., the CSO cannot maximize revenue indefinitely).
Next, we define the function $l: \set{Z} \times \set{Z} \times \set{D} \rightarrow \mathcal{L}$ such that, for all $z_1,z_2\in\set{Z}$ and $d\in \set{D}$, $l(z_1,z_2,d)$ returns the pricing level $l\in\set{L}$ that must be applied to trips from zone $z_1$ to zone $z_2$ in order for $P_d$ to accurately describe the distribution of demand. 

The following result introduces a finite set $\set{O}$ of optimality cuts. 

\begin{proposition}\label{prop: IP_cut}
Let $(x^k, s^k, \phi^k)$ be the $k$-th feasible solution to \eqref{RMP} and $d_k \in \set{D}$ such that $(x^k, s^k) \in \set{C}_{d_k}$. 
Define 
\begin{itemize}
    \item $\set{X}^+_k:=\{(z_1,z_2,l)\in \set{Z} \times \set{Z} \times \set{L}\vert x_{z_1, z_2,l}^k=1 \}$
    \item $\set{X}^-_k:=\{(z_1,z_2,l)\in \set{Z} \times \set{Z} \times \set{L}\vert x_{z_1, z_2,l}^k=0 \}$ 
    \item $\set{S}^+_k:=\{(v,i)\in \set{V} \times \set{I}\vert s_{vi}^k=1 \}$
    \item $\set{S}^-_k:=\{(v,i)\in \set{V} \times \set{I}\vert s_{vi}^k=0 \}$
\end{itemize}
The set of cuts
\begin{align}
    \label{form: IP_cut}
    &\phi \leq o^k(x,s)  & k=1,\ldots, \vert \set{C}\vert 
\end{align}
implies $\phi \leq Q(x,s) $ for all $(x,s)\in \set{C}$
where 
\begin{align}
\begin{split}
    o^k(x,s)= \left(Q(x^k, s^k) - U_{d_k}\right)\left(\sum_{(z_1,z_2, l)\in \mathcal{X}^+_k}x_{z_1,z_2,l}-\sum_{(z_1,z_2,l) \in \mathcal{X}^-_k}x_{z_1,z_2,l} + \sum_{(v, i)\in \mathcal{S}^+_k}s_{vi}-\sum_{(v,i) \in \mathcal{S}^-_k}s_{vi}\right)
    \\ -\left(Q(x^k, s^k)-U_{d_k}\right)\left(\vert \set{X}^+_k \vert + \vert \set{S}^+_k \vert - 1\right)  +U_{d_k}+ U\left(2\vert \set{Z} \vert - \sum_{z_1 \in \set{Z}}\sum_{z_2 \in \set{Z}}x_{z_1,z_2, l(z_1,z_2,d_k)}\right)
\end{split} 
\end{align} 
\end{proposition}
\begin{proof}{Proof of \Cref{prop: IP_cut}}
    By contradiction, assume \eqref{RMP} includes cuts \eqref{form: IP_cut} for all $k=1,\ldots, \vert\set{C}\vert $ (i.e., for all, finitely many, $(x,s)$ solutions).
    Assume at the iteration $k=\vert\set{C}\vert+1$ we obtain a solution 
    $(x^k, s^k, \phi^k)$ such that $\phi^k >Q(x^k,s^k)$.
    Observe that the cut generated for $(x^k,s^k)$ (which is by assumption present in \eqref{RMP})
    gives
    \begin{align*}
    \phi \leq o^k(x^k,s^k)= &\left(Q(x^k, s^k) - U_{d_k}\right)\left(\vert \mathcal{X}^+_k\vert + \vert \mathcal{S}^+_k\vert\right)\\ 
    -&\left(Q(x^k, s^k)-U_{d_k}\right)\left(\vert \set{X}^+_k \vert + \vert \set{S}^+_k \vert - 1\right)  +U_{d_k}+ U\left(2\vert \set{Z} \vert - 2\vert \set{Z} \vert\right)\\
    =&\left(Q(x^k, s^k)-U_{d_k}\right)+U_{d_k}\\
    =&Q(x^k, s^k)
    \end{align*}
    This contradicts the presence of all $\vert\set{C}\vert$ cuts in \eqref{RMP} and completes the proof.\Halmos
\end{proof}

Hence, upon finding a solution $(x^k, s^k, \phi^k)$ for which $\phi^k > Q(x^k, s^k)$, we can generate cut $o^k$ as in \eqref{form: IP_cut}. For $(x,s)=(x^k, s^k)$ this cut reduces to $\phi \leq Q(x^k, s^k)$ as shown in the proof and hereby \textcolor{black}{cuts off the solution $(x^k, s^k, \phi^k)$.} This cut is, however, redundant, and thus safe, for solutions satisfying $\phi^k \leq Q(x^k, s^k)$ as shown next.

\begin{remark}\label{rem:ip:safe}
Cuts \eqref{form: IP_cut} do not cut off solutions $(x^n,s^n)\neq (x^k,s^k)$.
    
\end{remark}
\begin{proof}{Proof of \Cref{rem:ip:safe}}
We assess two separate cases, namely when $d_n\neq d_k$ and when $d_n=d_k$.

Assume $d_n\neq d_k$. Then the last term of $o^k$ reduces to 
$$U\left(2\vert \set{Z} \vert - \sum_{z_1 \in \set{Z}}\sum_{z_2 \in \set{Z}}x^n_{z_1,z_2, l(z_1,z_2,d_k)}\right) \geq U\left(2\vert \set{Z} \vert - (2\vert \set{Z} \vert-1)\right)= U$$
Furthermore, 
$$\left(\sum_{(z_1,z_2, l)\in \mathcal{X}^+_k}x^n_{z_1,z_2,l}-\sum_{(z_1,z_2,l) \in \mathcal{X}^-_k}x^n_{z_1,z_2,l} + \sum_{(v, i)\in \mathcal{S}^+_k}s^n_{vi}-\sum_{(v,i) \in \mathcal{S}^-_k}s^n_{vi}\right)\leq \left(\vert \set{X}^+_k \vert + \vert \set{S}^+_k \vert - 1\right)$$
so that 
$$\phi \leq o^k(x^n,s^n)= K +U_{d_k}+ U $$
with $K > 0$ since $(Q(x^k,s^k)-U_{d_k})<0$, and is therefore valid at $(x^n,s^n)$.

When $d_n=d_k$ the last term is zero and the cut reduces to 
$$\phi \leq o^k(x^n,s^n)= K +U_{d_k} $$
which is, again, valid under distribution $d_k=d_n$.\Halmos
\end{proof}
From \Cref{rem:ip:safe} we see that the cuts do not cut off solutions other than that for which they were generated. In addition, when a solution enforces the same distribution as the one enforced by a given $x^k$, the cut enforces an upper bound greater than the distribution-specific upper-bound $U_{d_k}$.

The computational cost of generating cuts \eqref{form: IP_cut} is essentially that of computing 
$$Q(x^k, s^k) =  \sum_{n=1}^{N_{d_k}}  \pi^{nd_k} Q(x^k, s^k, \xi^{nd_k})$$ for each $(x^k,s^k)$ returned by \eqref{RMP}. 

\subsection{Valid inequalities}
\label{sec: VIs}
We devise a set of valid inequalities that provide non-trivial upper bounds on $\phi$ within each and across different subsets $\set{C}_d$. 

\color{black}
The first set of valid inequalities is distribution-specific. It is constructed by assuming that the operator has unlimited supply and is able to fulfill all materialized demand. The valid inequality is as follows \color{black}
\begin{align}
    \label{VI1} \tag{VI-1}
    \phi & \leq  \sum_{n=1}^{N_d} \pi^{nd} \sum_{(i,j) \in \set{A}} \left(\sum_{l \in \set{L}} x_{\zeta(i),\zeta(j), l} P_{ijl} \right)  \xi^{nd}_{ij} + U(2|\set{Z}| - \sum_{z_1 \in \set{Z}}\sum_{z_2 \in \set{Z}} x_{z_1,z_2, l(z_1,z_2,d)}), &\forall d \in \set{D}
\end{align}
The first term computes the expected revenue under distribution $d$ assuming demand is fully satisfied at the price determined by $x$ decisions. 
The last term makes the cut redundant for distributions other than $d$. \color{black} Valid inequality \eqref{VI1} applies to both RPWoPA and RPWPA formulations. \color{black}

The second set of valid inequalities is also distribution-specific. It is constructed by assuming that the operator receives unlimited demand and is able to entirely allocate supply. For the RPWoPA the valid inequality is as follows
\begin{align}
\label{VI2_RPWoPA} \tag{VI-2-RPWoPA}
\phi &\leq \sum_{n=1}^{N_d} \pi^{nd} \sum_{i\in\set{I}}\max_{j:(i,j) \in \set{A}}\{ P_{i,j,l(\zeta(i),\zeta(j), d)}\} \sum_{v \in \set{V}}s_{vi}  + U(2|\set{Z}| - \sum_{z_1 \in \set{Z}}\sum_{z_2 \in \set{Z}} x_{z_1,z_2, l(z_1,z_2,d)}), &\forall d \in \set{D}
\end{align}  
Here, the first term on the right-hand side computes the expected revenue obtained under distribution $d$ assuming all vehicles at a given station $i$ are assigned to the destination for which the largest price is asked for. For the RPWPA the valid inequality is as follows
\begin{align}
\label{VI2_RPWPA} \tag{VI-2-RPWPA}
\phi &\leq \sum_{n=1}^{N_d} \pi^{nd} \sum_{(i,j) \in \set{A}} P_{i,j,l(\zeta(i),\zeta(j), d)} (\theta_{ij}^{nd}\sum_{v \in \set{V}}s_{vi} + 1) + U(2|\set{Z}| - \sum_{z_1 \in \set{Z}}\sum_{z_2 \in \set{Z}} x_{z_1,z_2, l(z_1,z_2,d)}), &\forall d \in \set{D}
\end{align}  
In this case, the first term on the right-hand side computes the expected revenue obtained under distribution $d$ assuming all vehicles are rented out.

The third set of valid inequalities is distribution-independent and applies to both the RPWoPA and RPWPA. It is constructed by exploiting concavity of the recourse problem (or its LP relaxation in the RPWPA case). 
Assume, for now, we use the RPWoPA formulation. Observe $Q(x,s,\xi)$ is concave in $\xi$ (see e.g., \cite{birge2011introduction}). 
Given a pair $(x,s)$, using Jensen's inequality we can state that
\begin{align}
    \mathbb{E}_{P_x}\left[Q(x, s, \boldsymbol{\xi})\right] & \leq Q(x, s, \mathbb{E}_{P_x}[\boldsymbol{\xi}]) \nonumber\\
    & \leq \max_{d \in \set{D}} Q(x, s, \mathbb{E}_{P_d}[\boldsymbol{\xi}])\nonumber \\
    & \leq \max_{d\in \set{D}} \sum_{(i,j) \in \set{A}} \sum_{l \in \set{L}}P_{ijl} x_{\zeta(i),\zeta(j), l} \mathbb{E}_{P_d}[\xi_{ij}]\label{eq:jensens:eq3}\\
    & = \sum_{(i,j) \in \set{A}} \sum_{l \in \set{L}} P_{ijl} x_{\zeta(i),\zeta(j),l} \max_{d \in \set{D}}\mathbb{E}_{P_d}[\xi_{ij}]\nonumber
\end{align}
where in \eqref{eq:jensens:eq3} the argument of the maximization provides an upper bound to $Q(x,s,\mathbb{E}_{P_d}[\xi_{ij}])$, as in \eqref{VI1}.
Therefore, in the RPWoPA case we can use valid inequality 
\begin{align}
\label{VI3} \tag{VI-3}
    \phi \leq \sum_{(i,j) \in \set{A}} \sum_{l \in \set{L}} P_{ijl} x_{\zeta(i),\zeta(j),l}  \max_{d \in \set{D}}\mathbb{E}_{P_d}[\xi_{ij}]
\end{align}

When using the RPWPA formulation, $Q(x,s,\xi)$ forms an ILP and concavity does not necessarily hold. However, concavity holds for the linear programming relaxation of the problem.
This entails that we can replace $Q(x,s,\xi)$ by its LP relaxation $Q^{LP}(x,s,\xi)$ in the derivation above. Hence, \eqref{VI3} applies under the RPWPA formulation as well. 

\color{black}
In addition, for the RPWPA problem, we can support cuts \eqref{form: IP_cut} with duality-based cuts. We do that by generating duality-based cuts from the LP relaxation of problem \eqref{form: rpwpa_i}. 
In particular, at a certain iteration $t$, let $(\rho^{A,tn}_{ijl}, \rho^{B,tn}_{i}, \rho^{C, tn}_{ijl})$ be the optimal dual solution corresponding to constraints \eqref{Constr: a} - \eqref{Constr: c} for each station $i\in\set{I}$ and each scenario $n=1,\ldots,N_{d_t}$ under distribution $d_t$. The corresponding duality-based cut is constructed as
\begin{align}
    \label{form: LP_cuts_dual_initial}
    \begin{split}
        \phi \leq \sum_{i\in\set{I}} \sum_{n=1}^{N_{d_t}} \pi^{nd_t}\bigg\{\sum_{j \in \set{J}_i} \sum_{l \in \set{L}} \bigg [\rho_{ijl}^{A,tn}\xi_{ij}^{nd_t}x_{\zeta(i), \zeta(j), l} + \rho_{i}^{B,tn}\sum_{v\in\set{V}}s_{vi} + \rho^{C,tn}_{ijl} (\theta_{ij} \sum_{v\in\set{V}}s_{vi}+1) \bigg] \bigg \} \\+ U(2|\set{Z}| - \sum_{z_1 \in \set{Z}} \sum_{z_2 \in \set{Z}} x_{z_1,z_2, l(z,d_t)}) 
    \end{split}
\end{align}
\color{black}
These cuts provide upper bounds to the true value of $\sum_{n=1}^{N_{d_t}} \pi^{nd_t}Q(x,s,\xi^{nd_t})$ for $(x,s)\in\set{C}_{d_t}$, and are redundant for $(x,s)\notin\set{C}_{d_t}$. \color{black} 

Finally, we show that the duality-based cuts \eqref{form: LP_cuts_tp_dual} developed for the RPWoPA formulation can be re-used as valid inequalities for the RPWPA formulation. 
\begin{proposition}
\label{prop:VI_validity}
   Optimality cuts \eqref{form: LP_cuts_tp_dual} are valid inequalities when using the RPWPA.
\end{proposition}
\begin{proof}{Proof of \Cref{prop:VI_validity}}
Let $(x^t, s^t)$ be the solution to \eqref{RMP} and $\xi^{nd_t}$ be the $n$-th realization of $\boldsymbol{\xi}$ under distribution $d_t$. 
Denote $Q^{RPWoPA}(x,s,\xi)$ the value of $Q(x,s,\xi)$ when using the RPWoPA and $Q^{RPWoPA}_{LP}(x,s,\xi)$ the value of its LP relaxation. Similarly, denote $Q^{RPWPA}(x,s,\xi)$ the value of $Q(x,s,\xi)$ when using the RPWPA and $Q^{RPWPA}_{LP}(x,s,\xi)$ the value of its LP relaxation. Then the following chain of inequalities holds for all $n=1\ldots,N_{d_t}$, that is
$$Q^{RPWPA}(x^t,s^t,\xi^{nd_t})\leq Q^{RPWPA}_{LP}(x^t,s^t,\xi^{nd_t})\leq Q^{RPWoPA}_{LP}(x^t,s^t,\xi^{nd_t})=Q^{RPWoPA}(x^t,s^t,\xi^{nd_t})$$
where the first inequality holds because $Q^{RPWPA}_{LP}(x^t,s^t,\xi^{nd_t})$ is the LP relaxation of $Q^{RPWPA}(x^t,s^t,\xi^{nd_t})$, the second inequality holds because $Q^{RPWPA}_{LP}(x,s,\xi)$ is obtained by adding a set of constraints to $Q^{RPWoPA}_{LP}(x,s,\xi)$, and the equality holds by \Cref{prop:rpwopa_tu}. Furthermore, the inequalities are preserved when taking the expectation. 

Therefore, we can write
\begin{align*}
    &\sum_{n=1}^{N_{d_t}}\pi^{nd_t}Q^{RPWPA}(x^t,s^t,\xi^{nd_t}) 
    \leq \sum_{n=1}^{N_{d_t}}\pi^{nd_t}Q^{RPWPA}_{LP}(x^t,s^t,\xi^{nd_t})\\
    \leq& \sum_{n=1}^{N_{d_t}}\pi^{nd_t}Q^{RPWoPA}_{LP}(x^t,s^t,\xi^{nd_t}) \stackrel{\texttt{Strong duality}}{=}\sum_{n=1}^{N_{d_t}} \pi^{nd_t} \sum_{i \in \set{I}} \left (\sum_{(j,l)\in \set{J}_i \times \set{L}} \alpha_{ijl}^{nd_t} \xi^{nd_t}_{ij} x^t_{\zeta(i), \zeta(j), l} + \beta_{i}^{nd_t} \sum_{v \in \set{V}} s^t_{vi} \right)\\
    =&\sum_{n=1}^{N_{d_t}} \pi^{nd_t} \sum_{i \in \set{I}}  \left ( \sum_{(j,l)\in \set{J}_i \times \set{L}} \alpha_{ijl}^{nd_t} \xi^{nd_t}_{ij} x^t_{\zeta(i), \zeta(j), l} + \beta_{i}^{nd_t}  \sum_{v \in \set{V}} s^t_{vi} \right) + U(2|\set{Z}| - \sum_{z_1 \in \set{Z}} \sum_{z_2 \in \set{Z}} x^t_{z_1, z_2, l(z_1,z_2,d_t)}) 
\end{align*}
The last equality holds as the last term is zero for $d=d_t$. The above chain of inequalities shows that the right-hand side of cut \eqref{form: LP_cuts_tp_dual} generates a valid upper bound to the true value of  $\sum_{n=1}^{N_{d_t}}\pi^{nd_t}Q^{RPWPA}(x^t,s^t,\xi^{nd_t})$. It follows that optimality cuts \eqref{form: LP_cuts_tp_dual}, generated for the RPWoPA, can be used as valid inequalities for the RPWPA.\Halmos
\end{proof}

Note that cuts \eqref{form: LP_cuts_tp_dual} can be obtained more cheaply than \eqref{form: LP_cuts_dual_initial} as the dual solution $(\alpha^{nd_t}_{ijl}, \beta_i^{nd_t})$ can be obtained analytically (see \Cref{prop:dual_solutions}). Furthermore, cuts \eqref{form: LP_cuts_tp_dual} may provide non-trivial upper bounds in addition to the weak bounds provided by \eqref{form: IP_cut}. 

\color{black}

\Cref{tab:VI_summary} summarizes the applicability of the valid inequalities introduced.
\begin{table}[htbp!]
\caption{Applicability of the valid inequalities developed.}
    \centering
    \small
    \begin{tabular}{c|cccccc}
    \hline
        Problem specification & \eqref{VI1}  & \eqref{VI2_RPWoPA} & \eqref{VI2_RPWPA} & \eqref{VI3} & Cut \eqref{form: LP_cuts_dual_initial} & \textcolor{black}{Cut \eqref{form: LP_cuts_tp_dual}} \\
    \hline
        RPWoPA & ${\surd}$ & ${\surd} $ & & ${\surd}$ &  & \\
        RPWPA  & ${\surd}$ & & ${\surd}$ & ${\surd}$ & $\surd$ &\textcolor{black}{${\surd}$} \\
    \hline
    \end{tabular}
    \label{tab:VI_summary}
\end{table}

\section{Numerical tests on the solution algorithm} \label{sec: computational study}
In this section, we report on extensive experiments on the proposed solution algorithm. We run such experiments on artificially-generated instances of size comparable to real-world instances of the problem. The goal of this section is to provide empirical evidence on the performance of the algorithm. 

We run experiments on both versions of the problem, i.e., with and without proportional allocation decisions, hence forming the RPWPA and RPWoPA version of the second-stage problem. 
In particular, we compare our method (i.e., \LS) against the commercial solver Gurobi 11.0.1 (henceforth \textit{solver}) solving the monolithic formulation of the two versions of the problem. The monolithic formulations can be found in \aref{app:compact_milp_form}. 
The \LSt method is implemented in Python 3.11 using Gurobi 11.0.1 callable libraries. In all tests we set a time limit of $3600$ seconds. All tests were executed on machines equipped with a $2.10$GHz Intel(R) Xeon(R) Gold 6230 40 core CPU and $240$ Gb RAM. No parallelization technique was used when implementing the \LSt method. However, the default parallel search was applied by Gurobi when solving monolithic formulations. 

In the remainder of this section, we introduce three sets of synthetic instances of different sizes in \Cref{sec: instance_gen} while in \Cref{sec: compute_perform} we present and discuss the results of the experiments.

\color{black}
\subsection{Instance generation} 
\label{sec: instance_gen}

We constructed instances based on a synthetic carsharing system. The system is supposed to operate on a rectangular region composed of a collection of $1Km \times 1Km$ units. A carsharing station is randomly located within each unit.
We considered three different sizes of the carsharing system, denoted ``SMALL", ``MEDIUM" and ``LARGE". They have $15$ (i.e., $5\times 3$), $24$ (i.e., $6\times 4$), and $35$ (i.e., $7\times 5$) units, respectively. The units were further partitioned into $3$, $4$, or $5$ pricing zones. The partitioning procedure will be illustrated later. Furthermore, each synthetic carsharing system was associated with different volumes of customers and vehicles as shown in \Cref{tab:instance_para}. 
In total, we obtained $81$ synthetic problem instances. 

\begin{table}[!htbp]
   \caption{Instance parameters on size of the carsharing system network.}
    \centering
    \small
    \begin{tabular}{ccccc}
    \hline
   \multirow{1}[4]{*}{Region size} & Station no. & Zone no. & Customer no. & Vehicle no.\\
    & $|\set{I}|$ & $|\set{Z}|$ & $K$ & $|\set{V}|$ \\
    \hline
    SMALL & 15 (5$\times$3)  & 3, 4, 5 & 20, 40, 60  &  40, 50, 60 \\
    MEDIUM & 24 (6$\times$4) & 3, 4, 5 & 40, 60, 80 & 60, 80, 100 \\
    LARGE & 35 (7$\times$5) & 3, 4, 5 & 60, 80, 100 & 80, 100, 120 \\
    \hline
    \end{tabular}
    \label{tab:instance_para}
\end{table}

To simulate a geographical distribution of customers we proceeded as follows.
We randomly assigned each unit (hence, each station) a probability of being the origin and a probability of being the destination of a trip. 
\Cref{fig:css_map} reports the origin probabilities used for all the SMALL, MEDIUM and LARGE instances. 
Then, for each customer $k$ in the instance, we randomly picked one origin and one destination unit according to the assigned probabilities. We then randomly picked the coordinates of the final origin and destination of the trip within those units. 
The initial station of shared vehicles was instead obtained by randomly assigning vehicles to stations with equal probability.

\begin{figure}[!hbtp]
    \centering
    \subfloat[SMALL region.]{
        \includegraphics[width=0.33\linewidth]{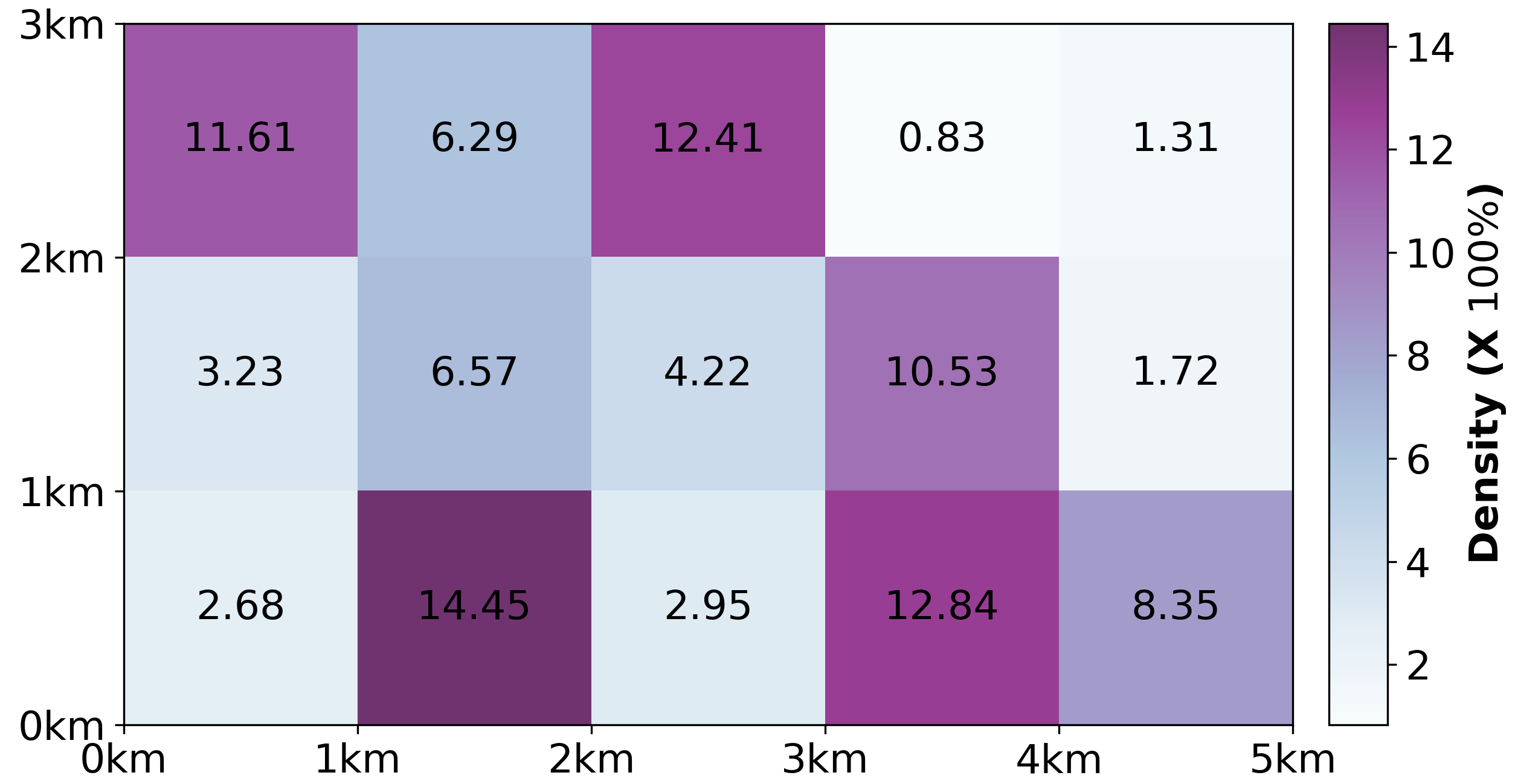}
        \label{fig:css_map15}
    }
    \subfloat[MEDIUM region.]{
        \includegraphics[width=0.30\linewidth]{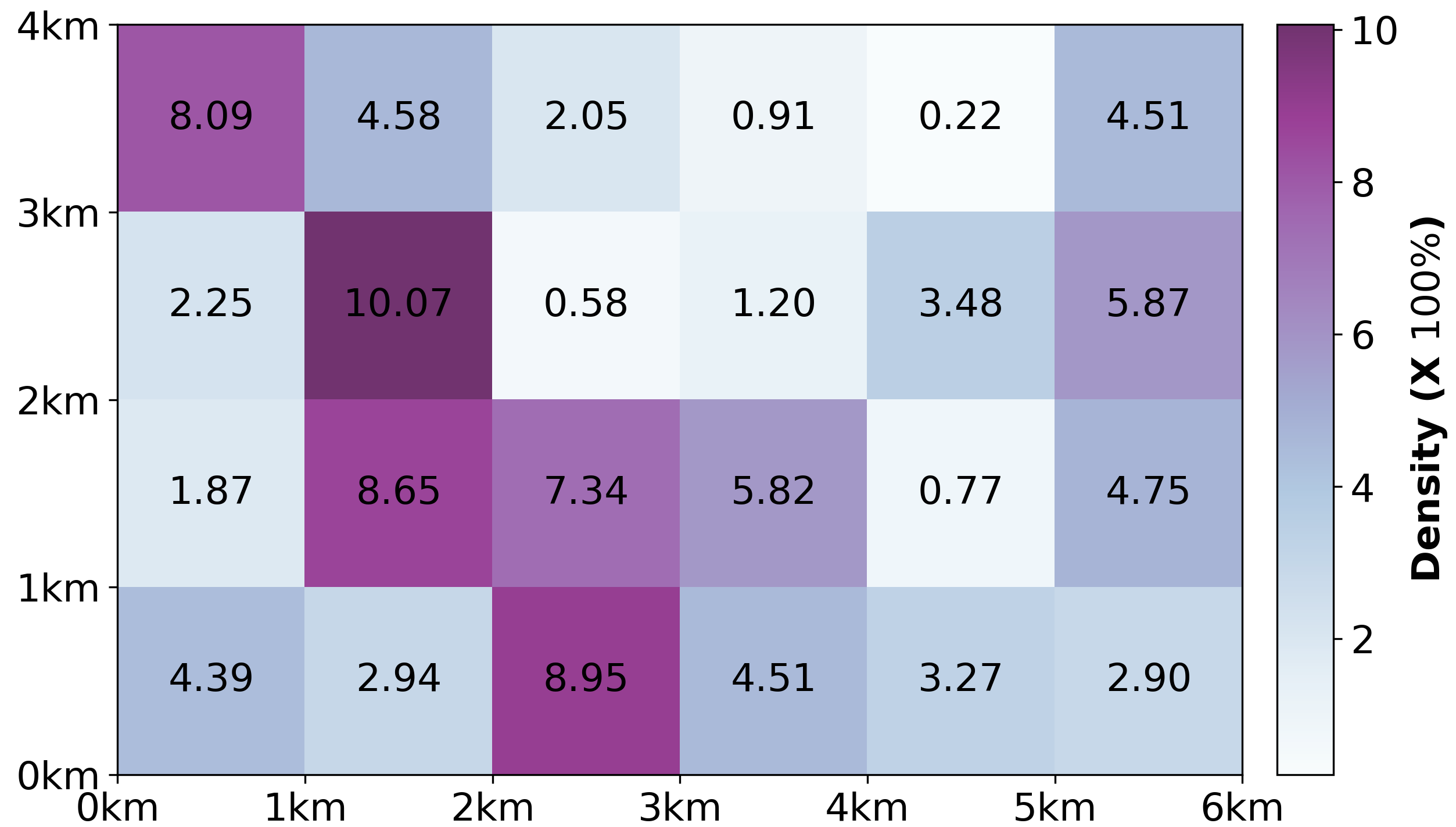}
        \label{fig:css_map24}
    }
    \subfloat[LARGE region.]{
        \includegraphics[width=0.28\linewidth]{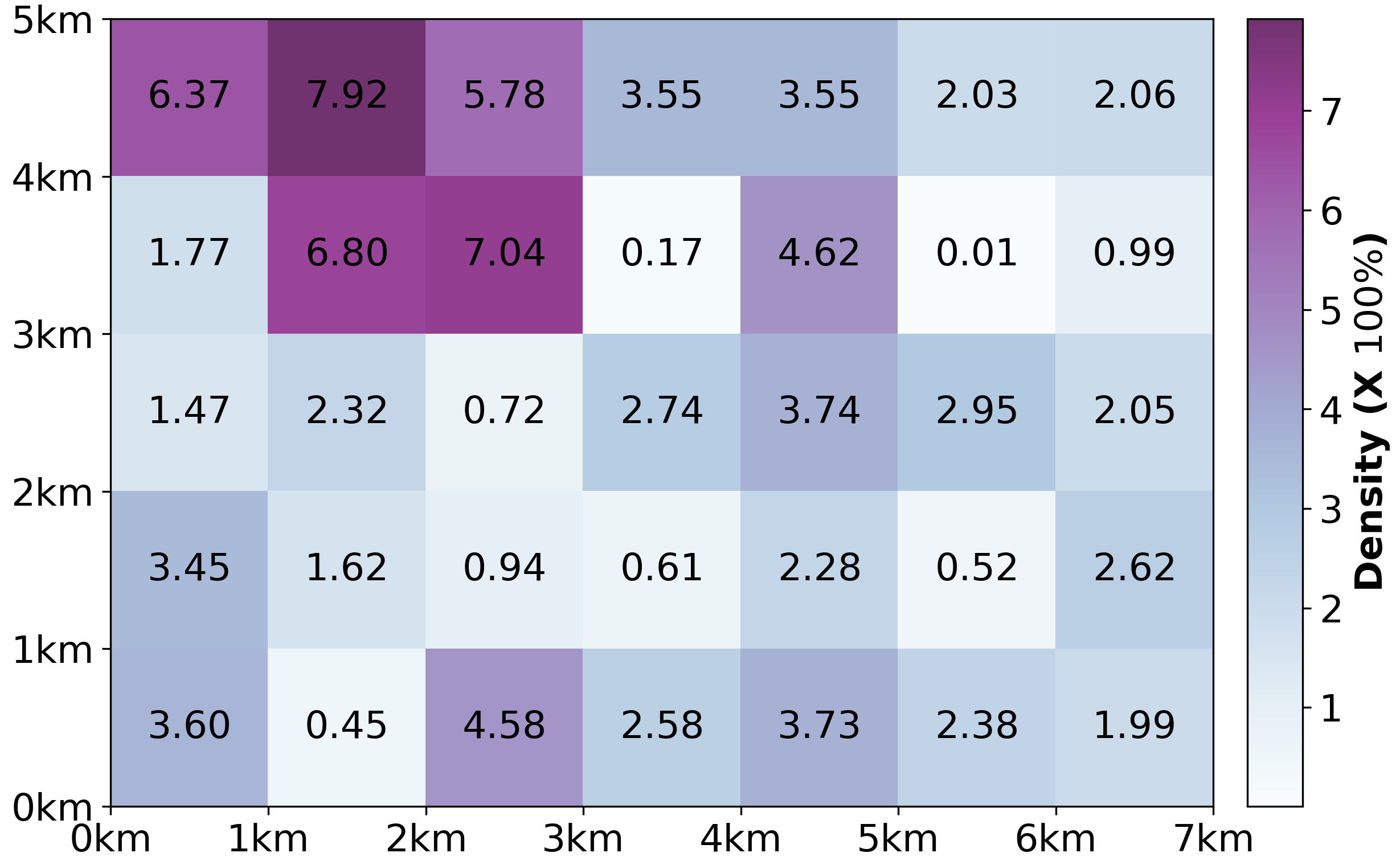}
        \label{fig:css_map35}
    }
    \vspace{0.1cm}
    \caption{Probability of each station being selected as the origin station for the three types of region.}
    \label{fig:css_maps}
\end{figure}

Pricing zones were obtained by partitioning the stations into subsets. In particular, the subdivision was performed in such a way to ensure that (i) the stations/units in each zone are geographically contiguous (i.e., the units within the zone are connected) and (ii) that potential demand is balanced (i.e., each zone covers approximately the same number of customers). 
Such partition was obtained by solving the MILP problem proposed by \citet{shirabe2009districting}. The formulation of this problem is reported in \aref{app:model_zone}. \Cref{fig:zone_division} reports the partition of the regions into four zones as an example.
\begin{figure}[t]
    \centering
    \subfloat[SMALL region with 4 zones.]{
        \includegraphics[width=0.33\linewidth]{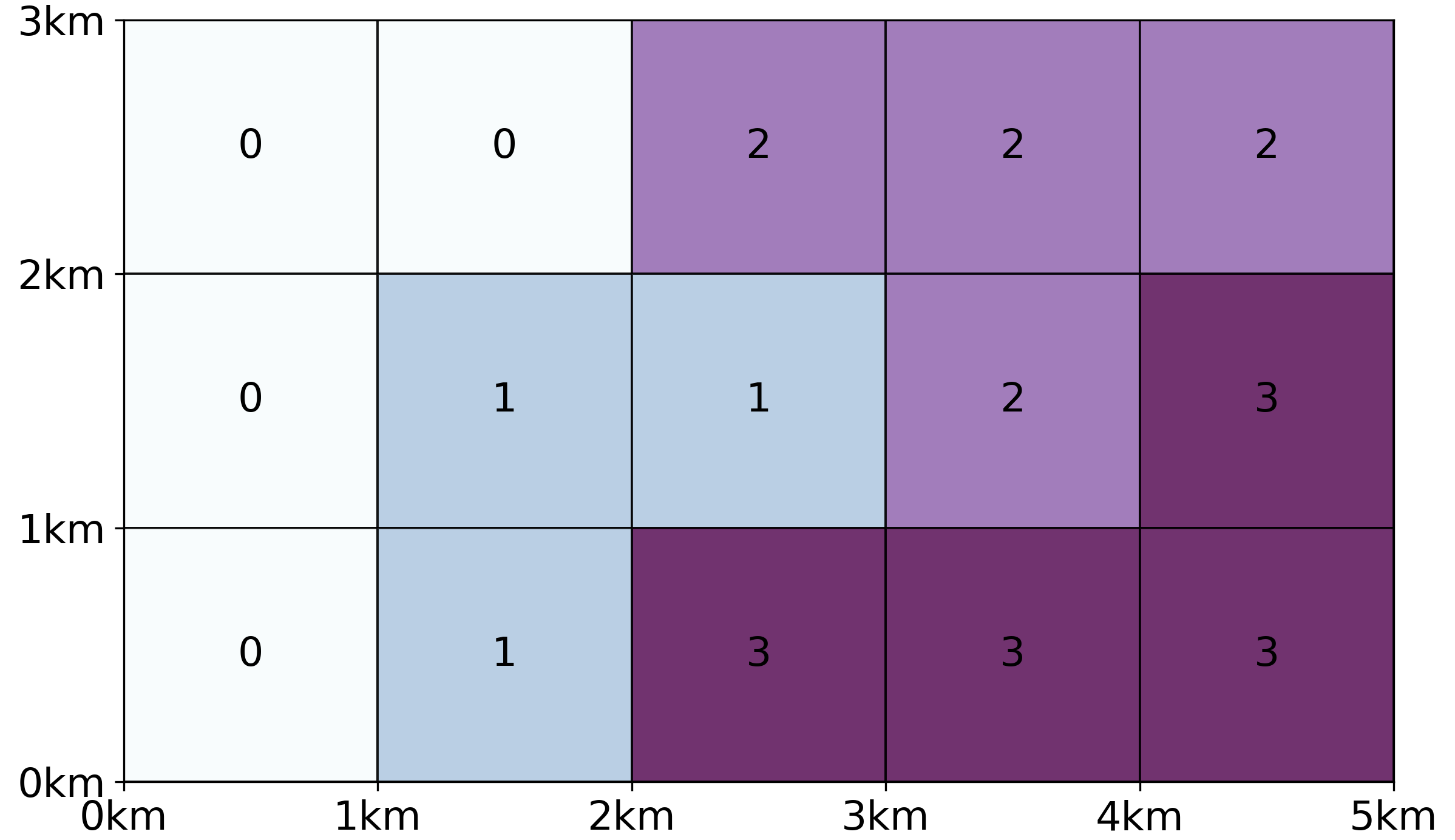}
        \label{fig:zone15}
    }
    \subfloat[MEDIUM region with 4 zones.]{
        \includegraphics[width=0.30\linewidth]{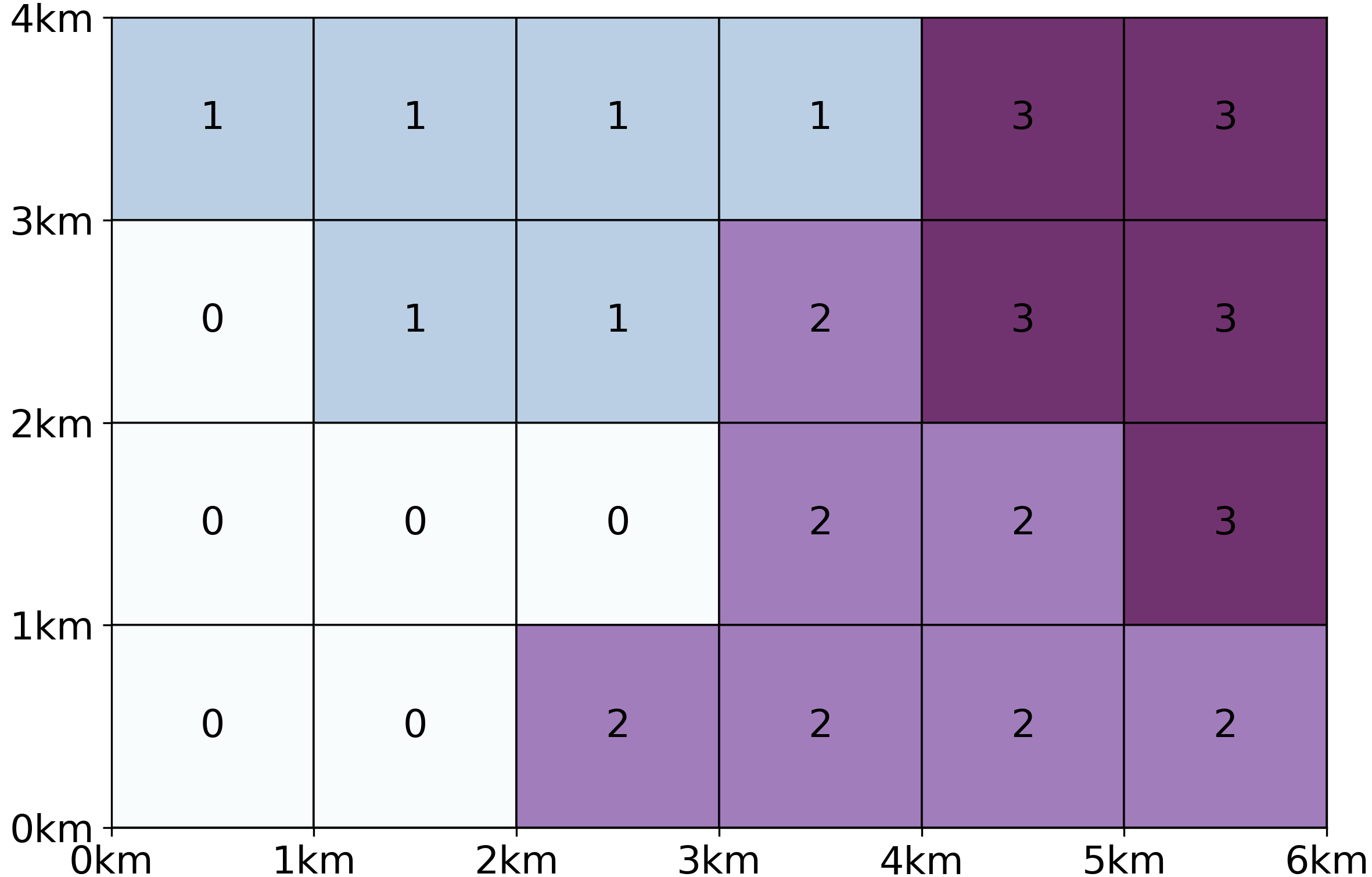}
        \label{fig:zone24}
    }
    \subfloat[LARGE region with 4 zones.]{
        \includegraphics[width=0.28\linewidth]{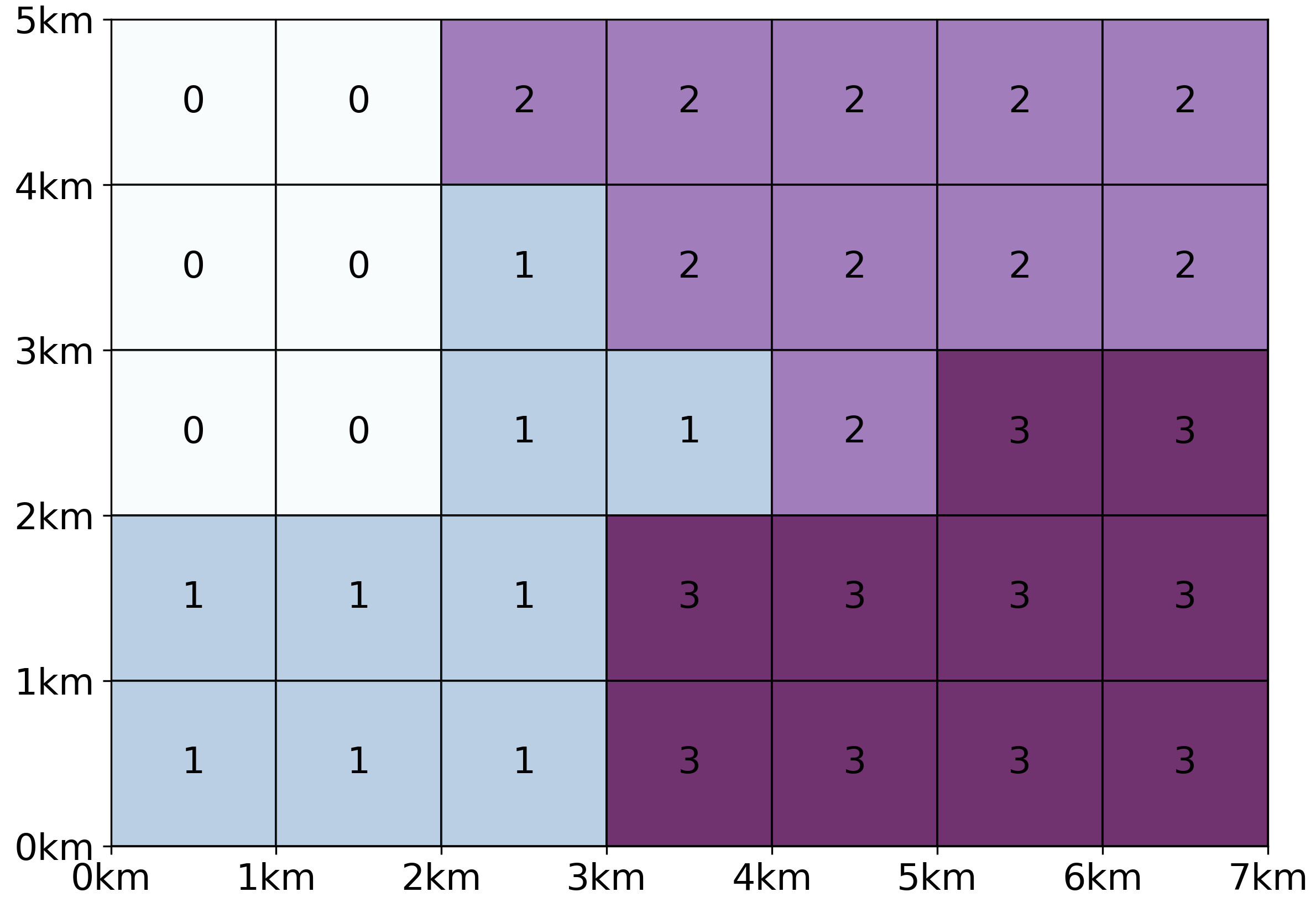}
        \label{fig:zone35}
    }
    \vspace{0.1cm}
    \caption{Partition of the service regions into contiguous pricing zones. The numbers within the units identify the zone number.}
    \label{fig:zone_division}
\end{figure}

We assumed a carsharing price made of a pick-up fee depending on the origin of the rental and a per-minute price common to all rentals. 
Decisions $x$ determine the pick-up fee level. In particular, the set $\set{L}$ contains the possible levels $P^l$, $l\in\set{L}$ the pick-up fee can take. Hence, decision $x_{z_1,z_2,l}=1$ enforces pick-up fee $P^l$ on rentals from zone $z_1$ to zone $z_2$. 
Since in our experiments the pick-up fee depends only on the origin of the trip, we enforced $x_{z_1,z_2,l}=x_{z_1,z_3,l}$ for all $z_1$ and $l$. This is equivalent to replacing variables $x_{z_1,z_2,l}$ by variables $x_{zl}$. 
Let $T_{ij}$ denote the travel time between stations $i$ and $j$ by means of a shared vehicle, and $P^{CS}$ the per-minute fee. The price of a carsharing trip between stations $i$ and $j$ at level $l$ is given by 
\begin{align}
\label{eq:price}
P_{ijl} = P^{CS} T_{ij} + P^l
\end{align}
In particular, in our experiments the variable pick-up fees could take values Euro $0, 1,2,3,4$, while the per-minute price was set to Euro $0.3$. 

The relocation cost comprises the expenses of staff salaries and fuel costs. 
Let $C^S$ denote the per-minute salary cost, $C^F$ the per-KM fuel cost, $i(v)\in \set{I}$ the initial location of each vehicle $v$ and $d_{ij}$ the distance between stations $i$ and $j$. The relocation cost is thus set as
$$C_{vi} = C^S T_{i(v), j} + C^F d_{i(v), j}, ~~\forall v \in \set{V}, j\in\set{I}$$
In particular we set $C^F=0.13$ Euro in line with current fuel prices in Denmark and $C^S=0.2$ Euro. 

Next, we describe how the decision-dependent demand distributions can be derived. In particular, we show how we obtained the probability of a customer $k$ choosing carsharing as a transport mode, given a price decision $x$ (see \Cref{sec:problem_describe} and, in particular, \eqref{eq:mu-omega}).
\color{black}We emphasize that, the choice model adopted is not meant to be the most representative model of actual human decisions, but to serve as a good proxy for that. We expect non-trivial generalizations of this model can be accommodated by our framework without compromising its applicability. 

We assumed the urban transport system offers a set of alternative transport modes $\set{M}=\{\text{walk}, \text{bike}, \text{public transit}\}$.
Given the options available to each customer, we modeled the probability of a customer choosing carsharing by means of a mixed-logit model, see, e.g., \cite{becker2017modeling, carrone2020understanding, rossetti2023commuter}.
In particular, for each customer $k$, we set a utility function $u_{k}(x,m)$ of the pricing decision and of the transport mode $m\in \set{M}\cup\{\text{cs}\}$ where $cs$ stands for carsharing. The utility function is based on the model calibrated by \citet{becker2017modeling}. The full specification of the utility function is provided in \aref{app:utility function}. In particular, for carsharing services the utility function reduces to a customer-dependent affine function of pricing decisions, while for other transport modes it reduces to a customer-dependent constant.  
Given such utility functions, the probability of customer $k$ choosing carsharing given a decision $x$ is given by \eqref{eq:choice_prob}.
\begin{align}
    P(\omega_k=1|x) & = \frac{e^{u_k(x,cs)}}{\sum_{m \in \mathcal{M}\cup\{cs\}} e^{u_k(x,m)}} \label{eq:choice_prob}
\end{align}
Based on these customer-specific probabilities, we can compute the probability mass function for the total demand on each arc, given a decision $x$, using \eqref{eq:mu-omega} and \eqref{form: prob}. Given that prices depend only on the origin station, the number of distributions becomes $|\set{D}|=|\set{L}|^{|\set{Z}|}$. 
In our experiments, we have $|\set{D}|=125$, $625$, and  $3125$ for instances with $3$, $4$, and $5$ zones, respectively. 
Since each distribution, though naturally discrete, contains a very large support, we approximated them by taking i.i.d. samples. In particular, we generated instance with $5$, $20$, $50$, and  $100$ samples. This entails that in our instances we worked with up to $312,500$ scenarios across all distributions.

\subsection{Performance of the algorithm}
\label{sec: compute_perform}
In this section, we report on the computational performance of the \LSt algorithm, compared to the solver Gurobi solving the extensive linearized formulations presented in \aref{app:compact_milp_form}.
When using the \LSt method, we apply valid inequalities \eqref{VI1}, \eqref{VI2_RPWPA}, \eqref{VI2_RPWoPA} and \eqref{VI3} statically to \eqref{RMP} as specified in \Cref{tab:VI_summary}. \color{black} For the RPWPA version we assess separately the application of cuts \eqref{form: LP_cuts_dual_initial} and \eqref{form: LP_cuts_tp_dual}. These valid inequalities are separated dynamically, i.e., upon reaching integer feasible solutions to \eqref{RMP} in the Branch-and-Bound tree that solves \eqref{RMP}. We name the \LSt method with the addition of \eqref{form: LP_cuts_dual_initial} as \LS v1 and the \LSt method with the addition of \eqref{form: LP_cuts_tp_dual} as \LS v2.
\color{black}

\color{black}
In \Cref{tab:solved_instances_rpwopa} and \Cref{tab:solved_instances_rpwpa} we report statistics on the performance of the \LSt and of the solver for all instances with RPWoPA and RPWPA, respectively. \color{black} The results are aggregated by the number of stations and zones. In particular, we report the following statistics: 
\begin{itemize}
    \item the number of instances solved to optimality within the time limit, denoted \# Solved Instances,
    \item the number of instances solved to an optimality gap smaller than $0.5\%$ within the time limit, denoted \# Solved Instances ($0.5\%$),
    \item the number of instances for which no feasible solution could be found within the time limit, denoted \# Failed Instances -- note that this statistics applies only to the solver, since the \LSt method found a feasible solution to all instances,
    \item the average solution time in seconds, denoted Avg. Solution Time,
    \item \color{black}the average optimality gap, denoted Avg. Gap. The optimality gap is computed as $\vert$\texttt{best\_bound}-\texttt{best\_objective\_value}$\vert$/$\vert$\texttt{best\_objective\_value}$\vert \times 100\%$. \color{black}
\end{itemize}
\color{black}
\begin{table}[!htbp]
  \centering
  \tiny
  \caption{Statistics on the computational performance of the \LSt method for solving the problem with RPWoPA.}
    \begin{tabular}{cccccccccc>{\color{black}}c}
    \toprule
    & & & \multicolumn{2}{c}{\# Solved Instances} & \multicolumn{2}{c}{\makecell{\# Solved Instances \\ ($0.5\%$)}} & \# Failed Instances & \multicolumn{2}{c}{\makecell{Avg. Solution Time (sec.)}} &  {Avg. Gap} \\
\cmidrule{4-11}    $\vert\set{I}\vert $     & $\vert \set{Z}\vert$     & $\vert \set{D}\vert$     & Solver & \LSt  & Solver & \LSt & Solver & Solver & \LSt &  {\LSt} \\
    \midrule
    15    & 3     & 125   & 25/36 & 36/36 & 25/36 & 36/36 & 0/36  & 1042.80 & 70.31 &  {0.00\%}\\
    24    & 3     & 125   & 15/36 & 32/36 & 15/36 & 32/36 & 14/36 & 1187.88 & 357.75 &  {0.07\%}\\
    35    & 3     & 125   & 10/36 & 16/36 & 10/36 & 30/36 & 18/36 & 1035.26 & 693.89 &  {0.85\%} \\
\cmidrule{4-11}    \multicolumn{3}{c}{\textbf{Zones =3 Avg.}} & \textbf{50/108} & \textbf{84/108} & \textbf{50/108} & \textbf{98/108} & \textbf{32/108} & \textbf{1088.65} & \textbf{373.98} &  {\textbf{0.31\%}} \\
    \midrule
    15    & 4     & 625   & 12/36 & 35/36 & 12/36 & 35/36 & 18/36 & 1227.34 & 266.04 &  {0.09\%} \\
    24    & 4     & 625   & 8/36  & 26/36 & 8/36  & 29/36 & 27/36 & 2077.73 & 795.70 &  {0.97\%} \\
    35    & 4     & 625   & 1/36  & 14/36 & 1/36  & 21/36 & 33/36 & 2895.55 & 1298.00 &  {2.35\%} \\
    \multicolumn{3}{c}{\textbf{Zones =4 Avg.}} & \textbf{21/108} & \textbf{75/108} & \textbf{21/108} & \textbf{85/108} & \textbf{78/108} & \textbf{2066.87} & \textbf{786.58} &  {\textbf{1.13\%}}\\
    \midrule
    15    & 5     & 3125  & 4/36  & 34/36 & 4/36  & 34/36 & 27/36 & 2066.91 & 480.09 &  {0.10\%} \\
    24    & 5     & 3125  & 0/36  & 17/36 & 0/36  & 21/36 & 36/36 & -     & 552.10 &  {1.92\%} \\
    35    & 5     & 3125  & 0/36  & 15/36 & 0/36  & 17/36 & 36/36 & -     & 845.59 &  {2.35\%} \\
\cmidrule{4-11}    \multicolumn{3}{c}{\textbf{Zones = 5 Avg.}} & \textbf{4/108} & \textbf{66/108} & \textbf{4/108} & \textbf{72/108} & \textbf{99/108} & \textbf{2066.91} & \textbf{625.93} &  {\textbf{1.45\%}}\\
    \midrule
    \multicolumn{3}{c}{\textbf{Avg.}} & \textbf{75/324} & \textbf{225/324} & \textbf{75/324} & \textbf{255/324} & \textbf{209/324} & \textbf{1647.64} & \textbf{595.50} &  {\textbf{0.96\%}}\\
    \bottomrule
\end{tabular}%
\label{tab:solved_instances_rpwopa}%
\end{table}%
\color{black} 

In \Cref{tab:solved_instances_rpwopa} we observe that, for the version RPWoPA, the \LSt solves significantly more instances to optimality compared to the solver. It also solves many more instances within a $0.5\%$ optimality gap. In addition, the \LSt algorithm delivers a feasible solution to all instances tested, while the solver fails to deliver a solution in approximately two thirds of the instances. The average solution time is also significantly lower for the \LSt algorithm (for the solver, the average solution time is computed over the instances for which a solution was found).  Finally, the number of instances solved to optimality or within a $0.5\%$ gap by the \LSt method decreases slightly with the size of the instances. This is expected, given the large number of distributions and scenarios involved. Nevertheless, the \LSt was able to deliver a feasible solution in all cases with average optimality gap of $0.96\%$. 

\begin{table}[htbp]
  \centering
  \tiny
  \caption{Statistics on the computational performance of \LSt algorithms (namely, \LS v1 and \LS v2) and the solver for solving the problem with RPWPA.}
    \begin{tabular}{ccccc>{\color{black}}ccc>{\color{black}}cccc>{\color{black}}c>{\color{black}}c>{\color{black}}c}
    \toprule
          &       &       & \multicolumn{3}{c}{\# Solved Instances} & \multicolumn{3}{c}{\# Solved Instance ($0.5\%$)} & \makecell{\# Failed \\ Instances}  & \multicolumn{3}{c}{Avg. Solution Time (s)} & \multicolumn{2}{c}{Avg. Gap} \\
\cmidrule{4-15}    $|\set{S}|$   &  $|\set{Z}|$     &  $|\set{D}|$   & Solver & \LS v1 & \LS v2 & Solver & \LS v1 & \LS v2 & Solver & Solver & \LS v1 & \LS v2 & \LS v1 & \LS v2 \\
    \midrule
    15    & 3     & 125   & 24/36 & 35/36 & 36/36 & 24/36 & 35/36 & 36/36 & 0/36  & 903.65 & 360.74 & 120.62 & 0.00\% & 0.00\% \\
    24    & 3     & 125   & 15/36 & 26/36 & 34/36 & 15/36 & 27/36 & 34/36 & 12/36 & 1278.23 & 971.15 & 628.03 & 1.65\% & 0.48\% \\
    35    & 3     & 125   & 10/36 & 11/36 & 30/36 & 10/36 & 16/36 & 30/36 & 18/36 & 1084.23 & 1131.92 & 883.11 & 4.68\% & 1.41\% \\
\cmidrule{4-15}    \multicolumn{3}{c}{\textbf{Zones =3 Avg.}} & \textbf{49/108} & \textbf{72/108} & \textbf{100/108} & \textbf{49/108} & \textbf{78/108} & \textbf{100/108} & \textbf{40/108} & \textbf{1088.70} & \textbf{821.27} & \textbf{543.92} & \textbf{2.11\%} & \textbf{0.63\%} \\
    \midrule
    15    & 4     & 625   & 13/36 & 32/36 & 36/36 & 13/36 & 32/36 & 36/36 & 12/36 & 1417.88 & 605.62 & 330.27 & 0.42\% & 0.00\% \\
    24    & 4     & 625   & 7/36  & 20/36 & 27/36 & 7/36  & 20/36 & 28/36 & 27/36 & 1846.68 & 919.94 & 792.38 & 2.86\% & 1.95\% \\
    35    & 4     & 625   & 1/36  & 9/36  & 23/36 & 1/36  & 9/36  & 23/36 & 34/36 & 2854.35 & 1367.42 & 991.25 & 6.68\% & 4.24\% \\
\cmidrule{4-15}    \multicolumn{3}{c}{\textbf{Zones =4 Avg.}} & \textbf{21/108} & \textbf{61/108} & \textbf{86/108} & \textbf{21/108} & \textbf{61/108} & \textbf{87/108} & \textbf{73/108} & \textbf{2039.64} & 964.33 & 704.63 & \textbf{3.32\%} & \textbf{2.06\%} \\
    \midrule
    15    & 5     & 3125  & 3/36  & 29/36 & 33/36 & 3/36  & 29/36 & 34/36 & 27/36 & 1880.71 & 803.44 & 588.55 & 0.88\% & 0.67\% \\
    24    & 5     & 3125  & 0/36  & 14/36 & 25/36 & 0/36  & 16/36 & 26/36 & 36/36 & -     & 923.18 & 1047.90 & 3.71\% & 2.95\% \\
    35    & 5     & 3125  & 0/36  & 6/36  & 16/36 & 0/36  & 6/36  & 16/36 & 36/36 & -     & 766.13 & 1142.14 & 8.06\% & 8.95\% \\
\cmidrule{4-15}    \multicolumn{3}{c}{\textbf{Zones = 5 Avg.}} & \textbf{3/108} & \textbf{49/108} & \textbf{74/108} & \textbf{3/108} & \textbf{51/108} & \textbf{74/108} & \textbf{99/108} & \textbf{1880.71} & \textbf{830.92} & \textbf{926.20} & \textbf{4.22\%} & \textbf{4.19\%} \\
    \multicolumn{3}{c}{\textbf{Avg.}} & \textbf{73/324} & \textbf{182/324} & \textbf{260/324} & \textbf{73/324} & \textbf{190/324} & \textbf{261/324} & \textbf{212/324} & \textbf{1609.39} & \textbf{840.08} & \textbf{1038.74} & \textbf{3.21\%} & \textbf{2.29\%} \\
    \bottomrule
    \end{tabular}%
  \label{tab:solved_instances_rpwpa}%
\end{table}%

In \Cref{tab:solved_instances_rpwpa}, for the version RPWPA, we also observe that both \LS v1 and \LS v2 solve significantly more instances than the solver. Moreover, we observe that the use of valid inequalities \eqref{form: LP_cuts_tp_dual} (\LS v2) improves convergence of the \LSt algorithm. That is, \LS v2 solves more instances to optimality, and significantly reduces the average optimality gap, compared to \LS v1. Recall that the difference between \LS v1 and \LS v2 is the set of valid inequalities added to \eqref{RMP}. In \LS v1 we add valid inequalities \eqref{form: LP_cuts_dual_initial} which require numerical solution of the LP relaxation to the the integer program $Q(x,s,\xi)$. In \LS v2 we use valid inequalities \eqref{form: LP_cuts_tp_dual}, which can be generated more efficiently from the analytical solution to the dual of \eqref{form: rpwopa_i}. 

Comparing \Cref{tab:solved_instances_rpwopa} and \Cref{tab:solved_instances_rpwpa}, we observe that the algorithm terminates with a larger optimality gap when solving instances with RPWPA. This behavior is expected and is due to the nature of the recourse problem, which is a linear program in the RPWoPA case and an integer program in the RPWPA case. The duality-based cuts for the RPWoPA case represent supporting hyperplanes at the epigraph of the expected revenue for each distribution. The same type of cuts provides only an upper bound in the RPWPA case, as its version of $Q(x,s,\xi)$ is non-convex .

In \aref{app:small_results}, for both versions of the problem, we report disaggregated statistics on the SMALL instances where the performance of the solver and of the \LSt method can be assessed on individual instances. Furthermore, we report the distribution of optimality gaps across all instances solved and for both \LS v1 and \LS v2. 
\color{black}

\section{A case study} \label{sec: case_study}
In this section, we investigate the impact of accounting for decision-dependent demand uncertainty in a close-to-real-world carsharing system described in \Cref{sec: case_describe}. 
\color{black} In \Cref{sec: case_solution_time} we assess the applicability of the proposed solution method to the case study in terms of computational performance. In \Cref{sec:VDDU} we assess the value of incorporating decision-dependent uncertainty. In \Cref{sec:system_perform} we compare the two different vehicle allocation policies considered using various performance metrics. Throughout this section, the model with RPWoPA is solved using \LSt with distribution-specific optimality cuts \eqref{form: LP_cuts_tp_dual}, and the model with RPWPA is solved using \LS v2. \color{black}

\subsection{Case description} \label{sec: case_describe}
We built real-world instances based on a carsharing service operating in Copenhagen, Denmark. 
The instances comprise the twenty stations illustrated in \Cref{fig:css_map}. These stations were divided into four pricing zones according to their zip code.
Potential customers were created based on the points of interest (POIs) located in an area with an $800$-meter radius centered around each station. POIs include locations such as schools, hospitals, shopping venues and museums. We assumed they represent the origins and destinations of the customers. The transport modes available in the city comprise public transport (a service offering buses, metro, and surface trains), walking, and bicycling, in addition to the focal carsharing service. 

\begin{figure}[htbp!]
    \centering
    \includegraphics[width=0.57\linewidth]{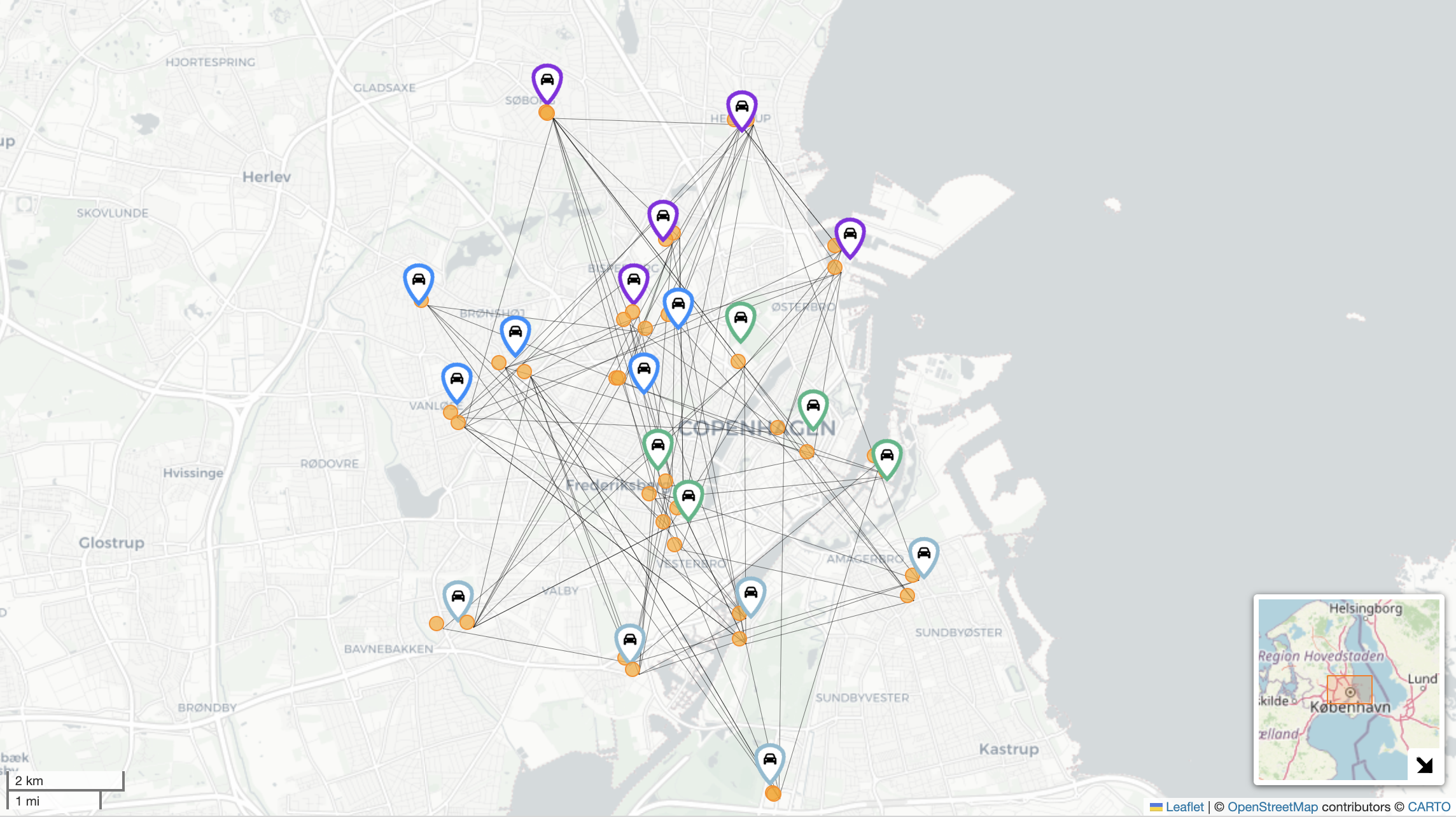}
    \caption{Geographical location of the $20$ stations considered in the case study. Colors indicate zones. Orange points represent POIs around the stations. Gray lines represent customer trips.}
    \label{fig:css_map}
\end{figure}

We generated a set of potential customer trips by randomly drawing without replacement two POIs as the origin and destination of each customer, see the gray lines in \Cref{fig:css_map}. 
For each trip and each transport mode in $\set{M} \cup \{cs\}$, we obtained trip information using Google Map APIs. This information includes the duration of bicycle rides $T_{k, b}$, access time to carsharing stations $T_{k, cs}^A$, in-vehicle time with carsharing $T_{k, cs}^V$, public transport time $T_{k,pt}^V$, and walking time $T_{k, walk}$. Each customer's travel behavior was simulated by the utility function presented in \aref{app:utility function} using real-world travel times. The values of the $\beta$ coefficients were taken from \Cref{tab: para_traveller}. The remaining parameters regarding prices and costs were set as explained in \Cref{sec: instance_gen}. 

\color{black}
Instances were generated with $100$ customers and varying fleet sizes. In particular, we used vehicle-to-customer (V/C, in the following) ratios $1:5$, $3:10$, $2:5$ and $1:2$. \color{black}
Each vehicle was initially assigned to a random station with equal probability. We approximated the demand distributions $d \in \set{D}$ with $20$ random samples per distribution.
\color{black}
Finally, we considered different customers price-sensitivity levels. We did this by increasing the initial $\beta_{cost}$ (i.e., the cost preference attribute in utility function) value by $10\%$ and $20\%$. In this way, we obtained $12$ instances in total. 
\color{black}

\subsection{Computational performance} \label{sec: case_solution_time}
The case-study instances were solved on the same platform described in \Cref{sec: computational study}. Also in this case, instances were solved to a target $10^{-4}$ optimality gap. However, in this case we did not set a time limit. This allowed us to always analyze optimal solutions.

\color{black}
\Cref{tab:compute_real} reports statistics on the performance of the \LSt method. 
The results are consistent with our observations in \Cref{sec: computational study}. 
Specifically, for the instances with RPWoPA, the \LSt method solved to optimality $11$ out of the $12$ instances within $1$ hour. The computational performance is sensitive to the V/C ratio. Lower ratios (particularly, $3:10$) determine longer solution time. 
Nevertheless, the average solution time is $2096.50$ seconds across all instances. This demonstrates sufficient scalability. 

When solving the instances with RPWPA, the performance of the \LSt method is nearly identical to that exhibited when solving the instances with RPWoPA. In this case, the method solved to optimality $10$ out the $12$ instances tested within $1$ hour. This substantiates good adaptability of the \LSt method to problems where the allocation of vehicles to customers forms a more constrained optimization problem. \color{black}

In summary, the performance of the \LSt method appears to be compatible with a real-world scenario where prices are updated, for example, every hour or less frequently. 

\begin{table}[htbp]
  \centering
  \footnotesize
  \caption{Computational performance of \LSt method under the RPWoPA and RPWPA formulations. The columns report the average solution time (Avg. Time), and the number of instances solved to provable optimality (\# Solved Instances) within $1$ hour.}
    \begin{tabular}{>{\color{black}}cc>{\color{black}}c>{\color{black}}c|>{\color{black}}c>{\color{black}}c}
    \toprule
    \multirow{2}[4]{*}{V/C ratio} & \multirow{2}[4]{*}{\# Scenarios} & \multicolumn{2}{c|}{RPWoPA} & \multicolumn{2}{c}{RPWPA} \\
\cmidrule{3-6}          &       &  Avg. Time (sec.)  & \# Solved Instances & Avg. Time (sec.) & \# Solved Instances \\
    \midrule
    1:5  & 12,520 & 2290.01 & 3/3 & 2444.66 & 3/3\\
    3:10 & 12,520 & 3083.75 & 2/3 & 4108.67 & 1/3 \\
    2:5  & 12,520 & 1906.18 & 3/3 & 2158.39 & 3/3 \\
    1:2  & 12,520 & 1106.05 & 3/3 & 1209.99 & 3/3 \\
\cmidrule{3-6}    \multicolumn{2}{c}{\textbf{Avg.}} & \textbf{2096.50} & \textbf{11/12} & \textbf{2480.43} & \textbf{10/12} \\
    \bottomrule
    \end{tabular}
  \label{tab:compute_real}%
\end{table}

\color{black} 
\subsection{Value of incorporating decision-dependent uncertainty} \label{sec:VDDU}
To evaluate the effectiveness of the proposed model incorporating decision-dependent uncertainty, we compare its solutions to those obtained when decision-dependent uncertainty is suppressed. 

In particular, we compare the proposed model to two benchmarks. The first benchmark preserves demand elasticity but suppresses the uncertainty. That is, the operator acknowledges that demand reacts to price changes. However, they assume that the demand response to prices is fully known, i.e., deterministic. 
The second benchmark preserves demand uncertainty, but suppresses the dependence of this to prices. That is, the operator accounts for the stochasticity of demand. However, it assumes that the demand distribution is fully known and independent of prices. This entails solving an ordinary stochastic program with decision-independent distribution. 

The models corresponding to the two benchmarks are constructed as follows.

\begin{enumerate}
    \item \textbf{\textit{Deterministic elastic demand benchmark}} \\ This benchmark is constructed by assuming deterministic price elasticity of demand. This is a common assumption in the carsharing pricing literature, see, e.g., \citet{jorge2015trip, wang2019pricing, xu2018electric}. In particular, we assume that demand is a deterministic function $\bar{\xi}(x):\{0,1\}^{\set{Z}\times \set{Z}\times \set{L}}\to \mathbb{N}^{\vert\set{A}\vert\times \vert\set{L}\vert}$ defined as
    $$\bar{\xi}_{ijl}(x) = \left \lceil \sum_{k|\omega_k = 1, i(k)=i, j(k)=j} P(\omega_k=1|x_{\zeta(i),\zeta(j), l}=1)- 0.5 \right \rceil, ~~\forall (i, j) \in \set{A}, \forall l \in \set{L}$$
    That is, for a given price level $l$, the expected demand associated with each arc $(i, j)$ is obtained by rounding the expected demand to the nearest integer.     
    Pricing decisions are then made using the following deterministic MILP
    \begin{align}
    \max \left \{-\sum_{v\in\set{V}} \sum_{i\in\set{I}}C_{vi} s_{vi} + Q(x, s, \bar{\xi}(x))|s.t.~\eqref{constr: price_set}-\eqref{constr: veh_relocation}, {(x,s)\in\{0,1\}^{|\set{Z}|\times|\set{Z}|\times|\set{L}|}\times \{0,1\}^{|\set{V}|\times|\set{I}|}} \right \}
    \tag{ELD-model}
    \label{form: dp}
    \end{align}
    Note that the relationship between $x$ and the probability distribution is hereby suppressed as $\bar{\xi}(x)$ is a singleton for all $x$. 
    \item \textbf{\textit{Ordinary stochastic program benchmark}} \\ 
    The second benchmark is set up by simply assuming exogenous (i.e., decision-independent) uncertain demand. In this case, the random demand follows a predefined distribution. In particular, the exogenous distribution is chosen as the distribution $\bar{d} \in \set{D}$ enforced by the average price level. 
    The corresponding ordinary stochastic program can be formulated as
    \begin{align}
    \max \left\{-\sum_{v\in\set{V}} \sum_{i\in\set{I}}C_{vi} s_{vi} + \mathbb{E}_{P_{\Bar{d}}}[Q(x, s, \xi)] | s.t.~\eqref{constr: price_set}-\eqref{constr: veh_relocation}, (x,s)\in\{0,1\}^{|\set{Z}|\times|\set{Z}|\times|\set{L}|}\times \{0,1\}^{|\set{V}|\times|\set{I}|} \right\}
    \tag{OSP-model}
    \label{form: osp}
\end{align}
\end{enumerate}

Let $(x^E, s^E)$ be an optimal solution to \eqref{form: dp}. To evaluate the performance of $(x^E, s^E)$, we need to evaluate it in the actual context where demand is uncertain and decision-dependent. To do so, we first identify the demand distribution enforced by solution $x^E$, denoted by $d^E$. The performance of the solution $(x^E, s^E)$ can be measured as the expected profit with respect to distribution $P_{d^E}$. \color{black} We denote this expected profit as the \textit{Expected Return of the Elastic Demand Solution} ($EELD$), which is computed as
$$EELD =  -\sum_{v\in\set{V}}\sum_{i\in\set{I}}C_{vi}{s}^E_{vi} + \mathbb{E}_{P_{d^E}}[Q(x^E, s^E, \xi)]$$
This value can now be compared to the expected profit delivered by problem \eqref{form:compact_rpwopa}, say $EDDU$, which accounts for decision-dependent uncertainty when making pricing and relocation decisions. 
We denote the profit increase 
$$\Delta_{EELD}=100\%\times \frac{EDDU-EELD}{EELD}$$
\color{black}
In a similar way, let $(x^{SP}, s^{SP})$ be an optimal solution to \eqref{form: osp} and let $P_{d^{SP}}$ be the distribution enforced by $x^{SP}$. The expected profit of solution $(x^{SP}, s^{SP})$ is denoted the \textit{Expected Return of the Stochastic Programming Solution} ($EOSP$), and is computed as 
$$EOSP =  -\sum_{v\in\set{V}}\sum_{i\in\set{I}}C_{vi}{s}^{SP}_{vi} + \mathbb{E}_{P_{d^{SP}}}[Q({x}^{SP}, {s}^{SP}, \xi)]$$
The profit increase obtained by accounting for decision-dependent uncertainty compared to $EOSP$ is obtained as
$$\Delta_{EOSP}=100\%\times \frac{EDDU-EOSP}{EOSP}$$

In \Cref{tab:profit_compare} we report $EDDU$, $EELS$ and $EOSP$ for the instances solved. The results demonstrate that accounting for decision-dependent demand uncertainty generates a significant increase in expected profits. In particular, accounting for decision-dependent stochastic demand increases profits by $8.39\%$, on average, compared to using decision-dependent deterministic demand (i.e., deterministic price-elastic demand) and by $8.53\%$, on average, compared to using decision-independent stochastic demand. The improvement is more pronounced when customers are more sensitive to prices. These results substantiate the effectiveness of accounting for decision-dependent demand uncertainty. 

\begin{table}[htbp]
  \centering
  \footnotesize
  \caption{Expected profits obtained by the proposed model (using RPWoRP), i.e., \eqref{form:compact_rpwopa}, and by the two benchmarks, \eqref{form: dp} and \eqref{form: osp}.}
   \begin{tabular}{cc>{\color{black}}c>{\color{black}}c>{\color{black}}c>{\color{black}}c>{\color{black}}c}
    \toprule
    Price sensitivity & V/C ratio & {$EDDU$ (Euro)} & $EELD$ (Euro) & $\Delta_{EELD}$ & $EOSP$ (Euro) & $\Delta_{EOSP}$\\
    \midrule
    \multirow{4}[2]{*}{100\%} & 1:5    & 122.71 & 112.04 & 9.52\% & 119.48 & 2.70\% \\
          & 3:10    & 162.11 & 150.00 & 8.07\% & 159.18 & 1.84\% \\
          & 2:5    & 198.15 & 191.08 & 3.70\% & 192.54 & 2.92\% \\
          & 1:2    & 216.75 & 202.82 & 6.87\% & 207.29 & 4.57\% \\
\cmidrule{3-7}    \multicolumn{2}{c}{\textbf{Avg.}} & \textbf{174.93} & \textbf{163.99} & \textbf{7.04\%} & \textbf{169.62} & \textbf{3.01\%} \\
    \midrule
    \multirow{4}[2]{*}{110\%} & 1:5    & 111.01 & 99.22 & 11.88\% & 102.51 & 8.29\% \\
          & 3:10    & 144.93 & 133.66 & 8.43\% & 134.97 & 7.38\% \\
          & 2:5    & 177.08 & 166.63 & 6.27\% & 164.82 & 7.44\% \\
          & 1:2    & 191.29 & 186.42 & 2.61\% & 169.10 & 13.12\% \\
\cmidrule{3-7}    \multicolumn{2}{c}{\textbf{Avg.}} & \textbf{156.07} & \textbf{146.48} & \textbf{7.30\%} & \textbf{142.85} & \textbf{9.06\%} \\
    \midrule
    \multirow{4}[2]{*}{120\%} & 1:5    & 101.27 & 86.37 & 17.25\% & 88.10 & 14.94\% \\
          & 3:10    & 128.07 & 111.58 & 14.78\% & 116.48 & 9.95\% \\
          & 2:5    & 156.51 & 143.10 & 9.37\% & 140.12 & 11.70\% \\
          & 1:2    & 170.15 & 166.98 & 1.90\% & 144.74 & 17.55\% \\
\cmidrule{3-7}    \multicolumn{2}{c}{\textbf{Avg.}} & \textbf{139.00} & \textbf{127.01} & \textbf{10.82\%} & \textbf{122.36} & \textbf{13.54\%} \\
    \midrule
    \multicolumn{2}{c}{\textbf{Total Avg.}} & \textbf{156.67} & \textbf{145.83} & \textbf{8.39\%} & \textbf{144.94} & \textbf{8.53\%} \\
    \bottomrule
    \end{tabular}%
  \label{tab:profit_compare}%
\end{table}

\subsection{Impact of vehicle allocation policy} \label{sec:system_perform}
In this section, we investigate the impact of different vehicle allocation policies in terms of  the following statistics.
\begin{itemize}
    \item \textit{relocation rate}, that is, the percentage of vehicles that are relocated to another station in the first decision stage. 
    \item \textit{expected service rate}, that is, the expectation of the ratio of served customers to materialized demand in the distribution enforced by the optimal decision. 
    \item \textit{average price}, that is, the average drop-off fee set by the operator. 
\end{itemize}
In \Cref{tab:avg_perform} we report the average value of these statistics for the two allocation policies.

The results show that adopting a proportional vehicle assignment policy necessarily decreases expected profit. This is due to the fact that the RPWPA is more constrained than the RPWoPA. However, the decrease in expected profits is marginal. At the same time, a proportional assignment improves expected service rates. The improvement in the expected service rate is significantly higher than the decrease in expected profits. In particular, the increase in service rate is more pronounced when the customers are more price-sensitive.
We also observe that the higher number of customers served is mainly enabled by more preventive relocations. In addition, the decrease in expected profits is mainly compensated by higher prices. 

These results demonstrate that additional control over the vehicle allocation scheme may improve levels of service with only a marginal impact on profits. However, higher levels of service appear to be driven mainly by additional relocations and higher prices.

\begin{table}[htbp]
  \centering
  \footnotesize
  \caption{Statistics aggregated by fleet size, price sensitivity and vehicle allocation policy. For the RPWPA we report the change with respect to the results obtained for the RPWoPA.}
    \begin{tabular}{cc>{\color{black}}c>{\color{black}}c>{\color{black}}c>{\color{black}}c|>{\color{black}}c>{\color{black}}c>{\color{black}}c>{\color{black}}c}
    \toprule
          &       & \multicolumn{4}{c|}{RPWoPA (Baseline)} & \multicolumn{4}{c}{RPWPA (Increase rate)} \\
\cmidrule{3-10}    \makecell{Price \\ sensitivity} & \makecell{V/C ratio} & \makecell{Exp. profit \\ (Euro)} & \makecell{Exp. \\ service rate} & \makecell{Relocation \\ rate} & \makecell{Avg. price \\ (Euro)} & \makecell{Exp. \\ profit} & \makecell{Exp. \\ service rate} & \makecell{Relocation \\ rate} & \makecell{Avg. \\price} \\
    \midrule
    \multirow{4}[2]{*}{100\%} & 1:5    & 122.71 & 49.08\% & 25.00\% & 3.75  & -0.08\% & 7.50\% & 40.00\% & 6.67\% \\
          & 3:10    & 162.11 & 68.68\% & 23.33\% & 3.75  & -0.10\% & -0.43\% & -14.29\% & 0.00\% \\
          & 2:5    & 198.15 & 80.64\% & 10.00\% & 3.50  & -0.01\% & 0.00\% & 0.00\% & 0.00\% \\
          & 1:2    & 216.75 & 87.68\% & 10.00\% & 3.25  & 0.00\% & 0.00\% & 0.00\% & 0.00\% \\
\cmidrule{2-10}    \multicolumn{2}{c}{\textbf{Avg.}} & \textbf{174.93} & \textbf{71.52\%} & \textbf{17.08\%} & \textbf{3.56} & \textbf{-0.05\%} & \textbf{1.77\%} & \textbf{6.43\%} & \textbf{1.67\%} \\
\midrule
    \multirow{4}[1]{*}{110\%} & 1:5    & 111.01 & 55.56\% & 30.00\% & 3.50  & 0.00\% & 0.00\% & 0.00\% & 0.00\% \\
          & 3:10    & 144.93 & 67.10\% & 20.00\% & 3.00  & -0.86\% & 3.02\% & 0.00\% & 8.33\% \\
          & 2:5    & 177.08 & 79.06\% & 10.00\% & 2.75  & -0.12\% & 3.18\% & 0.00\% & 9.09\% \\
          & 1:2    & 191.29 & 83.04\% & 4.00\% & 2.75  & -0.12\% & 5.25\% & 100.00\% & 9.09\% \\
\cmidrule{2-10}    \multicolumn{2}{c}{\textbf{Avg.}} & \textbf{156.07} & \textbf{71.19\%} & \textbf{16.00\%} & \textbf{3.00} & \textbf{-0.28\%} & \textbf{2.86\%} & \textbf{25.00\%} & \textbf{6.63\%} \\
    \midrule
    \multirow{4}[2]{*}{120\%} & 1:5    & 101.27 & 57.40\% & 20.00\% & 3.25  & 0.00\% & 0.00\% & 0.00\% & 0.00\% \\
          & 3:10    & 128.07 & 67.22\% & 13.33\% & 2.50  & -0.16\% & 11.55\% & 50.00\% & 20.00\% \\
          & 2:5    & 156.51 & 76.58\% & 7.50\% & 2.00  & -0.03\% & 0.00\% & 0.00\% & 0.00\% \\
          & 1:2    & 170.15 & 83.40\% & 6.00\% & 2.00  & 0.00\% & 1.57\% & 33.33\% & 0.00\% \\
\cmidrule{2-10}    \multicolumn{2}{c}{\textbf{Avg.}} & \textbf{139.00} & \textbf{71.15\%} & \textbf{11.71\%} & \textbf{2.44} & \textbf{-0.05\%} & \textbf{3.28\%} & \textbf{20.83\%} & \textbf{5.00\%} \\
    \bottomrule
    \end{tabular}%
  \label{tab:avg_perform}%
\end{table}%

\color{black}

\section{Concluding remarks} \label{sec: conclusion}
In this paper we studied the problem of deciding carsharing pricing and relocations. We advanced the available literature by proposing mathematical programming formulations that account for decision-dependent demand uncertainty. This allows decision makers to account for the impact of pricing decisions on the distribution of random demand. The problem is formulated as a two-stage stochastic program with decision-dependent uncertainty. The formulation is general and can be adapted to various specifications of the carsharing service. We illustrate two such specifications, one purely profit-driven, the second aiming to distribute vehicles proportionally to materialized demand at different arcs.

To solve the problem, we designed a tailored L-shaped method which incorporates distribution-specific cuts and devised a number of problem-specific improvements. Particularly, we show that for the two configurations of the system considered, we can solve subproblems in closed form, gaining significant efficiency. 

We performed numerical experiments on synthetic instances of size comparable to real-world cases. The computational results illustrate that the tailored L-shaped method outperforms by far an off-the-shelf solver employed to solve a monolithic linearized version of the model. Particularly, our method solved the majority of the instances to optimality within one hour, while the solver failed to deliver feasible solutions for the majority of the instances. 

Furthermore, the practical applicability of the proposed model is validated through a case study inspired by a carsharing system operated in Copenhagen, Denmark. \color{black} Our results illustrate that accounting for decision-dependent demand distributions generates significant increases in expected profits, compared to suppressing decision-dependent uncertainty. 
Furthermore, we show that additional control over vehicle allocation policies may improve service rates with only marginal effects on profits.

The results of our study open avenues for future research. 
In our problem we assume that the endogenous demand for carsharing depends only on price. In practice, demand may also be influenced by other factors, such as supply (in this case, the availability of shared vehicles). 
This extension would likely require only minor adjustments to the model provided and to the solution algorithm, in particular to the optimality cuts. However, this would require accurate models of the relationship between supply and demand. More in general, we expect that further insights can be gained by incorporating more advanced models of human preferences with respect to transportation services.
Furthermore, it is likely that decisions made for a given interval of time have an impact on decisions and demand realizations beyond the target interval of time. A multistage model would serve better to capture the sequential nature of the decision process applied in reality. However, this would likewise pose methodological challenges as it would require the solution to multistage stochastic programs with decision-dependent uncertainty. 
\color{black}
Finally, while this work was performed with a carsharing application in mind, we believe a similar modeling and solution strategy can be adopted to modeling and solving other transportation pricing and (re)location problems where the distribution of demand (or other random parameters) is influenced by decisions. Hence, we believe this work can help the transportation science community expand the richness of decision-making models and the range of tractable transportation problems solved. 
\color{black}

\ACKNOWLEDGMENT{This research is partially supported by the following research project 
\begin{itemize}
    \item ``Shared mobility: Towards sustainable urban transport'' (grant no. 1127-00176B) funded by Danmarks Frie Forskningsfond (DFF).
    \item ``ODD: Optimization Under Decision-Dependent Uncertainty'' (grant no. NNF24OC0089770) funded by the Novo Nordisk Foundation.
\end{itemize}
 This work has been performed using the Danish National Life Science Supercomputing Center, Computerome.
 The authors are also grateful to the organizers of the EURO Summer Institute 2024 for the opportunity to discuss preliminary versions of this work. 
 }

\bibliography{reference}

\newpage
\renewcommand{\theHsection}{A\arabic{section}}
\begin{APPENDICES}
\counterwithin*{equation}{section} 
\renewcommand\theequation{\thesection.\arabic{equation}} 
\setcounter{page}{1}
\color{black}
\section{A proportional vehicle allocation strategy} \label{sec: recourse problem wa}
In this section we introduce a variant of the recourse problem which models a policy that allocates vehicles to customers requests in a proportional manner, see, e.g., \cite{soppert2022differentiated}.
The allocation problem can be formulated as a MILP by extending model \eqref{form: rpwopa}.

\subsection{Problem formulation}
Let $\theta_{ij}$ denote the fraction of the demand of rentals from station $i$ directed to station $j$, that is
$$\theta_{ij}=\frac{\xi_{ij}}{\sum_{k \in \set{I}:k \not= i} \xi_{ik}}$$
In order to allocate vehicles proportionally to the demand, the recourse problem \eqref{form: rpwopa} requires the addition of the following constraints
\begin{align}
    r_{ijl} \leq \theta_{ij}\sum_{v \in \set{V}}s_{vi} + 1, \forall j\in \set{J}_i, \forall l \in \set{L}
    \label{constr: prop_assign}
\end{align}
The resulting recourse problem is therefore 
\begin{equation} Q(x, s, \xi) = \max_{r_{ijl} \in \mathbb{Z}_{\geq 0}} \left\{\sum_{(i,j)\in\set{A}}\sum_{l\in\set{L}} P_{ijl} r_{ijl}   \vert \eqref{constr_tp1},\eqref{constr_tp2},\eqref{constr: prop_assign}\label{eq:rpwpa}\right\}
\end{equation}
Constraints \eqref{constr: prop_assign} ensure that supply on arc $(i,j)$ is allocated proportionally to the demand. Observe that, since $r_{ij}$ is integer, the right-hand side $\theta_{ij}\sum_{v \in \set{V}}s_{vi} + 1$ ensures that demand can be satisfied up to $\lceil \theta_{ij}\sum_{v \in \set{V}}s_{vi} \rceil$ when the quota $\theta_{ij}\sum_{v\in\set{V}}s_{vi}$ is fractional. The right-hand side of constraints \eqref{constr: prop_assign} leaves the redundancy of allocating one more vehicle to destination $j$ when the proportional quota associated with arc $(i, j)$ is already an integer. This redundancy prevents the insufficient allocation of all available vehicles. Nevertheless, the number of allocated vehicles is still bounded by materialized demand as stated in constraints \eqref{constr_tp1}. The resulting recourse problem is a MILP problem.
\color{black}

\subsection{Closed-form solutions} \label{sec: rpwpa_primal_sol} 

\color{black} We observe that problem \eqref{eq:rpwpa} is separable by $i\in\set{I}$. That is, for each origin station $i\in\set{I}$, we define $Q^i(x,s,\xi)$ as
\begin{subequations}\label{form: rpwpa_i}
\begin{align}
Q^i(x, s, \xi) = \max & \sum_{j\in \set{J}_i} \sum_{l \in \mathcal{L}} P_{ijl} r_{ijl}   &  &\\
s.t. ~& r_{ijl} \leq \xi_{ij} x_{\zeta(i),\zeta(j),l}, &\forall j \in \set{J}_i, \forall l\in\set
{L} &~~~(\rho^A_{ijl}) \label{Constr: a}\\
&\sum_{(j,l)\in\set{J}_i\times \set{L}} r_{ijl} \leq \sum_{v \in \set{V}}s_{vi}, & &~~~(\rho^B_i) \label{Constr: b}\\
& r_{ijl} \leq \theta_{ij}\sum_{v \in \set{V}}s_{vi} + 1, &\forall j\in \set{J}_i, \forall l\in\set{L} &~~~(\rho^C_{ijl}) \label{Constr: c}\\
&r_{ijl} \in \mathbb{Z}_{\geq 0}, &\forall j \in \set{J}_i,\forall l\in\set{L} & &
\end{align}
\end{subequations}
where $\rho^A_{ijl}$, $\rho^B_i$ and $\rho^C_{ijl}$ are dual variables associated with the corresponding primal constraints. 
Problem \eqref{form: rpwpa_i} is a MILP. We now show that numerical solutions of $Q(x^k, s^k, \xi^{nd_k})$ are available in closed form. 
\color{black}

\Cref{prop: all_served} provides closed form solution to problem $Q^i(x^k,s^k,\xi^{nd_k})$ for the case when supply is larger than or equal to demand. Likewise, \Cref{prop: veh_lack case 1} addresses the case when supply is not sufficient to cover the entire demand. 
In the former case, the solution is found simply by allocating the entire supply. In the latter case demand is fulfilled in non-increasing order of the revenue.
\begin{proposition}
\label{prop: all_served}
    Let $(x,s)$ be a feasible solution to \eqref{RMP} and $\xi$ a realization of demand.
    Given an index $\hat{i}\in \set{I}$, assume that 
    $$\sum_{v\in\set{V}}s_{v\hat{i}} \geq \sum_{(\hat{i},j) \in \set{A}}\xi_{\hat{i}j}$$
    Then the optimal solution to $Q^{\hat{i}}(x, s, \xi)$, denoted $\Hat{r}:=(\Hat{r}_{jl})_{j \in \set{J}_{\hat{i}}, l\in\set{L}}$, is 
    $$\Hat{r}_{jl} = \xi_{\hat{i}j}x_{\zeta(i),\zeta(j),l} \qquad\forall j\in\set{J}_{\hat{i}}, \forall l\in\set{L}$$
\end{proposition}
\begin{proof}{Proof of \Cref{prop: all_served}}
\label{proof: all_served}
Observe that constraints \eqref{constr: prop_assign} and \eqref{constr_tp1} are equivalent to 
$$r_{jl} \leq \min\left\{\theta_{\hat{i}j} \sum_{v \in \set{V}}s_{v\hat{i}} + 1, \xi_{\hat{i}j} j\in\set{J}_{\hat{i}}, \forall l\in\set{L} \right\}$$
Then 
$$\sum_{v\in\set{V}}s_{v\hat{i}} \geq \sum_{(\hat{i},j) \in \set{A}}\xi_{\hat{i}j}$$
entails that the minimum is attained at $\xi_{\hat{i}j} x_{\zeta(i),\zeta(j),l}$.
Hence, we can set $\Hat{r}_{jl} = \xi_{\hat{i}j} x_{\zeta(i),\zeta(j),l}$ for all $j\in\set{J}_{\hat{i}}, \forall l\in\set{L}$. This solution will satisfy constraint \eqref{constr_tp2} inasmuch as 
$$\sum_{(\hat{i},j)\in\set{A}}\xi_{\hat{i}j} \leq \sum_{v\in\set{V}}s_{v\hat{i}}$$    \Halmos
\end{proof}

\begin{proposition} \label{prop: veh_lack case 1}
Let $(x, s)$ be a feasible solution to \eqref{RMP} and $\xi$ a realization of demand. Given a station $\hat{i} \in \set{I}$ assume that 
$$\sum_{v\in\set{V}}s_{v\hat{i}} < \sum_{(\hat{i},j) \in \set{A}}\xi_{\hat{i}j}$$
Let $\hat{P}_{jl}= P_{jl} x_{\zeta(\hat{i}), \zeta(j), l}, \forall j \in \set{J}_{\hat{i}}, \forall l \in \set{L}$. Recall that only one $l\in\set{L}$ leads to $x_{\zeta(\hat{i}), \zeta(j), l}=1, \forall (\hat{i}, j)\in\set{A}$, therefore the number of $\hat{P}_{jl}$ taking a value larger than $0$ can be at most $|\set{J}_i|$.  
Sort the values $\hat{P}_{jl}$ such that, for $(j,l) \in \{2, 3, \cdots, |\set{J}_{\hat{i}}|-1\} \times \set{L}$ 
$$\hat{P}_{(j,l)^{(k-1)}} \geq \hat{P}_{(j,l)^{(k)}} \geq \hat{P}_{(j,l)^{(k+1)}}$$
For a given $(j,l)^k$, denote the corresponding $j$ and $l$ as $j^{(k)}$, $l^{(k)}$, that is $(j,l)^k = (j^{(k)}, l^{(k)})$
Identify $2 \leq \hat{k} \leq |\set{J}_i|$ such that 
$$\sum_{k=1}^{\hat{k}-1} \min\{\xi_{\hat{i}, j^{(k)}}, \lfloor \theta_{\hat{i}, j^{(k)}} \sum_{v\in\set{V}}s_{v\hat{i}}+1 \rfloor \} \leq \sum_{v\in\set{V}}s_{v\hat{i}}$$ 
and 
$$\sum_{k=1}^{\hat{k}} \min\{\xi_{\hat{i}, j^{(k)}}, \lfloor \theta_{\hat{i}, j^{(k)}} \sum_{v\in\set{V}}s_{v\hat{i}}+1 \rfloor \} > \sum_{v\in\set{V}}s_{v\hat{i}}$$ 
Then, the optimal solution to $Q^{\hat{i}}(x, s, \xi)$, denoted $\Hat{r}:=(\Hat{r}_{j})_{j \in \set{J}_{\hat{i}}}$, is 
\begin{equation}
    \hat{r}_{jl} = \begin{cases}
\min \left\{\xi_{\hat{i},j}, \lfloor\theta_{\hat{i},j}\sum_{v\in\set{V}}s_{v\hat{i}}+1\rfloor \right\} & \quad \forall (j,l) \in \{(j,l)^{(1)}, (j,l)^{(2)}, \cdots, (j,l)^{(\hat{k}-1)} \}, \\[1ex]
\sum_{v\in\set{V}}s_{v\hat{i}}-\sum_{k=1}^{\hat{k}-1}\min\{\xi_{\hat{i},j^{(k)}}, \lfloor\theta_{\hat{i},j^{(k)}}\sum_{v\in\set{V}}s_{v\hat{i}}+1\rfloor\} & \quad (j, l) = (j,l)^{(\hat{k})}, \\[1ex]
0 & \quad \forall (j,l) \in \set{J}_{\hat{i}} \times \set{L} \setminus \{(j,l)^{(1)}, (j,l)^{(2)}, \cdots, (j,l)^{(\hat{k})} \}.
\end{cases}
\end{equation}
\end{proposition}

\begin{proof}{Proof of \Cref{prop: veh_lack case 1}}
The proof is omitted here as it is similar to that of the second part of \Cref{prop: veh_lack 2}.
\end{proof}

\section{Deterministic equivalent MILP formulations} \label{app:compact_milp_form}
In this appendix, we present the compact MILP formulations of problem \eqref{form: general_form} with recourse problem \eqref{form: rpwopa}. \color{black}For constructing compact formulations, we first define a continuous variable $\rho_d$ denoting the revenue obtained under distribution $d \in \set{D}$, and binary variable $\delta_d$ denoting whether distribution $d$ is enforced or not. 

The compact MILP formulation for problem \eqref{form: general_form} with RPWoPA is given as follows.
\begin{subequations}
    \label{form:compact_rpwopa}
    \begin{align}
        \max &-\sum_{v\in\set{V}}\sum_{i\in\set{I}}C_{vi}s_{vi} + \sum_{d \in \set{D}}\rho_d & \\
        s.t. ~& \rho_d \leq \sum_{n=1}^{N_d} \pi^{nd}\sum_{i \in \set{I}} \left [\sum_{(j, l) \in \set{J}_i \times \set{L}} \left(P^{CS}T_{ij} +  P^l \right) r^{nd}_{ijl}  \right] & \forall d \in \set{D} \label{constr_b1} \\
        & \rho_d \leq \delta_d U_d & \forall d \in \set{D} \label{constr_b2}\\
        & \sum_{d \in \set{D}} \delta_d = 1 &  \label{constr_b3}\\
        & \sum_{z_1 \in \set{Z}} \sum_{z_2 \in \set{Z}} x_{z_1, z_2, l(z_1,z_2, d)} \geq 2 |\set{Z}| \delta_d & \forall d \in \set{D}\\
        & \sum_{l \in \set{L}} x_{z_1, z_2, l} = 1 & \forall z_1, z_2 \in \set{Z} \\
        & \sum_{i \in \set{I}} s_{vi} = 1 & \forall v \in \set{V}\\ 
        & r_{ijl}^{nd} \leq \xi^{nd}_{ij} x_{\zeta(i), \zeta(j), l} &\forall i\in \set{I}, \forall (j, l) \in \set{J}_i \times \set{L}, \forall d \in \set{D}, n=1,\cdots,N_d \label{2nd_constr_start}\\
        &\sum_{(j, l) \in \set{J}_i \times \set{L}} r_{ijl}^{nd} \leq \sum_{v \in \set{V}}s_{vi} &\forall i \in \set{I},\forall d \in \set{D},n=1,\cdots,N_d \\
        &r_{ijl}^{nd} \leq \sum_{v \in \set{V}}s_{vi} & \forall i \in \set{I}, \forall d \in \set{D}, n=1,\cdots,N_d \label{2nd_constr_end}\\
        &r_{ijl}^{nd} \in \mathbb{Z}_{\geq 0}, &\forall i \in \set{I}, \forall (j, l) \in \set{J}_i \times \set{L}, \forall d \in \set{D}, n=1,\cdots,N_d \label{2nd_var}\\
        & x_{z_1, z_2, l} \in \{0,1\} &\forall z_1, z_2 \in \set{Z}, \forall l \in \set{L} \\
        & \rho_d \in \mathbb{R} & \forall d \in \set{D}\\
        & \delta_d \in \{0, 1\} & \forall d \in \set{D}
     \end{align}
\end{subequations}
Problem \eqref{form:compact_rpwopa} is constructed by adding constraints \eqref{constr_b1} - \eqref{constr_b3} to the original problem \eqref{form: general_form}. Specifically, constraints \eqref{constr_b1} state the expected revenue upper bound that can be obtained under each distribution $d \in \set{D}$. Constraints \eqref{constr_b2} force $\rho_d$ to take the value of $0$ when the corresponding distribution $d$ is not indicated by current price $x$, and constraint \eqref{constr_b3} ensure that only one distribution can be activated. The decision variables as well as constraints in recourse problem \eqref{form: rpwopa} are expanded by taking a distribution index $d\in\set{D}$ and a scenario index $n$ as expressed in variable domains \eqref{2nd_var} and constraints \eqref{2nd_constr_start}-\eqref{2nd_constr_end}.

\section{A flow-based MILP formulation for zone division} \label{app:model_zone}
We report the MILP model for the partition of the carsharing stations into zones. The model was proposed in \citet{shirabe2009districting}. Let $\hat{\set{G}}=(\set{I}, \set{E})$ represent the one-way carsharing system network, where $\set{I}$ represents set of nodes composed by stations, and $\set{E}$ denotes the edge set containing bi-directional arcs $(i, j)$ and $(j, i)$ for each pair of nodes. We further let $\set{I}^-(i):\{(j,i)\in\set{E}\}$ and $\set{I}^+(i):\{(i,j)\in\set{E}\}$.
Let $\theta_i$ denote the demand density of station $i \in \set{I}$ (i.e., the percentage of customers available in the unit containing station $i$). Let $\overline{\theta}$ denote the target demand density within each zone, which is calculated by $\sum_{i \in \set{I}} \theta_i /|\set{Z}|$. Let $d_{ij}$ denote the distance between station $i$ and $j$. 
We define the imbalance level of the demand distribution, denoted $\beta$, as the largest deviation of the demand density from the target density $\Bar{\theta}$.
The $\beta$ value is set to $10\%$ in our experiments. 
We introduce binary variable $y_{ij}$, which takes value $1$ if station $i \in \set{I}$ is assigned to the zone (centered at) station $j \in \set{I}$, $0$ otherwise. A set of flow variable $f_{ik}^j$ is introduced to indicate the amount of flow originating at zone $j$ that is sent across link $(i, k)$. The flow-based MILP formulation for zone division is formulated as below.
\begin{subequations}
 \begin{align}
    \min &\sum_{i \in \set{I}} \sum_{j \in \set{I}} d_{ij}^2 y_{ij} & \label{zone_divison_obj}\\
    s.t. ~ & \sum_{j \in \set{I}} y_{ij} = 1 & \forall i \in \set{I} \label{zd_c1}\\
    & \sum_{j \in \set{I}} y_{jj} = \vert \set{Z} \vert & \label{zd_c2}\\
    &(1-\beta) \Bar{\theta} y_{jj} \leq \sum_{i \in \set{I}}\theta_i y_{ij} \leq (1+\beta) \Bar{\theta} y_{jj} & \forall j \in \set{I} \label{zd_c3}\\
    &y_{ij} \leq y_{jj} & \forall i \in \set{I}, \forall j \in \set{I} \label{zd_c4}\\
    & \sum_{k \in \set{I}^-(i)} f^j_{ki} - \sum_{k \in \set{I}^+(i)} f^j_{ik} = y_{ij}, & \forall i \in \set{I}\setminus\{j\}, \forall j \in \set{I} \label{zd_c5}\\
    & \sum_{k \in \set{I}^-(i)} f^j_{ki} \leq (|\set{I}| - 1) y_{ij}, & \forall i \in \set{I}\setminus\{j\}, \forall j \in \set{I} \label{zd_c6}\\
    & \sum_{k \in \set{I}^-(j)} f^j_{kj} = 0, & \forall j \in \set{I} \label{zd_c7}\\
    & f^j_{ik} \geq 0, &\forall (i,k) \in \set{E}, \forall j \in \set{I} \label{zd_c8}\\
    & y_{ij} \in \{0,1\}, & \forall i \in \set{I}, \forall j \in \set{I} \label{zd_c9}
\end{align}
\label{form:zone_division}
\end{subequations}
In model \eqref{form:zone_division} we minimize the sum of intra-zone squared Euclidean distances \eqref{zone_divison_obj}. Constraints \eqref{zd_c1} and \eqref{zd_c2} ensure that each station is assigned to one zone and that exactly $|\set{Z}|$ zones are created, respectively. Constraints \eqref{zd_c3} make sure that the cumulative demand density of a certain zone $j$ is within the allowed imbalance level $\beta$ from the target demand density (i.e., $\overline{\theta}$). 
Constraints \eqref{zd_c4} ensure that station $i$ can be assigned to the zone centered at station $j$ only if station $j$ is designated as the center of a zone. Constraints \eqref{zd_c5} ensure that node $i$ consumes one unit flow of type $j$ if $i$ is assigned to zone $j$, and consumes $0$, otherwise. Constraints \eqref{zd_c6} ensure that node $i$ can receive flow of type $j$ only if it is assigned to zone $j$, while constraints \eqref{zd_c7} prevent flow circulation. Constraints \eqref{zd_c8} and \eqref{zd_c9} define the domain of decision variables.

The above MILP formulation is solved using Gurobi Solver 11.0.1 in our instance generation process. Even with the largest of our instances, model \eqref{form:zone_division} can be solved efficiently within seconds.

\section{Utility function used to model customer mode choices}
\label{app:utility function}
In this appendix, we provide the explicit expression of the utility function based on \citet{becker2017modeling}. 
The utility function is a function of travel time, access time (i.e., the time necessary to reach the nearest offer of a specific transport mode), and travel cost. These attributes are listed in \Cref{tab: para_traveller} together with their value. 
The utility function is specified, for all $m \in \set{M} \cup \{cs\}$ and $k=1,2,\cdots,K$ as follows
\begin{align}
\label{eq:utility_expression}
        u_{k}(x,m) = 
        \begin{cases}
        T_{k, b}\beta_{b} + \epsilon_{km}& \text{ if } m=\text{bike}\\
        T_{k, pt}^V \beta_{pt}^V +T_{k, pt}^W \beta_{pt}^W  + T_{k, pt}^A \beta_{pt}^A + C_{km} \beta_{cost} + \epsilon_{km}& \text{ if } m=\text{public transit}\\
        T_{k, walk} \beta_{walk} + \epsilon_{km}&\text{ if } m=\text{walk}\\
        T_{k,cs}^{V} \beta^V_{cs} + T^A_{k, cs} \beta^A_{cs} + \sum_{l \in \set{L}}P_{i(k),j(k),l}x_{\zeta(i(k)), \zeta(j(k)), l}\beta_{cost} + \epsilon_{km} &\text{ if } m=\text{cs}\\   
        \end{cases}   
\end{align}
where $P_{i(k),j(k),l}$ is the price of a carsharing ride between stations $i(k)$ and $j(k)$ when the price level is $l$ and is defined in \eqref{eq:price}. Furthermore, $\epsilon_{km}$ is randomly drawn from an extreme value type I distribution \citep{gumbel1960multivariate}. Observe that the utility function is affine in $x$ for $m=cs$ and constant in $x$ for $m\in\set{M}$.

\color{black}
\begin{table}[htbp]
\caption{Attributes and coefficients in the utility function. Real time values are obtained through Google Maps APIs. The coefficients are taken from \citet{becker2017modeling}.}
    \centering
    \scriptsize
    \begin{tabular}{l|c|c}
    \toprule
    Attribute &  Description & Value\\
    \hline
        $T_{k,cs}^V$ & In-vehicle travel time of customer $k$ by carsharing & Customer specific \\
        $T_{k,cs}^A$ & Access time of customer $k$ for carsharing & Customer specific\\
        $T_{k, b}$ & Travel time of customer $k$ by bike & Customer specific\\
        $T_{k, pt}^V$ &In-vehicle travel time of customer $k$ by public transport & Customer specific\\
        $T_{k, pt}^A$ &Access time of customer $k$ for public transport & A random value in $2-8$ min\\
        $T_{k, pt}^W$ &Waiting time of customer $k$ for public transport & A random value in $2-8$ min \\
        $T_{k, walk}$ &Travel time of customer $k$ by walking & Customer specific\\
        $C_{km}$ & Travel cost of customer $k$ taking mode $m \in \set{M}$ & 3.22 Euro for public transport, 0 for others.\\
    \midrule
    Coefficient &  Description & Value (Euro/h)\\
    \hline
    $\beta_{cs}^V$ & Coefficient for in-vehicle carsharing travel time & $-3.02$\\
    $\beta_{cs}^A$ & Coefficient for carsharing access time & $-0.70$ \\
    $\beta_b$ & Coefficient for biking time & $13.33$\\
    $\beta_{pt}^V$ & Coefficient for in-vehicle public transit time & $-5.05$ \\
    $\beta_{pt}^A$ & Coefficient for public transit access time & $-27.27$ \\
    $\beta_{pt}^W$ & Coefficient for public transit waiting time & $-4.63$ \\
    $\beta_{walk}$ & Coefficient for walking mode travel time & $-13.96$ \\
    $\beta_{cost}$ & Coefficient for travel costs (e.g., rental fees, ticket fees) & $-0.48$\\
    \bottomrule
    \end{tabular}
    \label{tab: para_traveller}
\end{table}

\section{Supplementary results on the performance of the algorithm}\label{app:small_results}
In this section we provide additional results on the performance of the \LSt algorithm. 

In \Cref{tab:compute_time} we report disaggregated results on the SMALL instances tested. In particular, we compare the performance of the \LSt against that of the solver. We limit our attention to problems containing only $5$ scenarios per distribution. The total number of scenarios becomes $625$, $3125$, and $15625$ in the instances with $3$, $4$, and $5$ zones, respectively. In \Cref{tab:compute_time} we report the solution time of the solver and \LSt for the instances for which an optimal solution has been found within a $10^{-4}$ optimality gap.
\begin{table}[!htbp]
  \centering
  \tiny
  \caption{Solution time (in seconds) for the solver and the \LSt on the SMALL instances with $5$ scenarios per distribution. Cells with ``-" indicate that the corresponding instance could not be solved to the desired $10^{-4}$ optimality gap within $3600$ seconds. }
    \begin{tabular}{cccccc|cc>{\color{black}}c}
    \toprule
    \multicolumn{3}{c}{Instance size} & \multirow{1}[4]{*}{\# Scenarios ($|\set{D}|\times5$)} & \multicolumn{2}{c|}{Time for problem with RPWoPA} & \multicolumn{3}{c}{Time for problem with RPWPA} \\
    \cmidrule{5-9}         
    $K$ & $|\set{V}|$ & $|\set{Z}|$ &   & Solver & \LSt  & Solver & \LS v1 & {\LS v2}\\
    \midrule
    20 & 40 & 3 & 625 $(125\times5)$ & 67.95 & 5.30 & 65.70 & 3.86 & {2.29} \\
    20 & 50 & 3 & 625 $(125\times5)$ & 69.32 & 1.22 & 73.40 & 1.62 & {1.27}  \\
    20 & 60 & 3 & 625 $(125\times5)$ & 63.27 & 0.02 & 62.23 & 0.04 & {0.08} \\
    40 & 40 & 3 & 625 $(125\times5)$ & 166.59 & 16.92 & 170.34 & 85.92 & {26.92} \\
    40 & 50 & 3 & 625 $(125\times5)$ & 135.09 & 5.79 & 129.82 & 39.76 & {12.06}\\
    40 & 60 & 3 & {625 $(125\times5)$} & 71.32 & 1.00 & 74.00 & 0.11 & {0.68} \\
    60 & 40 & 3 & {625 $(125\times5)$} & 71.40 & 25.12 & 314.89 & 238.35 & {45.13} \\
    60 & 50 & 3 & {625 $(125\times5)$} & 280.57 & 14.91 & 279.18 & 145.89 & {20.3} \\
    60 & 60 & 3 & {625 $(125\times5)$} & 287.18 & 7.45 & 291.75 & 65.69 & {8.19} \\
    \cmidrule{5-9}    
    \multicolumn{4}{c}{\textbf{3-zone Avg.}} & \textbf{156.97} & \textbf{8.64} & \textbf{162.37} & \textbf{64.58} & {\textbf{12.99}} \\
    \midrule
    20 & 40 & 4 & {3,125 $(625\times5)$} & 383.25 & 4.13 & 400.22 & 14.06 & {6.53} \\
    20 & 50 & 4 & {3,125 $(625\times5)$} & 453.77 & 2.85  & 407.52 & 7.31 & {4.56}\\
    20 & 60 & 4 & {3,125 $(625\times5)$} & 382.15 & 0.04 & 354.79 & 0.15 & {0.32}\\
    40 & 40 & 4 & {3,125 $(625\times5)$} & 955.09 & 36.14 & 888.07 & 135.24 & {55.83} \\
    40 & 50 & 4 & {3,125 $(625\times5)$} & 624.42 & 5.82 & 630.28 & 50.89 & {23.16}\\
    40 & 60 & 4 & {3,125 $(625\times5)$} & 477.98 & 2.56 & 449.81 & 8.80  & {3.29}  \\
    60 & 40 & 4 & {3,125 $(625\times5)$} & 2120.44 & 86.75 & 2152.84 & 890.92 & {108.18}\\
    60 & 50 & 4 & {3,125 $(625\times5)$} & 1669.76 & 49.78 & 2102.88 & 515.07 & {84.87} \\
    60 & 60 & 4 & {3,125 $(625\times5)$} & 1810.75 & 26.51 & 1560.01 & 221.48 & {38.13} \\
    \cmidrule{5-9}    
    \multicolumn{4}{c}{\textbf{4-zone Avg.}} & \textbf{986.40} & \textbf{23.84} & \textbf{994.05} & \textbf{204.88} &  {\textbf{36.10}} \\
    \midrule
    20 & 40 & 5 & {15,625 (3,125$\times5)$} & 2279.69 & 6.99 & 2027.68 & 29.55 & {9.07}\\
    20 & 50 & 5 & {15,625 (3,125$\times5)$} & 2047.30 & 5.41 & 1948.85 & 14.70 & {8.86} \\
    20 & 60 & 5 & {15,625 (3,125$\times5)$} & 1666.58 & 0.15 & 1665.62 & 0.52 &{1.53} \\
    40 & 40 & 5 & {15,625 (3,125$\times5)$} & -     & 36.80 & -     & 107.79 & {64.91} \\
    40 & 50 & 5 & {15,625 (3,125$\times5)$} &-     & 16.46 & -     & 47.92 & {61.52} \\
    40 & 60 & 5 & {15,625 (3,125$\times5)$} & 2274.05 & 4.18 & -     & 6.53  &{5.06} \\
    60 & 40 & 5 & {15,625 (3,125$\times5)$} & -     & 144.75 & -     & 1760.74 & {287.42} \\
    60 & 50 & 5 & {15,625 (3,125$\times5)$} & -     & 135.76 & -     & 850.45 & {205.61}\\
    60 & 60 & 5 & {15,625 (3,125$\times5)$} &-     & 37.84 & -     & 303.67 & {90.75}\\
    \cmidrule{5-9}    
    \multicolumn{4}{c}{\textbf{5-zone Avg.}} & \textbf{2066.91} & \textbf{43.15}  & \textbf{1880.71} & \textbf{346.87} & {\textbf{81.64}} \\
\midrule
 \multicolumn{4}{c}{\textbf{Avg. for all solved instances}} &\textbf{843.54} & \textbf{14.05} & \textbf{764.28} & \textbf{177.62} &{\textbf{43.57}}\\
    \bottomrule
    \end{tabular}%
  \label{tab:compute_time}%
\end{table}%

We observe that the \LSt and its variant versions solve all instances, while the solver fails to deliver optimal solutions in $6$ and $5$ instances with $3125$ distributions under RPWPA and RPWoPA formulations, respectively. 
For the instances that both the solver and the \LSt algorithm could solve to optimality we observe that the \LSt method was significantly faster. The performance gap is more pronounced when solving the problem with RPWoPA. The \LSt method solves all instances within $14.05$ seconds on average, while the solver takes $843.54$ seconds. \color{black} For the problem with RPWPA, the \LS v2 version of the method, i.e., with the addition of cuts \eqref{form: LP_cuts_tp_dual} as valid inequalities, significantly accelerates convergence. In particular, the addition of valid inequalities \eqref{form: LP_cuts_tp_dual} decrease the average solution time from $177.62$ seconds to $43.57$ seconds (see the \LS v1 and \LS v2 columns, respectively).
\color{black}
For the instances that cannot be solved to optimality by the solver (reported with ``-" in \Cref{tab:compute_time}), the average optimality gap is $11.05\%$ and $7.43\%$ for the RPWPA and RPWoPA formulations, respectively. Optimality gaps as percentages are calculated as ($\vert$\texttt{best\_bound}-\texttt{best\_objective\_value}$\vert$/$\vert$\texttt{best\_objective\_value}$\vert \times 100$). 
It is therefore evident that the performance of the \LSt algorithm is superior to that of the solver already on the smallest instances generated.

Finally, in \Cref{fig:optimality_gaps}. we report the distribution of optimality gaps for both the RPWoPA and the RPWPA across all instances. \color{black} For the version with RPWPA, we presents optimality gaps delivered by both the \LS v1 the \LS v2. Is is evident that the addition of valid inequalities \eqref{form: LP_cuts_tp_dual} (i.e., using \LS v2) reduces the optimality gap in most instance groups compared to \LS v1 (see \Cref{fig:pva_15_initial}-\Cref{fig:pva_35}). \color{black} On the SMALL instances, under both RPWoPA and RPWPA versions, the median optimality gap is always $0\%$, as illustrated in \Cref{fig:pva_15,fig:tpsp_15}. Only in a few cases it is higher than $5\%$. The median optimality gap is always lower than $5\%$ (see, \Cref{fig:tpsp_24,fig:tpsp_35}) for all the MEDIUM and LARGER instances under the RPWoPA formulation. 
The LARGE instances of the RPWPA formulation appear slightly more challenging. The median optimality gap is between $5\%$ and $15\%$, see \Cref{fig:pva_24,fig:pva_35}). This is mainly due to the fact that the cuts used with an IP problem in the second decision stage are typically weaker than those used with a LP problem in the second decision stage. 
\color{black}

\begin{figure}[!htbp]
    \centering
    \begin{subfigure}[h]{0.325\textwidth}
        \centering
        \includegraphics[width=\textwidth]{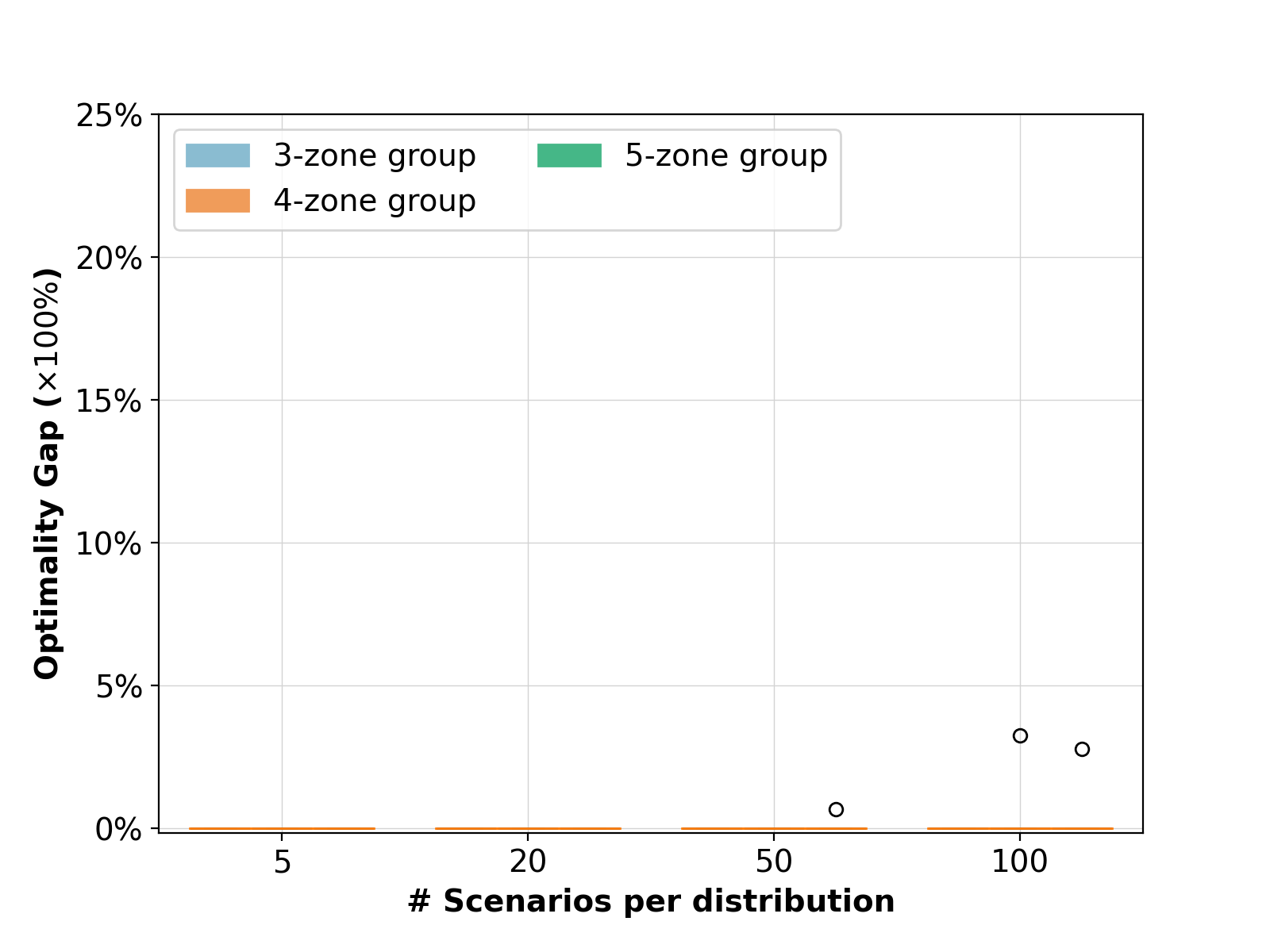}
        \caption{RPWoPA, SMALL.}
        \label{fig:tpsp_15}
    \end{subfigure}
    \hfill
    \begin{subfigure}[h]{0.325\textwidth}
        \centering
        \includegraphics[width=\textwidth]{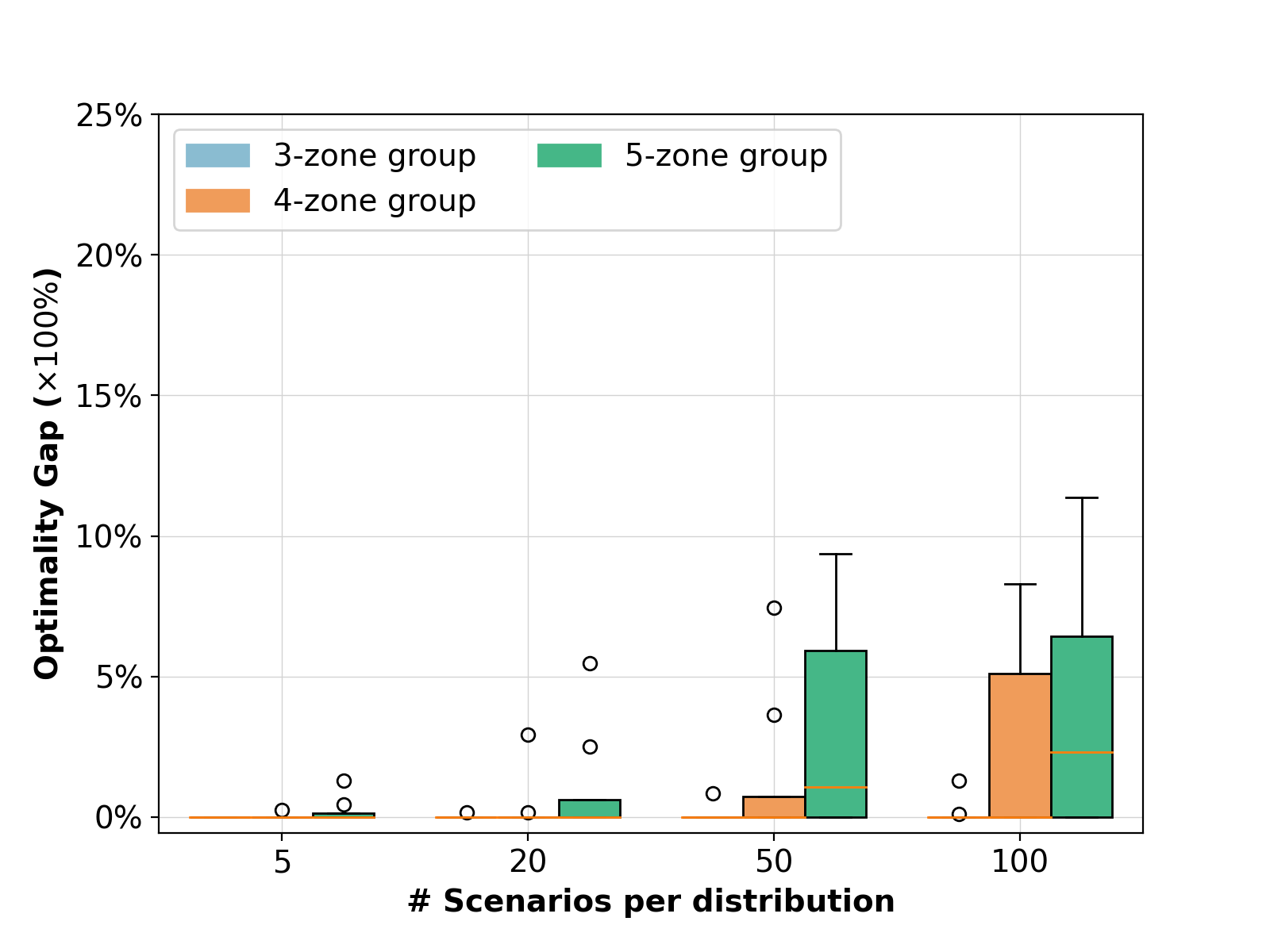}
        \caption{RPWoPA, MEDIUM.}
        \label{fig:tpsp_24}
    \end{subfigure}
    \hfill
    \begin{subfigure}[h]{0.325\textwidth}
        \centering
        \includegraphics[width=\textwidth]{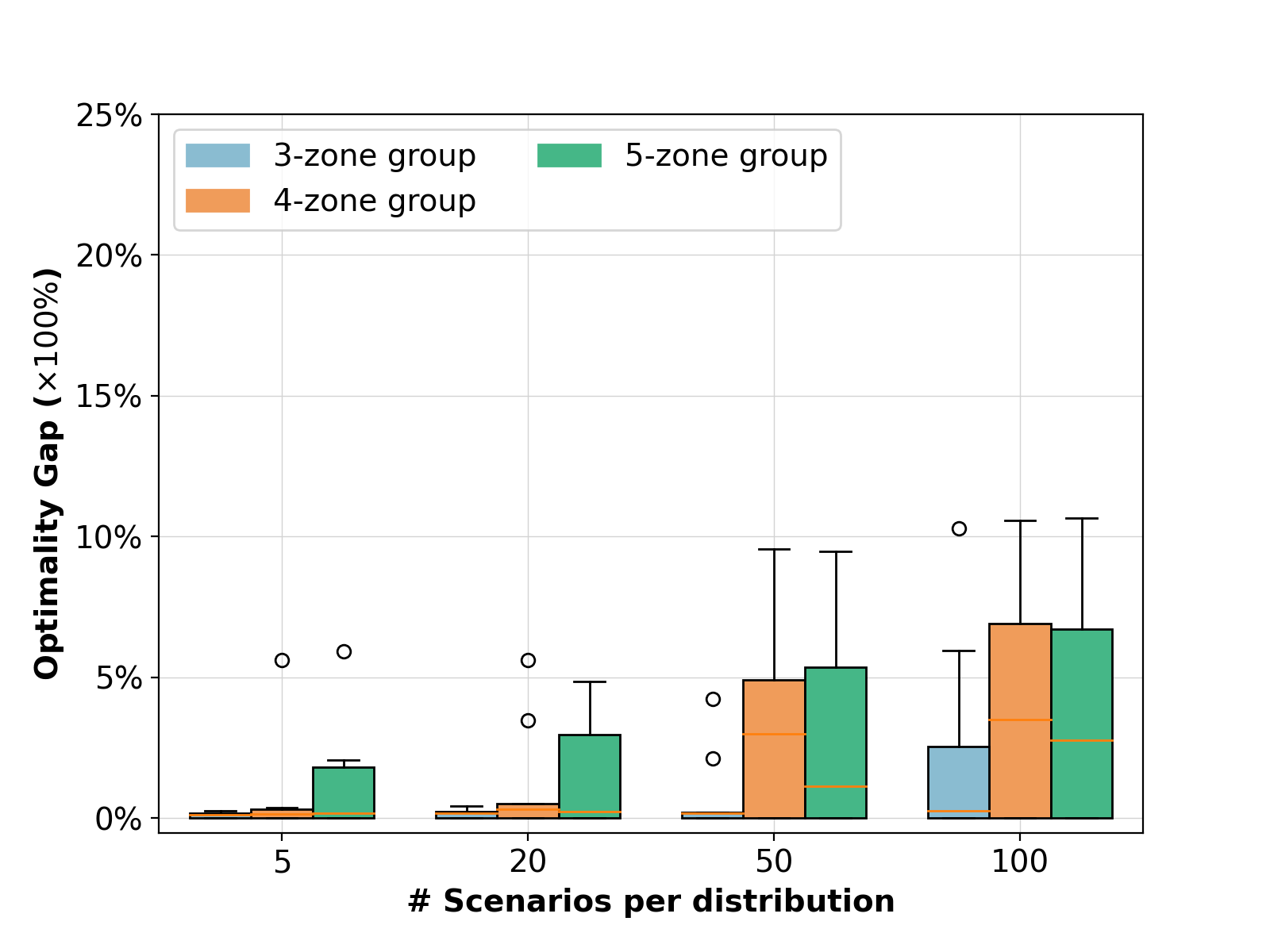}
        \caption{RPWoPA, LARGE.}
        \label{fig:tpsp_35}
    \end{subfigure}

    \begin{subfigure}[h]{0.325\textwidth}
        \centering
        \includegraphics[width=\textwidth]{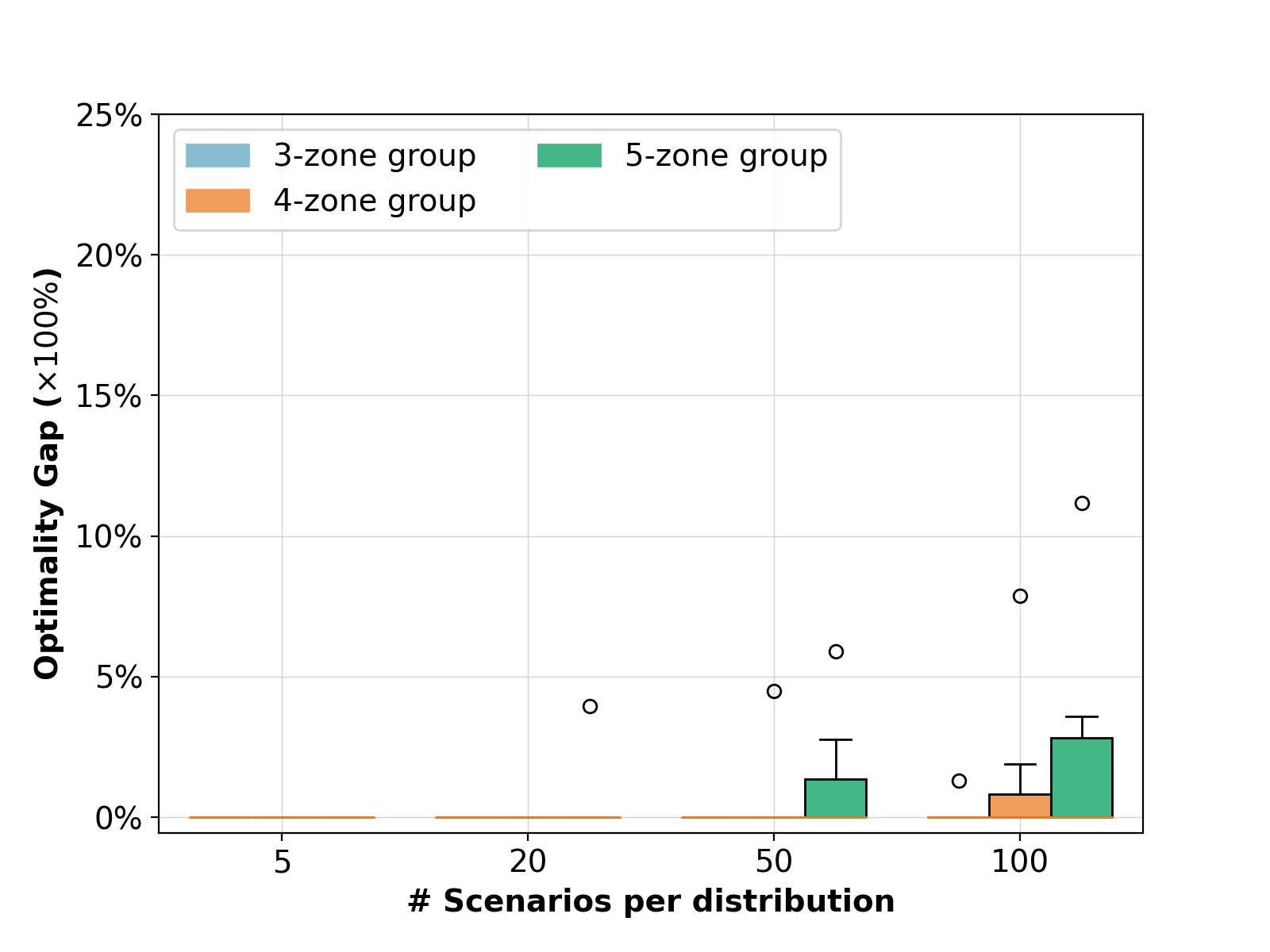}
        \caption{RPWPA, SMALL (\LS v1).}
        \label{fig:pva_15_initial}
    \end{subfigure}
    \hfill
    \begin{subfigure}[h]{0.325\textwidth}
        \centering
        \includegraphics[width=\textwidth]{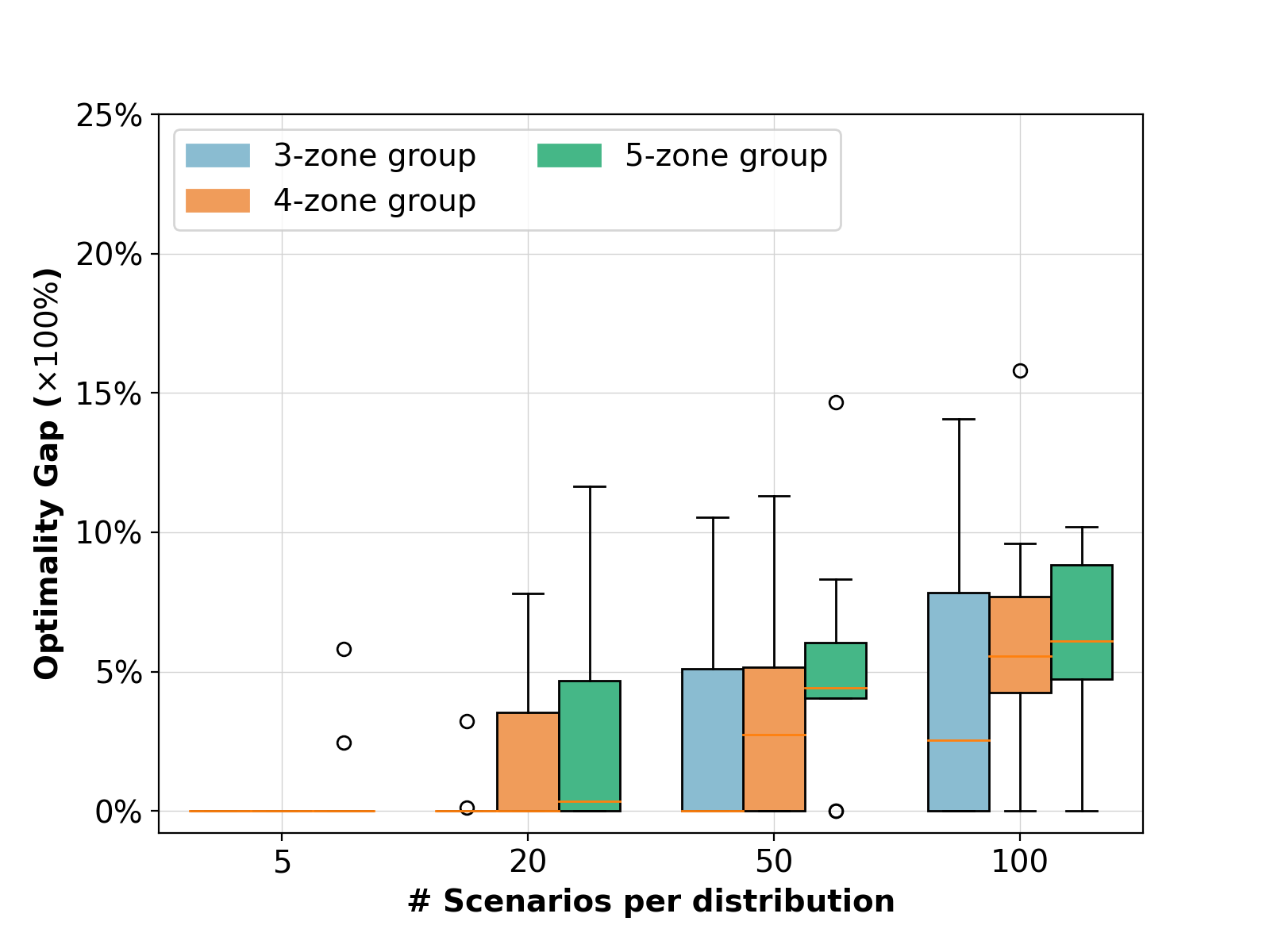}
        \caption{RPWPA, MEDIUM (\LS v1).}
        \label{fig:pva_24_initial}
    \end{subfigure}
    \hfill
    \begin{subfigure}[h]{0.325\textwidth}
        \centering
        \includegraphics[width=\textwidth]{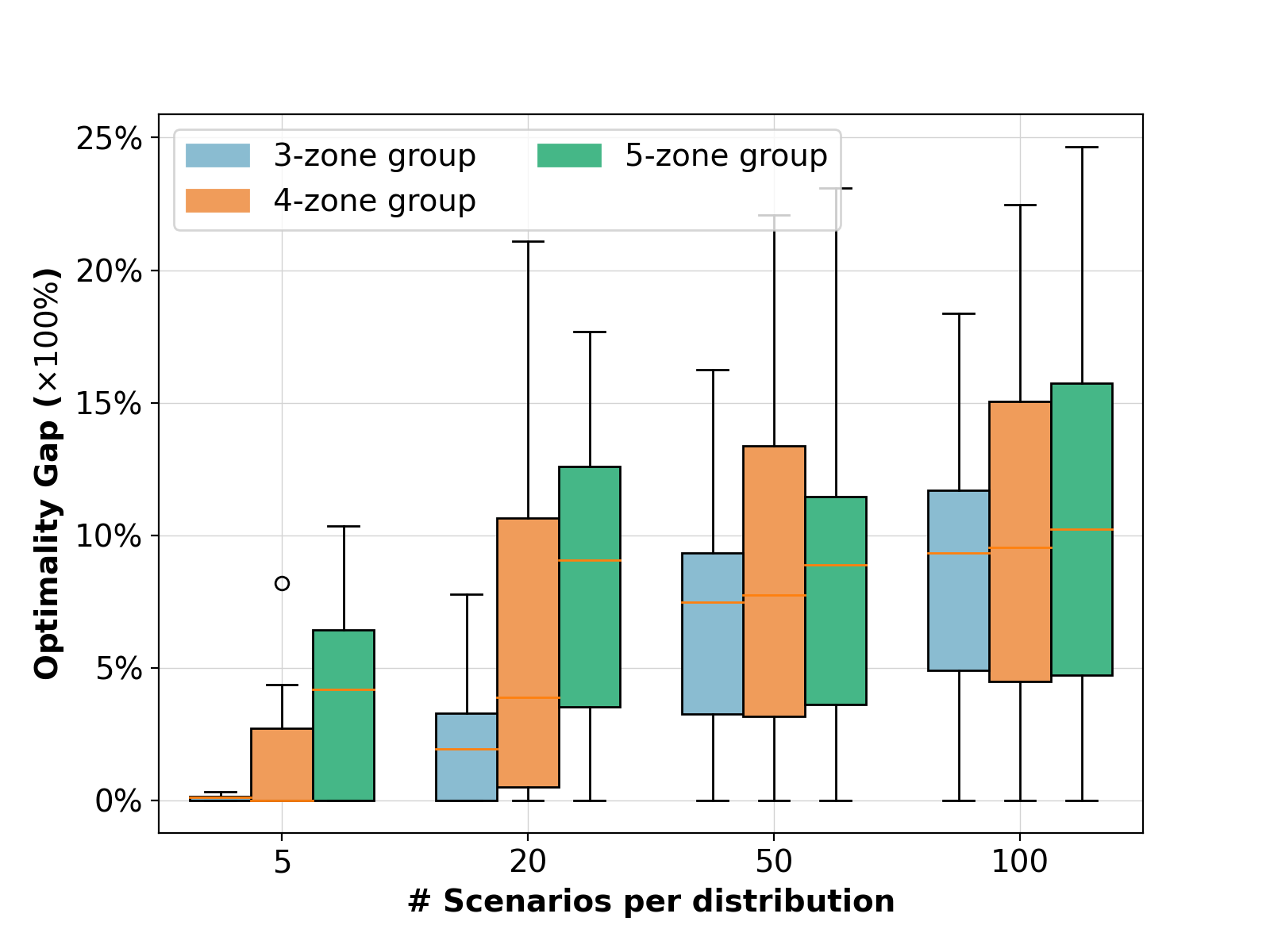}
        \caption{RPWPA, LARGE (\LS v1).}
        \label{fig:pva_35_initial}
    \end{subfigure}
    
    \begin{subfigure}[h]{0.325\textwidth}
        \centering
        \includegraphics[width=\textwidth]{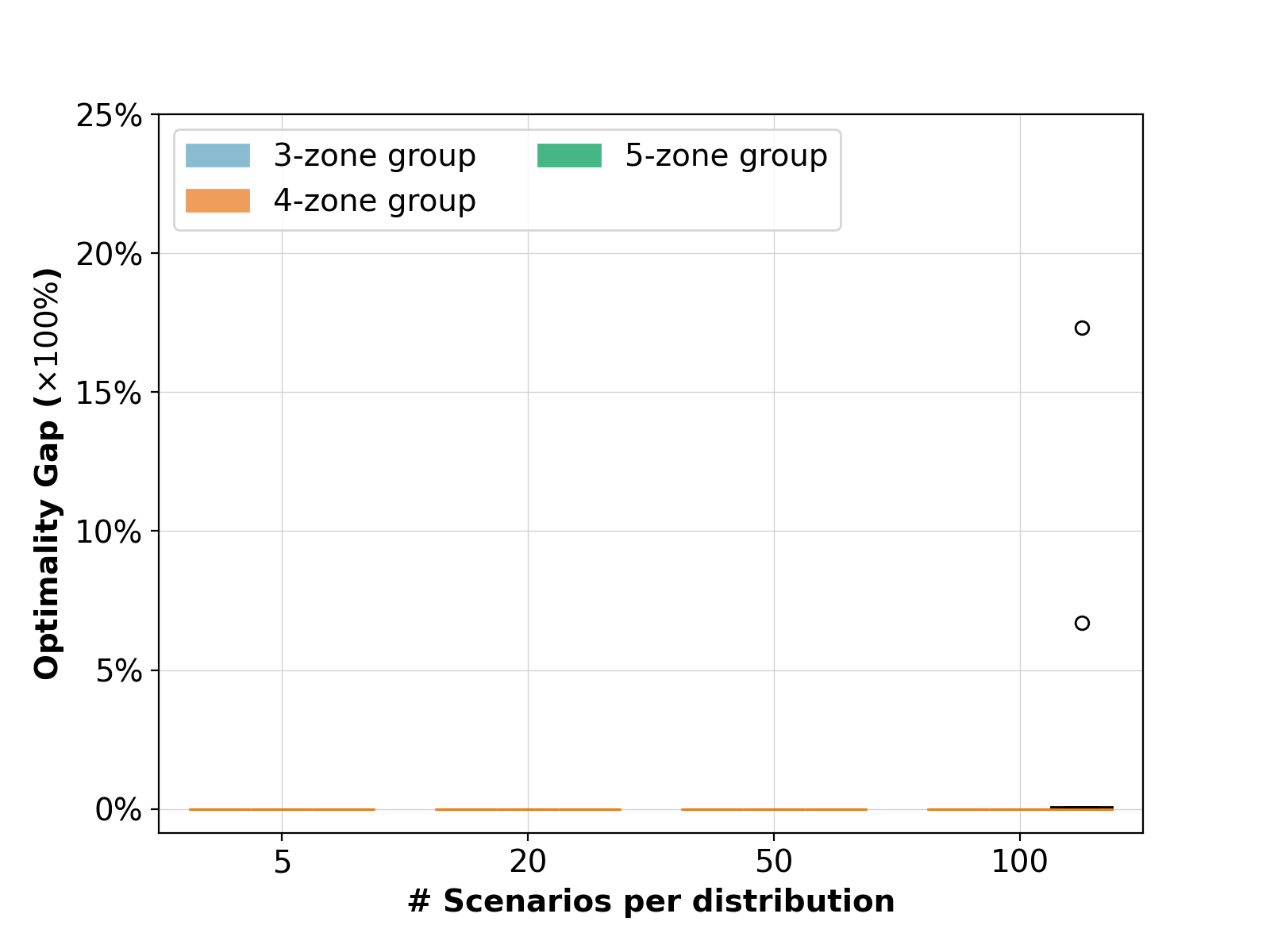}
        \caption{RPWPA, SMALL (\LS v2).}
        \label{fig:pva_15}
    \end{subfigure}
    \hfill
    \begin{subfigure}[h]{0.325\textwidth}
        \centering
        \includegraphics[width=\textwidth]{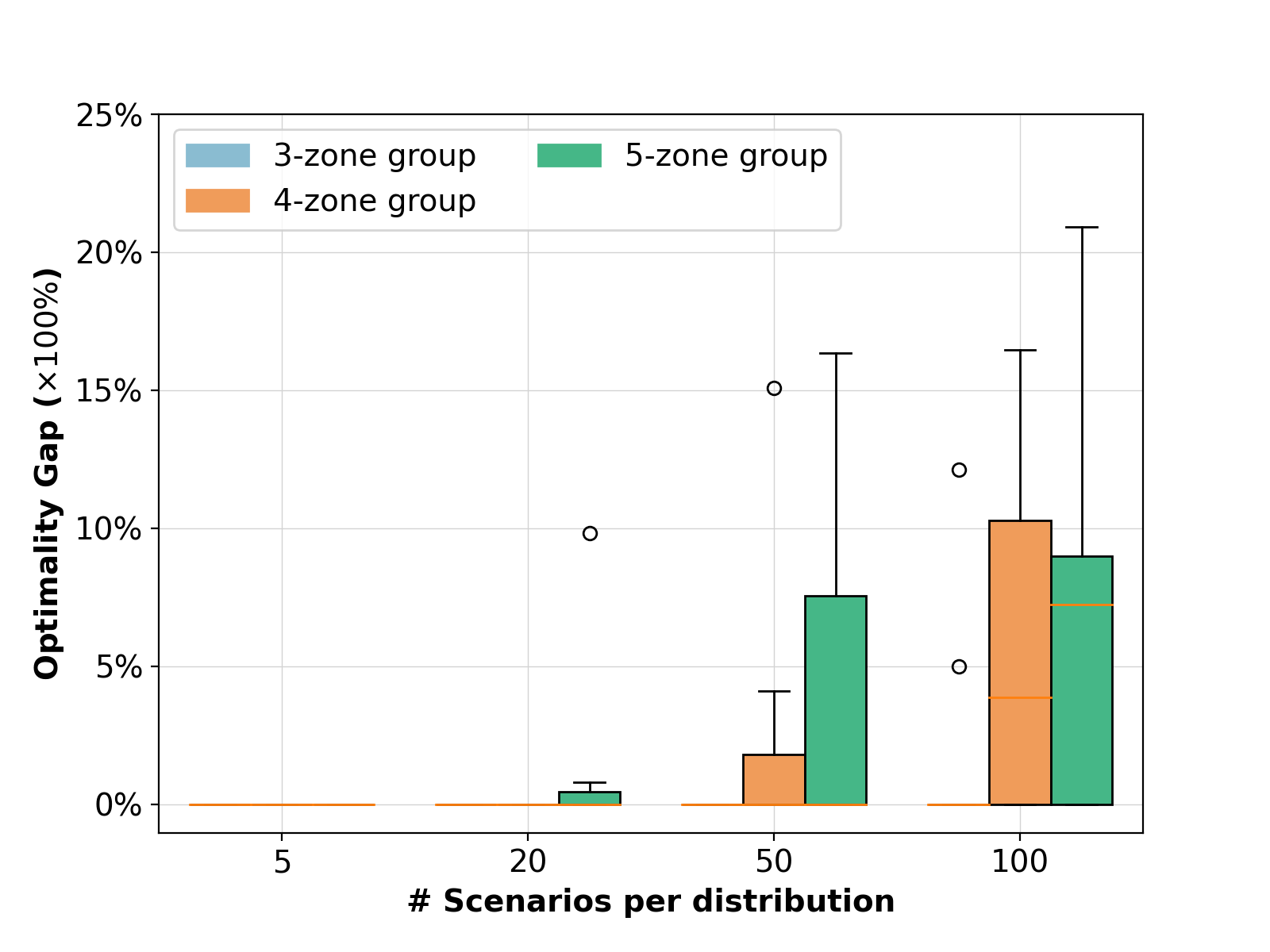}
        \caption{RPWPA, MEDIUM (\LS v2).}
        \label{fig:pva_24}
    \end{subfigure}
    \hfill
    \begin{subfigure}[h]{0.325\textwidth}
        \centering
        \includegraphics[width=\textwidth]{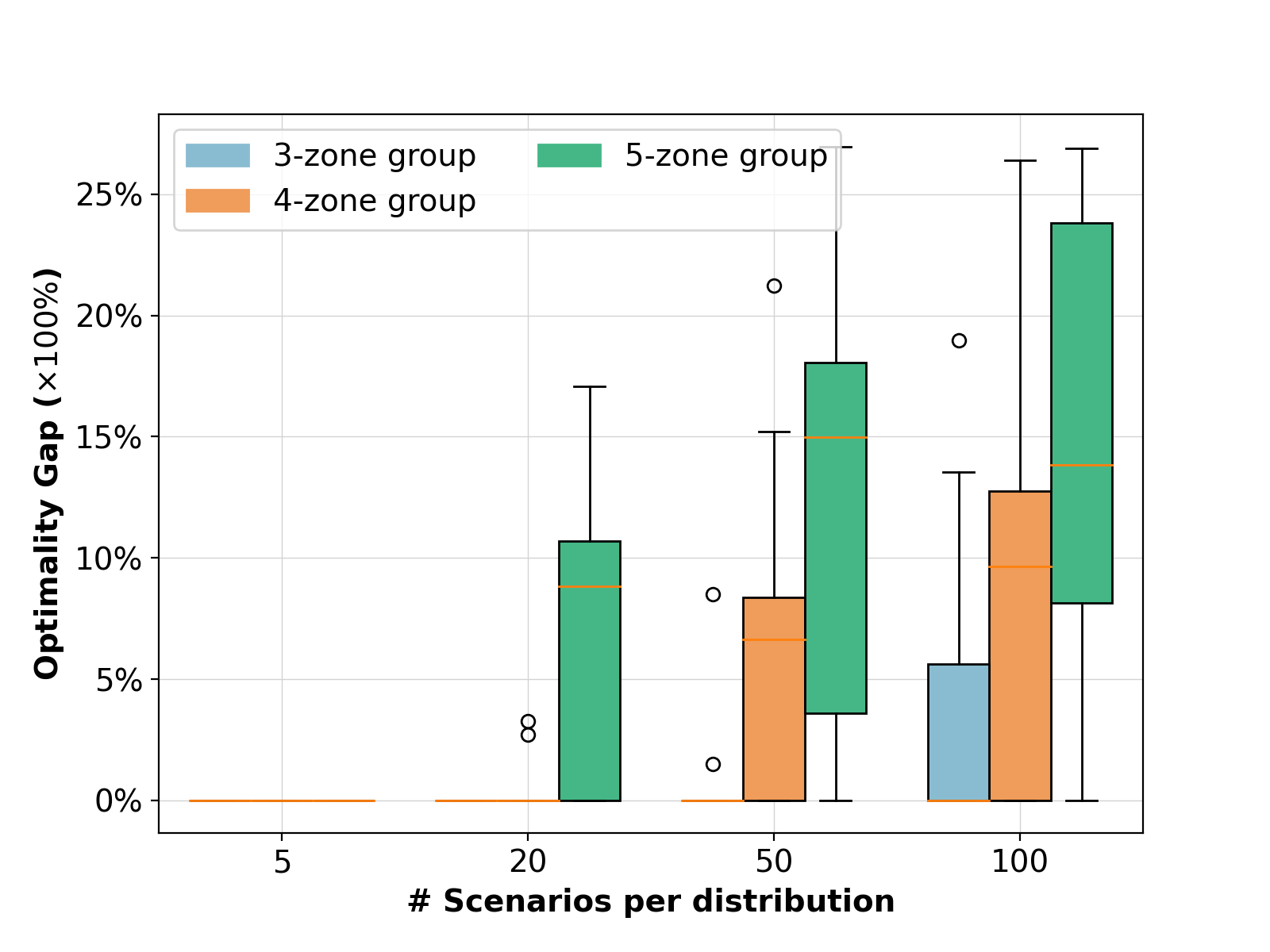}
        \caption{RPWPA, LARGE (\LS v2).}
        \label{fig:pva_35}
    \end{subfigure}
    \caption{Distributions of optimality gap for the instances solved by the \LSt method for the problems with RPWoPA and RPWPA, categorized by number of zones and number of scenarios for SMALL, MEDIUM and LARGE instances. The orange line in each boxplot represents the median.}
    \label{fig:optimality_gaps}
\end{figure}

\end{APPENDICES}

\end{document}